\documentclass[a4paper,11pt]{article}

\usepackage[english]{babel}
\usepackage[T1]{fontenc}
\usepackage[left, mathlines]{lineno}

\usepackage{hyperref}
\usepackage{placeins}
\usepackage{float}
\restylefloat{table}
\restylefloat{figure}

\usepackage[a4paper,top=2cm,bottom=2cm,left=2cm,right=2cm,marginparwidth=1.75cm]{geometry}

\usepackage{lmodern}
\usepackage[skip = 5pt, textfont=it]{caption}
\usepackage[aboveskip = -2pt]{subcaption}
\usepackage{mathtools, amsthm, amssymb}
\usepackage[dvipsnames]{xcolor}
\usepackage{stackengine, graphicx}
\usepackage{array}
\usepackage{calc}
\usepackage{dsfont}
\usepackage[nottoc, notlof, notlot]{tocbibind}
\usepackage{stmaryrd}
\usepackage{comment}
\usepackage{amsmath}
\usepackage{amsthm}
\usepackage{bbm}

\usepackage{enumitem}
\usepackage[normalem]{ulem} 

\allowdisplaybreaks[4] 
\binoppenalty=\maxdimen
\relpenalty=\maxdimen

\theoremstyle{plain}
\newtheorem{thm}{Theorem}
\newtheorem{lem}[thm]{Lemma}
\newtheorem{prop}[thm]{Proposition}

\newtheorem{cor}[thm]{Corollary}
\newtheorem{definition}[thm]{Definition}

\theoremstyle{remark}

\numberwithin{thm}{section}
\numberwithin{equation}{section}

\newcommand{\esp}{\mathbb{E}}
\newcommand{\proba}{\mathbb{P}}
\newcommand{\nmath}{\mathbb{N}}

\newcommand{\rmath}{\mathbb{R}}

\newcommand{\acal}{\mathcal{A}}
\newcommand{\fcal}{\mathcal{F}}
\newcommand{\hcal}{\mathcal{H}}
\newcommand{\lcal}{\mathcal{L}}
\newcommand{\pcal}{\mathcal{P}}
\newcommand{\bcal}{\mathcal{B}}
\newcommand{\ecal}{\mathcal{E}}

\newcommand{\mcal}{\mathcal{M}}

\newcommand{\gcal}{\mathcal{G}}
\newcommand{\ccal}{\mathcal{C}}
\newcommand{\tcal}{\mathcal{T}}

\newcommand{\ucal}{\mathcal{U}}

\newcommand{\rtwo}{\mathbb{R}^{2}}

\newcommand{\un}[1]{\mathds{1}_{\{#1\}}}
\newcommand{\hit}[2]{\tau^{#1}(#2)}
\newcommand{\bulk}[2]{\sigma^{#1}(#2)}

\newcommand{\ancx}{(x)}

\newcommand{\Vol}{\text{Vol}}

\renewcommand{\P}{\mathbb{P}}
\newcommand{\E}{\mathbb{E}}
\newcommand{\eps}{\varepsilon}

\newcommand{\ind}{\mathbbm{1}}
\newcommand{\bzr}{\bcal(z,r)}

\title{Asymptotics for the growth of the infinite-parent Spatial Lambda-Fleming-Viot model}
\author{Apolline Louvet\thanks{Department of Mathematical Sciences, University of Bath, Bath BA2 7AY, UK. Email: \texttt{apolline.louvet@polytechnique.edu}}, Matthew I.~Roberts\thanks{Department of Mathematical Sciences, University of Bath, Bath BA2 7AY, UK. Email: \texttt{mattiroberts@gmail.com}}}

\begin{document}

\maketitle

\begin{abstract}
The $\infty$-parent spatial Lambda-Fleming-Viot (SLFV) process is a model of random growth, in which a set evolves by the addition of balls according to points of an underlying Poisson point process, and which was recently introduced to study genetic diversity in spatially expanding populations. In this article, we give asymptotics for the location and depth of the moving interface, and identify the exact asymptotic scale of the transverse fluctuations of geodesics. Our proofs are based on a new representation of the $\infty$-parent SLFV in terms of chains of reproduction events, and on the study of the properties of a typical geodesic. Moreover, we show that our representation coincides with the alternative definitions of the process considered in the literature, subject to a simple condition on the initial state. Our results represent a novel development in the study of stochastic growth models, and also have consequences for the study of genetic diversity in expanding populations.
\end{abstract}

\tableofcontents

\section{Introduction}
The \emph{spatial Lambda-Fleming-Viot (SLFV)} process was introduced in \cite{barton2010new,etheridge2008drift} as a new framework in which to model populations and their genealogies in a spatial continuum. This framework incorporates many variants---see~\cite{etheridge2008drift} for an introduction---and has been the subject of a substantial body of subsequent research: see e.g.~\cite{berestycki2013large, biswas:SLFV_fluctuating_selection, etheridge:BBM_mean_curvature,  forien2022stochastic, klimek:SLFV_random_environment} for just a subset of the wide variety of approaches and applications. One of the purposes of introducing this new collection of models was to avoid ``clumping and extinction'' effects seen in more traditional models~\cite{felsenstein1975pain}, by controlling local reproduction rates through the use of a Poisson point process of reproduction events. However, one of the assumptions behind this key principle of the standard SLFV framework is that the population is already in equilibrium, with an infinite population density across the whole space. Thus the otherwise flexible and adaptable SLFV models cannot be used to study, for example, the genetic diversity of growing bacterial colonies, e.g.~seen in \cite{hallatschek2007genetic}.

The $\infty$-parent SLFV process was introduced in \cite{louvet2023stochastic} as a way of adapting the SLFV framework to expanding populations. It also serves as a continuum analogue of models of growth or aggregation, such as the Eden model, first-passage percolation, or diffusion-limited aggregation, where progress is often extremely difficult and pronounced lattice effects can be seen even in scaling limits \cite{grebenkov2017anisotropy}.

The $\infty$-parent SLFV process can be non-rigorously described as follows. We begin with some measurable set $E\subset\rtwo$ that represents the initial occupied area. Reproduction events occur as a space-time Poisson point process of unit intensity, and each reproduction event attempts to populate a ball of positive radius around its spatial position at the specified time. However, the event is only successful if the ball intersects the current occupied area at that time, and if so, the parent of that reproduction event is sampled uniformly at random from the intersection of the occupied area and the reproduction event. Thus, if the initial occupied area is a compact set, then there is a first event whose ball intersects it, and causes the occupied area to grow; then there is a second event whose ball intersects the new occupied area, and so on. When the initial occupied area is not compact, the rigorous definition is somewhat more nuanced, but this intuitive description is still useful.

The idea behind the construction of the $\infty$-parent SLFV is to model a population expansion as the spread of a selectively advantageous mutant type (sometimes described as ``real'' individuals) in an established population (of ``ghost'' individuals). This approach has been used in population genetics in
\cite{hallatschek2008gene} and \cite{durrett2016genealogies}
(see also \cite{fan2021stochastic}), but can be traced back to interacting particle systems such as the contact process. The $\infty$-parent SLFV is then the limiting process obtained when letting the selective advantage of the mutation grow to~$+ \infty$ \cite{louvet2023stochastic}.

Louvet and V\'eber showed in \cite{louvet2023measurevalued} that, roughly speaking, the front of the $\infty$-parent SLFV grows at linear rate, and gave a bound on the speed. In this article we further investigate the $\infty$-parent SLFV process, and show that---in a sense that we will make precise later---geodesics in this model fluctuate spatially in the transverse direction on the scale $x^{1/2}$. This is in dramatic contrast with the $x^{2/3}$-order transverse spatial fluctuations expected for many two-dimensional models of random interface growth, including first-passage percolation models, that fall---or are conjectured to fall---within the Kardar-Parisi-Zhang (KPZ) universality class. See Section \ref{sec:related_work} for more discussion of related models.

We will also show that what we call the \emph{bulk} of the process---which can be thought of as the region of space that is completely covered without holes---grows at the same linear speed as the front. This gives a shape theorem in the spirit of the Cox-Durrett theorem \cite{cox1981some} for first passage percolation, and very similar to a result of Deijfen \cite{deijfen2003asymptotic}, who considered a similar growth model to the $\infty$-parent SLFV process, but with less pronounced activity at the expansion front. We are also able to provide substantially finer bounds on the difference between the front and the bulk, even in the more difficult case of starting from a half-plane. To be more precise, we show that the difference between the first hitting time of a point $(x,y)$ and the first time that the whole interval $\{(x',y) : x'\le x\}$ is covered, is of order at most $x^{1/2}$ with high probability.

\begin{figure}[h]
\centering
\captionsetup{justification=centering}
\includegraphics[width=\textwidth]{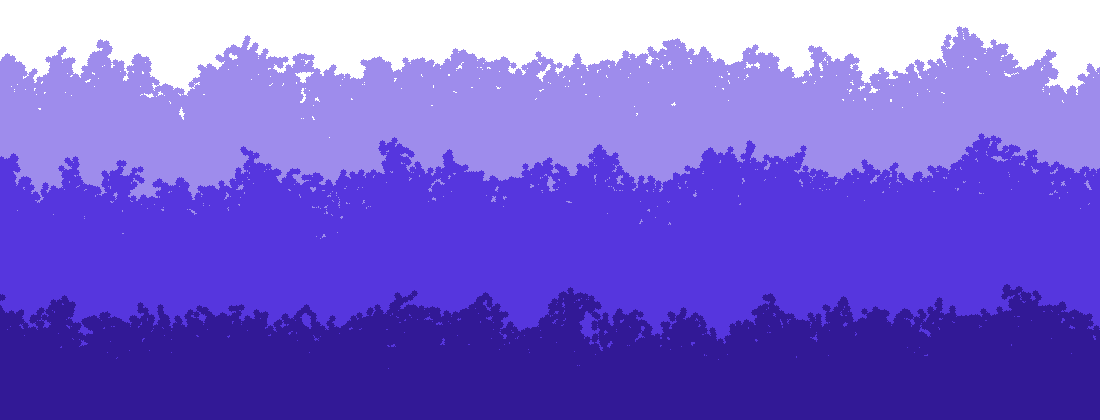}%
\caption{A section of the growing interface of an \protect{$\infty$}-parent SLFV process started from the lower half-plane. The paler regions show the process at later times.}
\label{fig:interface}
\end{figure}

We view both of these advances---identification of the order of transverse geodesic fluctuations, and control of the difference between the front and the bulk---in the context of attempts to provide rigorous mathematical underpinning for results on genetic diversity in expanding populations, as seen e.g.~in~\cite{hallatschek2007genetic,hallatschek2010life} (see also \cite{excoffier2009genetic}). Our results provide progress towards showing that fluctuations of genetic diversity patterns within the bulk of the process are diffusive in scale, and thus can be considered as frozen on the timescale of the expansion. This leads to the spontaneous emergence of a finite number of ``well defined, sector-like regions with fractal boundaries'' \cite{hallatschek2007genetic} from even a well-mixed population of equally fit organisms, observed experimentally as well as in simulations, illustrated in Figure~\ref{fig:stripes}. 
We hope to provide a fully rigorous mathematical justification of this picture in future work.

\begin{figure}[h]
\centering
\captionsetup{justification=centering}
\includegraphics[scale=0.15]{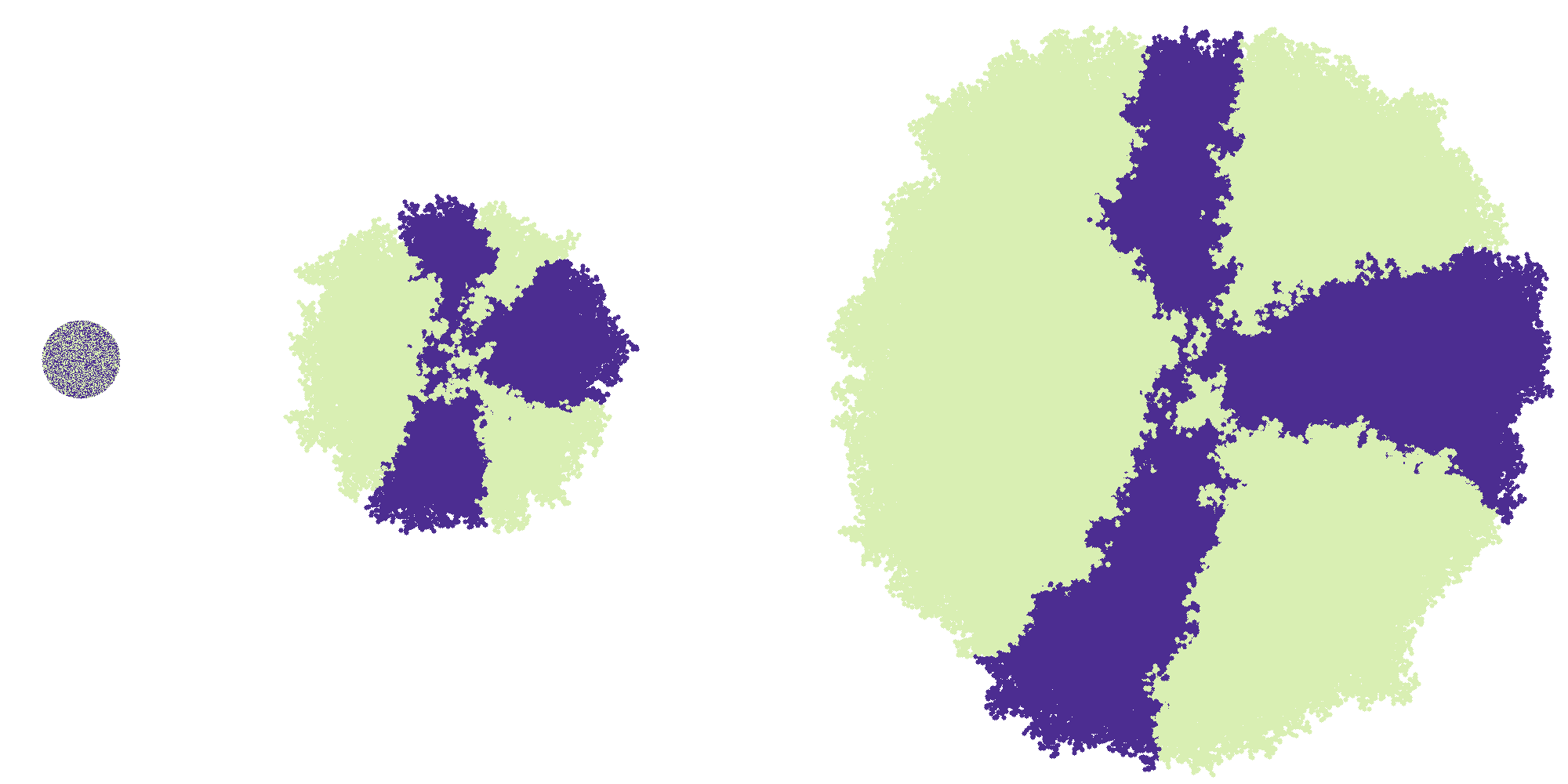}%
\caption{A two-type \protect{$\infty$}-parent SLFV process started from a ball with a 50:50 mix of green (pale) and purple (dark) particles grows a finite number of macroscopic sectors, as seen in experiments with bacteria \cite{hallatschek2007genetic}. The left-hand image shows the process at time $0$, and the two images to the right show the emergence of macroscopic regions of a single type as time increases.}
\label{fig:stripes}
\end{figure}

We mentioned above that the $\infty$-parent SLFV was originally defined in \cite{louvet2023stochastic}. In fact, we will not directly prove our results for the version introduced in \cite{louvet2023stochastic}, which characterised the process as the unique solution to a martingale problem. Instead, we will introduce a set-valued interpretation of the $\infty$-parent SLFV process, which we believe is more intuitive and flexible. In particular, it has the advantage that it can be initialised from sets of Lebesgue measure zero, whereas the measure-valued version of \cite{louvet2023stochastic} cannot. As part of this article we will show that under a simple condition on the initial state, our set-valued process is equal in distribution to the measure-valued version up to a change of state space.

We note that our results are stated in $\mathbb{R}^2$, but should be easily adaptable to $\mathbb{R}^d$ for any $d\ge 2$. We only consider $\mathbb{R}^2$ in order to keep the notational load manageable, and because dimension~$2$ is often the relevant dimension for biological applications~\cite{etheridge2008drift}, with the notable exception of tumour growth. There is also no reason that the reproduction events have to be $\ell^2$ balls; in fact, several of our proofs would be significantly simplified if we used $\ell^\infty$ balls. One could also work with reproduction events taken from a more general class of compact sets (see e.g. \cite{louvet2023measurevalued} for ellipsoid reproduction events). In this article we will use $\ell^2$ balls for concreteness, since these are again usually the most appropriate from a biological point of view.

\subsection{\texorpdfstring{The model: the $\infty$-parent SLFV as a set-valued process}{The model: the infinite-parent SLFV as a set-valued process}}

For $z \in \rtwo$ and $r \geq 0$, let $\bzr$ be the closed $\ell^2$ ball with centre~$z$ and radius~$r$. Let $\mu$ be a finite measure on $(0,+\infty)$ such that
\begin{equation}\label{cond:bounded_radius}
    \text{R}_{0} := \inf\left\{
r > 0: \mu(r,+\infty) = 0
    \right\} < + \infty;
\end{equation}
the measure $\mu$ will represent the distribution of the radius of reproduction events in the system, so that $\text{R}_0$ is the maximum such radius. Let $\Pi$ be a Poisson point process on $\rmath \times \rtwo \times (0,+\infty)$ with intensity
\begin{equation*}
    dt \otimes dz \otimes \mu(dr).
\end{equation*}
The point process $\Pi$ tells us the times, locations and radii of the reproduction events that characterise the growth of our process. We assume throughout the paper that $\Pi$ has no times $t$ at which \emph{two} events occur; that is, we work on the set
\begin{equation}\label{eqn:condition_on_pi}
\left\{\nexists t\in\mathbb{R} : (t,z_1,r_1)\in\Pi \text{ and } (t,z_2,r_2)\in\Pi \text{ for some } (z_1,r_1)\neq (z_2,r_2)\in\rtwo\times(0,+\infty)\right\}.
\end{equation}
The complementary event has zero probability, so we do not lose any generality by ruling it out.

We consider the stochastic growth model defined informally as follows. Fix a measurable set~$E \subseteq \rtwo$, which will represent the initially occupied area of our growth model. The idea is then that starting from $S^E_0=E$, for all~$(t,z,r) \in \Pi$, if the ball~$\bcal(z,r)$ intersects the occupied area $S^E_t$, we add it to the area initially occupied, and call it a \emph{reproduction event}, and we do nothing otherwise. We will give a precise definition below that works with all measurable initial conditions $E$, but note that if $E$ is bounded, then this intuitive description immediately yields a well-defined Markov process which jumps at a finite rate. In particular, this is the case if the initial area $E$ is the singleton~$\{(0,0)\}$.

On the other hand, if the area initially occupied is unbounded, then it is affected by an infinite number of reproduction events over any time interval, and some work is required in order to make our description of the process rigorous. To do so, one possibility is to introduce the following process, which identifies all the possible ancestors at time~$0$ of the individuals living in location $z \in \rtwo$ at time~$t \geq 0$.

\begin{definition}[$\infty$-parent ancestral skeleton]\label{defn:infty_parent_ancestral_skeleton}
Let $t \geq 0$ and $z \in \rtwo$. The $\infty$-parent ancestral skeleton $(B_{s}^{(t,z)})_{s\ge 0}$ is the Markov process defined as follows. 
\begin{enumerate}
    \item First we check whether there exist $z_0\in\rtwo$ and $r_0>0$ such that $(t,z_0,r_0)\in\Pi$ and $z\in\bcal(z_0,r_0)$. If so, we set $B_{0}^{(t,z)} = \bcal(z_0,r_0)$; otherwise we set $B_{0}^{(t,z)} = \{z\}$. In either case, set $t_0=t$.
    \item Then we wait until the last reproduction event $(t_{1},z_{1},r_{1}) \in \Pi$ to occur before time~$t_0$ and to intersect~$B_{0}^{(t,z)}$. That is, $(t_{1},z_{1},r_{1})$ is such that
    \begin{equation*}
        t_{1} = \max\left\{
t' < t_0 : \exists (z',r ') \in \rtwo \times (0,+\infty) \text{ with } (t',z',r ') \in \Pi \text{ and } B_{t-t_0}^{(t,z)} \cap \bcal(z',r ') \neq \emptyset
        \right\}.
    \end{equation*}
Then for all $s \in (t-t_0, t-t_{1})$, we set
\begin{equation*}
    B_{s}^{(t,z)} = B_{t-t_0}^{(t,z)}, 
\end{equation*}
and we set
\begin{equation*}
    B_{t-t_{1}}^{(t,z)} = B_{t-t_0}^{(t,z)} \cup \bcal(z_{1},r_{1}). 
\end{equation*}
    \item We then proceed recursively. Given the process $(B^{(t,z)}_s)_{s\le t-t_j}$ up to some time $t-t_j$, we look for the last reproduction event $(t_{j+1},z_{j+1},r_{j+1})$ before $t_j$ to intersect $B^{(t,z)}_{t-t_j}$, i.e.~such that
    \begin{equation*}
        t_{j+1} = \max\left\{
t' < t_j : \exists (z',r ') \in \rtwo \times (0,+\infty) \text{ with } (t',z',r ') \in \Pi \text{ and } B_{t-t_j}^{(t,z)} \cap \bcal(z',r ') \neq \emptyset
        \right\}.
    \end{equation*}
    Then for all $s \in (t-t_j, t-t_{j+1})$, we set
\begin{equation*}
    B_{s}^{(t,z)} = B_{t-t_j}^{(t,z)}, 
\end{equation*}
and we set
\begin{equation*}
    B_{t-t_{j+1}}^{(t,z)} = B_{t-t_j}^{(t,z)} \cup \bcal(z_{j+1},r_{j+1}). 
\end{equation*}
\end{enumerate}
\end{definition}

This process is well-defined, and its jump rate is bounded from above by that of a Yule process in which each particle splits into two at rate $\text{M}_0$, where
\begin{equation}\label{cond:finite_rate}
\text{M}_0 := \int_0^{\text{R}_0}\pi(\text{R}_0+r)^2 dr.
\end{equation}
The $\infty$-parent SLFV is then defined as follows. 

\begin{definition}[$\infty$-parent SLFV process]\label{defn:infty_parent_slfv}
Let $E$ be a measurable subset of $\rtwo$. The $\infty$-parent SLFV with initial condition~$E$ and intensity~$\mu$ is the process $(S^E_{t})_{t \geq 0}$ such that for all~$t \geq 0$ and~$z \in \rtwo$, 
\begin{equation*}
    z \in S^E_{t} \Longleftrightarrow B_{t}^{(t,z)} \cap E \neq \emptyset. 
\end{equation*}
\end{definition}
As a consequence of the structure of the underlying Poisson point process, $(S^E_{t})_{t \geq 0}$ is a Markov process and $S^E_{t}$ is measurable for all~$t \geq 0$. When the initial state $S_0 = E$ is clear, we will sometimes suppress it in the notation, writing $(S_t)_{t\ge 0}$. 

This definition is not the only possible way to construct the $\infty$-parent SLFV rigorously. In particular, two other approaches to constructing this process were introduced in \cite{louvet2023stochastic,louvet2023measurevalued}: as the unique solution to a martingale problem \cite[Theorem 2.14]{louvet2023stochastic}, or as the limit of other SLFV processes \cite[Theorems 2.9 and 2.10]{louvet2023stochastic}. 
Under some condition on~$\mu$, which is satisfied when~$\mu$ has bounded support, these two constructions are equivalent \cite[Theorem 2.14]{louvet2023stochastic}.
In Section~\ref{sec:equivalence_definitions}, we show that under some condition on the initial condition, these constructions are equivalent to the one considered in this article. This equivalence is interesting in its own right, and also allows us to transfer distributional results between the different versions. In recognition of the different definitions, sometimes we will call the version from Definition \ref{defn:infty_parent_slfv} the \emph{set-valued $\infty$-parent SLFV process}.

\subsection{Main results}

\subsubsection{Transverse fluctuations of geodesics}

For all $x > 0$, let $\mathcal{H}^{x}$ be the half-plane of points to the right of $x$, i.e.
\begin{equation*}
    \mathcal{H}^{x} := \left\{
(x',y') \in \rtwo : x' \geq x
    \right\},
\end{equation*}
and let $\lcal^{(r)}$ be the strip of radius $r$ around the $x$-axis, i.e.
\[\lcal^{(r)} = \left\{(x,y)\in\rtwo : |y|\le r\right\}.\]
Our first main result---which we state now only in an informal way since the precise statement relies on the development of substantial notation---describes the size of the fluctuations in the $y$-direction of the geodesics from~$\{(0,0)\}$ to the half-plane~$\mathcal{H}^{x}$. 

\begin{thm}\label{thm:geodesics}
Consider the $\infty$-parent SLFV started from $E=\{(0,0)\}$ and run until the first hitting time of $\hcal^x$. Say that any path of reproduction events leading from the origin to $\hcal^x$ at this time is a \emph{geodesic} from $0$ to $\hcal^x$. Then:\\
(i) For all $\eps>0$, there exists $A_\eps\in(0,\infty)$ such that for all $x>1$,
\[\P\left(\text{there exists a geodesic from $0$ to $\hcal^x$ that remains within $\lcal^{(A_\eps\sqrt{x})}$}\right)\ge 1-\eps;\]
(ii) For all $\delta>0$ and $\beta\in(1/2,1]$, there exists $c>0$ such that for all sufficiently large $x$,
\[\P\left(\text{there exists a geodesic from $0$ to $\hcal^x$ that remains within $\lcal^{(\delta x^\beta)}$}\right)\ge 1-\exp(-cx^{2\beta-1});\]
(iii) For all $\eps>0$, there exists $a_\eps>0$ such that for all $x$ sufficiently large,
\[\P\left(\text{all geodesics from $0$ to $\hcal^x$ finish outside $\lcal^{(a_\eps\sqrt{x})}$}\right)\ge 1-\eps.\]
\end{thm}

\noindent
In simple terms, parts (i) and (iii) say that with high probability, the spatial fluctuations in the $y$-direction of at least one geodesic from $0$ to $\hcal^x$ are of order $\sqrt{x}$. Part (ii) gives tail bounds on fluctuations of larger order. We will state this result in a more precise and detailed way in Section~\ref{sec:geodesics}. We expect that in fact \emph{all} geodesics from $0$ to $\hcal^x$ have fluctuations of order $\sqrt{x}$, but we do not currently have a bound on how many such geodesics there are (in general there will be more than one). As noted in the introduction, Theorem \ref{thm:geodesics} is in marked contrast to the behaviour of first-passage percolation models, in which geodesics are generally expected to fluctuate in the transverse spatial direction to the order of $x^{2/3}$.

We also show the equivalent of part (i) of Theorem \ref{thm:geodesics} for geodesics from the left half-plane to a point $(x,0)$, showing that with high probability at least one such geodesic stays within a strip of radius $A\sqrt{x}$ for sufficiently large $A$. This task, of controlling geodesics from a plane to a point, is significantly more difficult simply because we have an unbounded initial condition, which restricts our use of self-duality and makes proofs more intricate. We expect that equivalents of parts (ii) and (iii) of Theorem \ref{thm:geodesics} also hold for geodesics from the half-plane to a point, but due to space constraints, we save these for future work.

Define $\hcal$ to be the closed half-plane of points to the left of the origin, $\hcal := \{
(x,y) \in \rtwo : x\le 0 \}$. We again state an informal version of a theorem, this time to be made precise in Section \ref{sec:geodesics_plane_to_point}.

\begin{thm}\label{thm:geodesic_plane_to_point}
Consider the $\infty$-parent SLFV started from $E=\hcal$ and run until the first hitting time of the point $(x,0)$ for $x>0$. Say that any path of reproduction events leading from the $\hcal$ to $(x,0)$ at this time is a \emph{geodesic} from $\hcal$ to $(x,0)$. Then for all $\eps>0$, there exist $A_\eps,x_\eps\in(0,\infty)$ such that for all $x>x_\eps$,
\[\P\left(\text{there exists a geodesic from $\hcal$ to $(x,0)$ that remains within $\lcal^{(A_\eps\sqrt{x})}$}\right)\ge 1-\eps.\]
\end{thm}

\subsubsection{The shape of the expansion}\label{sec:shape_expansion}

Consider the $\infty$-parent SLFV started from some measurable set $E\subset\rtwo$, and for all $A\subset\rtwo$, let $\hit{E}{A}$ be the first time at which the set~$A$ is reached: 
\begin{equation*}
    \hit{E}{A} := \inf\left\{ t \geq 0 : A\cap S^E_{t}\neq\emptyset \right\}. 
\end{equation*}
In a slight abuse of notation, for $z=(x,y)\in\rtwo$ we will write
\[\hit{E}{z} = \hit{E}{x,y} := \hit{E}{\{z\}} = \inf\left\{ t \geq 0 : z\in S^E_{t} \right\}.\]
Again, if the initial state $E$ of the process is clear, we will sometimes omit this from the notation, simply writing $\hit{}{A}$ and $\hit{}{z}$ for the first hitting time of $A$ and $z$ respectively.

Recall that $\hcal$ is the closed half-plane of points to the left of the origin, $\hcal := \{
(x,y) \in \rtwo : x\le 0 \}$, and that $\hcal^x$ is the closed half-plane of points to the right of $(x,0)$. We will prove the following result on hitting times.
\begin{thm}\label{thm:hitting_times}
There exists a constant $\nu\in(0,\infty)$ such that
\begin{itemize}
\item $\displaystyle{\frac{\hit{\{0\}}{\hcal^x}}{x}} \to \nu$ as $x\to +\infty$, almost surely and in $L^1$;
\item for any $z\in\rtwo$ such that $\|z\|=1$, $\displaystyle{\frac{\hit{\{0\}}{uz}}{u}} \to \nu$ as $u\to +\infty$, almost surely and in $L^1$;
\item for any $y\in\mathbb{R}$, $\displaystyle{\frac{\hit{\hcal}{x,y}}{x}} \to \nu$ as $x\to +\infty$, in probability and in $L^1$.
\end{itemize}
\end{thm}

We will prove this result in Section \ref{sec:almost_sure_conv_hitting_times}. The last point was proved in \cite[Theorem 1.5]{louvet2023measurevalued} for another version of the $\infty$-parent SLFV; we shall see in Theorem \ref{thm:informal_equiv_defn} that this version is equal in distribution to our (set-valued) version, so the last point above would also follow from that. However we will prove it directly, as an easy corollary of the first point above together with the self-duality of the $\infty$-parent SLFV.

With the second part of Theorem \ref{thm:hitting_times} in hand, it is natural to ask whether there is some uniformity over hitting times in all directions simultaneously. This is indeed the case.

\begin{thm}\label{thm:ball_shape}
Suppose that $E\subset\rtwo$ is measurable and compact. For any $\eps>0$,
\[\P\left( \bcal\big(0,(\nu^{-1}-\eps)t\big)\subseteq S^E_t \subseteq \bcal\big(0,(\nu^{-1}+\eps)t\big) \,\,\text{ for all large } t\right) =1.\]
\end{thm}

We will prove this theorem in Section \ref{subsec:shape_thm}. It is very similar to a result of Deijfen \cite{deijfen2003asymptotic}. Her model begins with a compact set, and reproduction events occur just as for the $\infty$-parent SLFV, except that only those whose \emph{centre} falls within the previously occupied region are successful. This leads to less reproduction at the front, and less pronounced fluctuations in the front position. Our theorem is also very much in the spirit of the Cox-Durrett shape theorem \cite{cox1981some} for first passage percolation, and our proof at least partially follows the accessible treatment of Auffinger, Damron and Hanson \cite{auffinger201750}.

A result like Theorem \ref{thm:ball_shape}  is \emph{not} true when we begin with an unbounded initial occupied set, e.g.~the left half-plane $\hcal$. In this case we cannot bound the speed of the front at all locations simultaneously; if we look far enough from $0$, we can always find a ``fast area'' of the Poisson point process $\Pi$ that contains many reproduction events in a narrow strip over a short time interval, which allow the process to advance arbitrarily quickly in that strip.

Instead, when starting from $\hcal$, rather than attempting to control the shape in all regions simultaneously, we concentrate on the $x$-axis and give a finer bound on the difference between the hitting time of $(x,0)$ and the time at which the whole interval $\{(x',0): x'\le x\}$ has been covered. In other words, we describe the advance of the area in which the density is maximal at the \textit{mesoscopic scale}, which we will refer to as the \textit{bulk of the expansion}, in line with the terminology used for PDEs, see e.g. \cite{birzu2018fluctuations}.

In order to state the result, we need some more notation. Starting our $\infty$-parent SLFV from a measurable $E\subseteq\hcal$ that contains the origin $(0,0)$ (the reader should think primarily of the case $E=\hcal$, but the definition also works for $E=\{(0,0)\}$), for any $z = (x,y)$ with $x>0$, let $\bulk{E}{z}$ be the first time at which the straight line from $0$ to $z$, i.e.~$\{uz:u\in[0,1]\}$, lies entirely in the occupied area. That is,
\[\bulk{E}{z} := \inf\left\{ t \geq 0 : uz \in S^E_{t}\,\,\,\forall u\in[0, 1] \right\}.\]
We refer to $\bulk{E}{z}$ as the \emph{bulk coverage time} of $z$. Again if $z=(x,y)$ then we will sometimes write $\bulk{E}{x,y}$ instead of $\bulk{E}{(x,y)}$, and if the initial state $E$ is clear we will sometimes omit this from the notation.

\begin{thm}\label{thm:new_speed_growth_bulk}
For any $\eps>0$, there exist $\beta_\eps,x_\eps\in(0,\infty)$ such that for any $\beta>\beta_\eps$ and $x>x_\eps$,
\[\P\left(\bulk{\hcal}{x,0} > \hit{\hcal}{x,0} + \beta\sqrt{x}\right) < \eps\]
and
\[\P\left(\bulk{\{0\}}{x,0} > \hit{\{0\}}{\hcal^x} + \beta\sqrt{x}\right) < \eps.\]
\end{thm}

This result, which we will prove in Section \ref{sec:new_control_diff}, gives significantly finer control over the bulk of the process than Theorem \ref{thm:ball_shape}. It tells us that starting from $\hcal$, the difference between the first hitting time of $z=(x,0)$ and its bulk coverage time is at most of order $\sqrt{x}$ with high probability; and starting from $\{0\}$, the difference between the first hitting time of $\hcal^x$ and the bulk coverage time of $(x,0)$ is also at most of order $\sqrt{x}$ with high probability. The former statement is significantly more delicate than the latter, again because controlling the growth of the $\infty$-parent SLFV process started from an unbounded initial region such as $\hcal$ leads to inherently more delicate proofs than when started from compact initial conditions. However, the extra difficulty is already contained in the proof of Theorem \ref{thm:geodesic_plane_to_point}, and with this in hand, the proofs of the two parts of this theorem will be very similar.

In Section \ref{subsec:slow_coverage}, we also obtain the following upper bound on the upper tail of $\bulk{\{0\}}{z}$, which also acts as an upper bound on the upper tail of $\hit{\{0\}}{z}$. It is developed as a tool for proving Theorems~\ref{thm:ball_shape} and~\ref{thm:new_speed_growth_bulk}, but given its more quantitative form, it may also be useful in its own right.

\begin{prop}\label{prop:tail_tau_x}
There exist $\beta_{\mu},c_{\mu},\delta_\mu > 0$ depending only on~$\mu$ such that for any $\beta \geq \beta_{\mu}$ and $z\in\rtwo$,
\[\proba\left( \sup_{z'\in\bcal(z,\delta_\mu)} \bulk{\{0\}}{z'} > \beta \|z\| \right) \leq \exp\left( -c_{\mu}\beta \|z\| \right),\]
where $\|\cdot\|$ denotes the $\ell^2$ norm in $\rtwo$. In particular,
\[\proba\left( \hit{\{0\}}{z} > \beta \|z\| \right) \leq \exp\left( -c_{\mu}\beta \|z\| \right),\]
\end{prop}

\subsubsection{\texorpdfstring{Equivalence of definitions of $\infty$-parent SLFV processes}{Equivalence of definitions of infinite-parent SLFV processes}}

For any measurable set $E \subseteq \rtwo$, we say that \textit{$E$ satisfies condition}~($\triangle$) if
\begin{equation*}
    \Vol\left( \bcal(z,\eps) \cap \text{E} \right) > 0 \,\,\,\,\text{ for all } z \in E \text{ and } \eps > 0.
\end{equation*}
For instance, any open set of $\rtwo$ satisfies condition~($\triangle$), while a set of null Lebesgue measure does not satisfy it. 
The idea behind this condition is that in the set-valued $\infty$-parent SLFV, reproduction occurs when a reproduction event intersects the occupied area, while in the measure-valued $\infty$-parent SLFV introduced in \cite{louvet2023stochastic}, reproduction only occurs when the intersected occupied area has non-zero Lebesgue measure. Condition~($\triangle$) will ensure that the set of centres of reproduction events for which a discrepancy occurs has null Lebesgue measure, and hence that such reproduction events almost surely never occur.  
We will show that if we start an $\infty$-parent SLFV from a set satisfying ($\triangle$), then it will satisfy ($\triangle$) for all times.

\begin{lem}\label{lem:triangle_valued_process} Let $E \subseteq \rtwo$ be measurable, and let $(S^E_{t})_{t \geq 0}$ be the (set-valued) $\infty$-parent SLFV with initial condition $S^E_{0} = E$. If $E$ satisfies condition~($\triangle$), then $S^E_{t}$ satisfies condition~($\triangle$) for all $t \geq 0$.
\end{lem}

We now state another informal version of a theorem, this time on the equivalence of definitions of the $\infty$-parent SLFV, which we will state more precisely and in a more general form in Section \ref{sec:equivalence_definitions}.

\begin{thm}\label{thm:informal_equiv_defn}
Suppose that $E \subseteq \rtwo$ is measurable and satisfies condition~($\triangle$). Then the (set-valued) $\infty$-parent SLFV from Definition \ref{defn:infty_parent_slfv} started from $E$ is equal in distribution to the (measure-valued) $\infty$-parent SLFV defined in \cite{louvet2023stochastic} started from $E$, up to a change of state space. 
\end{thm}

\subsection{Notation}
Throughout the paper there will implicitly be a dependence of our process on the $\sigma$-finite measure $\mu$ (with essential supremum $\text{R}_{0}<\infty$), and the Poisson point process $\Pi$. For any c\`adl\`ag process $(Z_t)_{t\ge 0}$, and any $t\ge 0$, we write $Z_{t-}$ to mean the left limit $\lim_{s\uparrow t}Z_s$. Throughout, $\|\cdot\|$ will denote the $\ell^2$ norm in $\rtwo$ and for $z\in\rtwo$ and $r\in[0,\infty)$, $\bcal(z,r)$ will denote the \emph{closed} $\ell^2$ ball of centre $z$ and radius $r$.

We now provide a brief summary of the notation introduced above, for the reader's reference.
\begin{itemize}
\item $\text{R}_0$ is the maximal radius of reproduction events,
\[\text{R}_{0} := \inf\left\{
r > 0: \mu(r,+\infty) = 0
    \right\} < + \infty,\]
and
\[\text{M}_0 := \int_0^{\text{R}_0}\pi(\text{R}_0+r)^2 dr\]
is the maximum rate at which a new reproduction event intersects one existing reproduction event.
\item $(B_{s}^{(t,z)})_{0 \leq s \leq t}$ is the $\infty$-parent ancestral skeleton, a Markov process consisting of a sequence of reproduction events leading away from $z\in\rtwo$ backwards in time from $t$ to $s$; see Definition \ref{defn:infty_parent_ancestral_skeleton}.
\item $S_t^E$ is the state of the $\infty$-parent SLFV process at time $t$ when started from $E\subset\rtwo$; it consists of the points whose ancestral skeleton at time $t$ intersects $E$. See Definition \ref{defn:infty_parent_slfv}.
\item $\hcal := \{
(x,y) \in \rtwo : x\le 0 \}$ and $\mathcal{H}^{x} := \left\{ (x',y') \in \rtwo : x' \geq x \right\}$.
\item $\lcal^{(r)} = \left\{(x,y)\in\rtwo : |y|\le r\right\}$.
\item $\hit{E}{A} := \inf\left\{ t \geq 0 : A\cap S^E_{t}\neq\emptyset \right\}$ is the first hitting time of the set $A$ starting from $E$, and in a slight abuse of notation, for $z=(x,y)\in\rtwo$ we write
\[\hit{E}{z} = \hit{E}{x,y} := \hit{E}{\{z\}} = \inf\left\{ t \geq 0 : z\in S^E_{t} \right\}.\]
\item When $(0,0)\in E$,
\[\bulk{E}{z} := \inf\left\{ t \geq 0 : uz \in S^E_{t}\,\,\,\forall u\in[0, 1] \right\};\]
we call this the bulk coverage time of $z$, the first time at which the straight line from $0$ to $z$ lies entirely in the occupied area.
\item Condition ($\triangle$) is satisfied by a set $E\subset\rtwo$ if $\Vol\left( \bcal(z,\eps) \cap \text{E} \right) > 0$ for all $z \in E$ and $\eps > 0$.
\end{itemize}

\subsection{Motivation and related work}\label{sec:related_work}

The development of new inference tools
represents an active field of research at the interface between probability, statistics and population genetics. A series of experiments carried out by Hallatschek and collaborators~\cite{hallatschek2007genetic,hallatschek2010life} demonstrated that for populations in
spatial expansion, randomness in reproduction even amongst well-mixed populations of equally fit bacteria gives rise to ``well defined, sector-like regions with fractal boundaries'', which are not observed in comparable populations that are not undergoing a spatial expansion. These experimental results, seen also in simulations (see Figure~\ref{fig:stripes}), suggest that the applicability of classical inference tools to populations in spatial expansion is limited and could lead to erroneous conclusions of selective advantage. This motivates the study of stochastic population genetics models adapted to populations in expansion in a continuum.

Studying genetic diversity at the front edge in expanding populations requires results regarding the location of the front edge in the underlying stochastic population process. First results describing the evolution of the location of the front edge were obtained in \cite{louvet2023measurevalued}. In this paper, we considerably extend these results, and describe the growth of the \textit{area of maximal density}, in which the dynamics of genetic diversity is as in the absence of a spatial expansion, as well as the evolution of the width of the front area. Our results highlight how sector-like regions can persist, and show that their emergence can be studied by focusing on the reproduction dynamics in a narrow area at the front edge. We therefore see the results of this paper as a substantial step towards proving exactly the behaviour observed
in~\cite{hallatschek2007genetic,hallatschek2010life}. 

There has been a huge interest in recent years in models of random interface growth in stochastic geometry, partially motivated by the scaling relations of Kardar, Parisi and Zhang \cite{kardar1986dynamic}; see also the survey of Quastel \cite{quastel2011introduction}. One of the most prominent discrete models is that of first passage percolation (FPP), where an independent and identically distributed random variable $X_e$ is attached to each edge $e$ of a graph, for example the lattice $\mathbb{Z}^2$, and one studies the random metric given by $D(w,z) = \min_P\sum_{e\in P} X_e$ for vertices $w,z$, where the minimum is taken over all paths $P$ from $w$ to $z$; alternatively one can consider the growing sets $W_t^{\{0\}} := \{z : D(0,z)\le t\}$ or $W_t^\hcal := \{z : D(\hcal,z)\le t\}$, for a clearer analogue to our processes $S^{\{0\}}_t$ or $S^\hcal_t$. The survey of Grimmett and Kesten \cite{grimmett2012percolation} contains references to many important contributions in the study of this model. It is conjectured that for many reasonable distributions for the weights $X_e$, there are two ``universal scaling exponents'', the fluctuation exponent $\chi$ and the wandering exponent $\xi$, that satisfy (what is often referred to as the) KPZ relation $\chi = 2\xi -1$. Roughly speaking, the fluctuation exponent $\chi$ is defined to satisfy
\[D(w,z) - \E[D(w,z)] \approx |w-z|^\chi\]
whereas for the wandering exponent, if we look at any geodesic between $0$ and $(x,0)$ (i.e.~any path which minimises the sum in the definition of $D(0,(x,0))$), then its greatest displacement from the $x$-axis should be of order $x^\xi$. Several possible rigorous definitions have been given for $\chi$ and $\xi$; Chatterjee \cite{chatterjee2013universal} showed that for two such definitions, if they agree, then the KPZ scaling related does indeed hold if the weights are ``nearly gamma'' in the sense of Bena\"im and Rossignol \cite{benaim2008exponential}. It is believed that for FPP in two dimensions, $\xi = 2/3$ and $\chi = 1/3$. The same exponents have been rigorously proven for a model of \emph{last} passage percolation by Johansson \cite{johansson2000transversal}.  Our Theorem \ref{thm:geodesics} establishes that for the $\infty$-parent SLFV, the wandering exponent $\xi$ (for at least one point-to-line geodesic) is equal to $1/2$. We hope to provide a bound on the fluctuation exponent for the $\infty$-parent SLFV in future work.

As mentioned above, our Theorem \ref{thm:ball_shape} is very similar to a result of Deijfen \cite{deijfen2003asymptotic}. She begins with a compact set, and considers the same Poisson point process of reproduction events as in our model, but if $S_t$ is her occupied region at time $t$, she only accepts reproduction events whose centre is in $S_t$, rather than (as in the $\infty$-parent SLFV) any reproduction event that intersects $S_t$. Clearly this leads to less reproduction near the front, and on an intuitive level should lead to less pronounced fluctuations in the front position in Deijfen's model. She proves a shape theorem identical to Theorem \ref{thm:ball_shape} for her model; further results on this model include \cite{deijfen2004coexistence, gouere2008continuous}. Simulations of the $\infty$-parent SLFV process suggest that growth of the process is driven by ``spikes'' in the direction of expansion, which then thicken in the transverse direction \cite{louvet2023measurevalued}; this heuristic suggests that despite their apparent similarity, the expansion and particularly the fluctuations in our model could be markedly different from that of Deijfen.

\subsection{Layout of article}
In Section \ref{sec:backwards_chains} we set up some of the key concepts to appear throughout the paper, specifically \emph{chains of reproduction events} which will be at the centre of the proofs of all our main theorems. We also provide some consequences of the self-duality of the system, and set out a simple strategy for slowly but surely covering a narrow strip of space, which will be surprisingly useful despite its obvious inefficiency. One of the applications of this strategy will be the proof of Proposition \ref{prop:tail_tau_x}.

In Section \ref{sec:geodesics} we first define a geodesic from a point to a set, and use this definition to make the statement of Theorem \ref{thm:geodesics} precise. We then use the setup of chains of reproduction events from Section \ref{sec:backwards_chains} to give bounds on the fluctuations of those geodesics, in order to prove Theorem \ref{thm:geodesics}. The upper bounds will be relatively straightforward, using Doob's submartingale inequality, whereas the lower bound will be somewhat more involved. In Section \ref{sec:geodesics_plane_to_point} we then define geodesics from sets to points, and give a precise statement of Theorem \ref{thm:geodesic_plane_to_point}, which we then prove in the same section, taking advantage of many of the tools from earlier in Section \ref{sec:geodesics}.

Section \ref{sec:hitting_and_difference} contains the proofs of Theorems \ref{thm:hitting_times}, \ref{thm:ball_shape} and \ref{thm:new_speed_growth_bulk}. Section \ref{sec:almost_sure_conv_hitting_times} deals with Theorem \ref{thm:hitting_times}, of which the first point is essentially a repeat of a result from \cite{louvet2023measurevalued} adapted to our notation, and the third point is a simple time-reversal argument. The second point is the main contribution here, and relies on the bounds on transverse fluctuations of geodesics developed in Section \ref{sec:geodesics}. We then move on to Theorem \ref{thm:ball_shape} in Section \ref{subsec:shape_thm}, where our strategy is to use ergodicity arguments similar to the proof of the Cox-Durrett shape theorem \cite{cox1981some} but the details required are substantially different. Finally we prove Theorem \ref{thm:new_speed_growth_bulk} in Section \ref{sec:new_control_diff}, where again the control on geodesics from Section \ref{sec:geodesics} is crucial, along with the slow coverage strategy from Section \ref{sec:backwards_chains}.

Finally, Section \ref{sec:equivalence_definitions} is independent of what has come before, defining a simple condition on sets under which the set-valued definition of the $\infty$-parent SLFV process is equivalent, up to a change of state space, to the measure-valued definition given in \cite{louvet2023stochastic}. We translate our set-valued construction directly into a measure-valued setting, and use martingale problems to uniquely characterise the distribution of the process when started from compact initial conditions, showing that the two definitions must agree in distribution. A limiting argument then allows us to ensure that the equality in distribution remains true for unbounded initial conditions.

\section{Chains of reproduction events}\label{sec:backwards_chains}

One of the attractive features of the $\infty$-parent SLFV process is its self-duality (see \cite[Proposition 2.18]{louvet2023stochastic}). To demonstrate, one intuitively imagines building the process from sequences of reproduction events forwards in time, but the actual definition given in Definition \ref{defn:infty_parent_slfv} involves sequences of reproduction events backwards in time; and such forwards and backwards \emph{chains} of reproduction events have the same distribution given the time-reversibility of the underlying Poisson point process.

In this section, we begin our study by formally introducing the concept of \emph{chains of reproduction events} trailed above, and translate the study of the first hitting time $\hit{E}{z}$ and bulk coverage time $\bulk{E}{z}$ into this language.

\subsection{Definition and first properties of chains of reproduction events}
We will shortly give a definition of chains of reproduction events. However, we note from the discussion above that sometimes we will use chains that are built backwards in time from a given time $t$; to this end, it is useful to consider the time-reversal of the Poisson process $\Pi$. For all $t \geq 0$, let $\Pi^{(t)}$ be the Poisson point process on $\rmath \times \rtwo \times (0,+\infty)$ with intensity $dt \otimes dz \otimes \mu(dr)$ such that for all $(t',z',r') \in \rmath \times \rtwo \times (0,+\infty)$, 
\[(t',z',r') \in \Pi^{(t)} \Longleftrightarrow (t-t',z',r') \in \Pi.\]
In other words, $\Pi^{(t)}$ is constructed by reversing time in $\Pi$, starting at time~$t$. We note here two things:
\begin{itemize}
\item the first component of $\Pi$, representing the times at which reproduction events occur, runs over the whole of $\mathbb{R}$, not just $[0,\infty)$;
\item although there will almost surely be no point at time zero in $\Pi$ (that is, there will not exist $Z_0$ and $R_0$ such that $(0,Z_0,R_0)\in\Pi$) there will certainly exist times $t$ such that there is a point at time zero in $\Pi^{(t)}$; indeed, this will be true whenever there is a point at time $t$ in $\Pi$. This observation will be important below.
\end{itemize}
We now state the definition of chains of reproduction events. We do so for a general Poisson point process $\Xi$ on $\rmath \times \rtwo \times (0,+\infty)$, but in practice we will always use $\Xi=\Pi$ or $\Xi = \Pi^{(t)}$ for some $t$.

\begin{definition}\label{defn:chain_reproduction_events}
Let $z\in\rtwo$ and $\Xi$ a Poisson point process on $\rmath \times \rtwo \times (0,+\infty)$. A \textbf{chain of $\Xi$-reproduction events started from~$z$} is an $\rtwo \times (0,+\infty)$-valued process
\[\mathbf{C} = (\mathbf{C}_{s})_{s \geq 0} = (Z_{s},R_{s})_{s \geq 0}\]
such that
\begin{itemize}
\item either $\mathbf{C}_{0} = (z,0)$ or $\mathbf{C}_{0} = (Z_0,R_0)$ with $(0,Z_0,R_0)\in\Xi$ and $z\in\bcal(Z_0,R_0)$;
\item for all $t > 0$, if $\mathbf{C}_{t} \neq \mathbf{C}_{t-}$, then $(t,Z_{t},R_{t}) \in \Xi$ and 
\[\bcal\left( Z_{t},R_{t} \right) \cap \bcal\left( Z_{t-}, R_{t-} \right) \neq \emptyset.\]
\end{itemize}
\end{definition}
As mentioned above, if $\Xi=\Pi$, then almost surely there is no point at time $0$, so (almost surely) we must have $\mathbf{C}_{0} = (z,0)$. However, the same is not true if we consider $\Xi = \Pi^{(t)}$ for an uncountable number of values for~$t$, hence the slight complication in the first part of the definition.

For all $z\in \rtwo$, let $\ccal^{(z)}(\Xi)$ be the set of all chains of $\Xi$-reproduction events started from~$z$. Moreover, for all $t \geq 0$, let
\[\ccal_{t}^{(z)}(\Xi) := \left\{\mathbf{C}_{t} : \mathbf{C} \in \ccal^{(z)}(\Xi) \right\}\]
be the set of possible states (or ``endpoints'') at time $t$ of chains of $\Xi$-reproduction events started from~$z$.

We now recall from Definition \ref{defn:infty_parent_ancestral_skeleton} the $\infty$-parent ancestral skeleton $(B_s^{(t,z)})_{0\le s \le t}$. We also recall that we work on the almost sure event (\ref{eqn:condition_on_pi}), which ensures that at most one reproduction event occurs at each instant. Our first lemma in this section observes that in fact $B_{s}^{(t,z)}$ can be interpreted as the union of all possible endpoints at time~$s$ for chains of $\Pi^{(t)}$-reproduction events started from $z$.

\begin{lem}\label{lem:equivalence_dual_chain}
For all $t \geq 0$, $z \in \rtwo$ and $0 \leq s \leq t$, 
\[B_{s}^{(t,z)} = \bigcup_{(z',r') \in \ccal_{s}^{(z)}(\Pi^{(t)})} \bcal(z',r').\]
\end{lem}

Notice that as $B^{(t,z)}$ is constructed by going backwards in time in $\Pi$ starting at time~$t$, it can also be seen as constructed going \textit{forwards in time} in $\Pi^{(t)}$, starting at time~$0$. 
The idea behind the proof is simply that $B^{(t,z)}_{s}$ describes the area covered by the ancestral skeleton at time $s$, and for a point $z'$ to be in this area there must have been a sequence of reproduction events whose last event occurred by time $s$ and covered $z'$. 
There is a slight complication that there can be successful reproduction events that do not cover any new space (their ball is already covered by the union of previous successful reproduction events), in which case the ancestral skeleton does not change whereas the set of chains of reproduction events does change.

\begin{proof}
Let $t \geq 0$ and $z \in \rtwo$. At time~$s = 0$, if there is $(t,z_0,r_0)\in\Pi$ with $z\in\bcal(z_0,r_0)$ then
\[\ccal_{0}^{(z)}(\Pi^{(t)}) = \{(z_0,r_0),(z,0)\} \text{ and } B_{0}^{(t,z)} = \bcal(z_0,r_0) = \bcal(z_0,r_0)\cup \bcal(z,0);\]
and otherwise
\[\ccal_{0}^{(z)}(\Pi^{(t)}) = \{(z,0)\} \text{ and } B_{0}^{(t,z)} = \{z\},\]
so the result holds with $s=0$. Now let $T_0=0$ and for all $n\ge 1$, let
\[T_n := \inf\Big\{s>T_{n-1} : \ccal_{s}^{(z)}(\Pi^{(t)}) \neq \ccal_{s-}^{(z)}(\Pi^{(t)})\Big\}.\]
It is easy to see that, almost surely, $T_n>T_{n-1}$ for each $n$, and indeed that $T_n\to\infty$ almost surely. We work by induction on $n$; suppose that the result is true up to time $T_{n-1}$, so in particular
\begin{equation}\label{eq:induction}
B_{T_{n-1}}^{(t,z)} = \bigcup_{(z',r') \in \ccal_{T_{n-1}}^{(z)}(\Pi^{(t)})}\bcal(z',r').
\end{equation}
There can be no point $(t'',z'',r'')\in \Pi^{(t)}$ with $t''\in( T_{n-1},T_{n})$ such that $\bcal(z'',r'')$ intersects 
\begin{equation*} 
\bigcup_{(z',r') \in \ccal_{T_{n-1}}^{(z)}(\Pi^{(t)})}\bcal(z',r'), 
\end{equation*}
otherwise we could use this point to define a new chain of reproduction events, $\ccal_{s}^{(z)}(\Pi^{(t)})$ would change at time $s=t''$, and thus we would have $T_{n}=t''$. Therefore, by \eqref{eq:induction} and the definition of the ancestral process, $B_{s}^{(t,z)}$ also does not change for $s\in(T_{n-1},T_{n})$; and the result is true up to time $T_{n}-$.

Now, since $\ccal_{T_n}^{(z)}(\Pi^{(t)}) \neq \ccal_{T_n-}^{(z)}(\Pi^{(t)})$, there exists some chain $\mathbf{C} \in \ccal^{(z)}(\Pi^{(t)})$ with $\mathbf{C}_{T_n} \neq \mathbf{C}_{T_n-}$. By the definition of backwards chains of reproduction events, this is equivalent to the existence of a point $(T_n,z_1,r_1)$ in $\Pi^{(t)}$ such that $\bcal(z_1,r_1)$ intersects some reproduction event in $\ccal_{T_n-}^{(z)}(\Pi^{(t)}) = \ccal_{T_{n-1}}^{(z)}(\Pi^{(t)})$. By \eqref{eq:induction}, this reproduction event also intersects $B^{(t,z)}_{T_{n-1}}$, and therefore by the definition of the ancestral skeleton, \eqref{eqn:condition_on_pi} and \eqref{eq:induction},
\[B^{(t,z)}_{T_n} = \bcal(z_1,r_1)\cup B^{(t,z)}_{T_{n-1}} = \bcal(z_1,r_1) \cup \bigcup_{(z',r') \in \ccal_{T_{n-1}}^{(z)}(\Pi^{(t)})}\bcal(z',r') = \bigcup_{(z',r') \in \ccal_{T_{n}}^{(z)}(\Pi^{(t)})}\bcal(z',r').\]
By induction, the proof is complete.
\end{proof}

If we start our $\infty$-parent SLFV process from a singleton, then our process satisfies a very strong form of self-duality, in the following sense.

\begin{lem}\label{lem:strong_self_duality}
For any $w,z\in\rtwo$ and $t\ge 0$, we have
\[z\in\bigcup_{(z',r') \in \ccal_{t}^{(w)}(\Pi^{(t)})} \bcal(z',r')  \,\,\,\, \Longleftrightarrow \,\,\,\, w\in \bigcup_{(z',r')\in\ccal_t^{(z)}(\Pi)} \bcal(z',r').\]
\end{lem}

\begin{proof}
Since $\left(\Pi^{(t)}\right)^{(t)} = \Pi$, it suffices to show that the implication holds in one direction, so suppose that
\[z\in\bigcup_{(z',r') \in \ccal_{t}^{(w)}(\Pi^{(t)})} \bcal(z',r').\]
The basic idea is that any chain of $\Pi^{(t)}$-reproduction events that reaches $z$ from $w$ can be reversed to give a chain of $\Pi$-reproduction events that reaches $w$ from $z$. To be precise, take a chain $\mathbf{C} = (Z_s,R_s)_{s\ge0} \in\ccal^{(w)}(\Pi^{(t)})$ such that $z\in\bcal(Z_t,R_t)$. Let $0\le t_1<\ldots<t_n\le t$ be the jump times of $\mathbf{C}$ before $t$ (where we say that there is a jump at time $0$, i.e. $t_1=0$, if $(Z_0,R_0)\neq (w,0)$). Define a new chain $\tilde{\mathbf{C}}$ as follows:
\begin{itemize}
\item set $\tilde{\mathbf{C}}_s = (z,0)$ for all $s\in[0,t-t_n)$;
\item for each $j=n,n-1,\ldots,2$, set $\tilde{\mathbf{C}}_s = \mathbf C_{t_n}$ for all $s\in[t-t_j,t-t_{j-1})$;
\item set $\tilde{\mathbf{C}}_s = \mathbf C_{t_1}$ for all $s\ge t-t_1$.
\end{itemize}
It is easy to check straight from the definition that $\tilde{\mathbf{C}}\in\ccal^{(z)}(\Pi)$, so since $w\in \mathbf C_{t_1} = \tilde{\mathbf{C}}_t$, the proof is complete.
\end{proof}

As a first corollary of this duality, if we start our $\infty$-parent SLFV process from a singleton, then we can represent $S_t$ directly in terms of chains of $\Pi$-reproduction events, rather than $\Pi^{(t)}$-reproduction events as in Lemma \ref{lem:equivalence_dual_chain}.

\begin{cor}\label{cor:forwards_St}
Consider the $\infty$-parent SLFV process started from a singleton, $E=\{z\}$ for some $z\in\rtwo$. Then
\[S_t^{\{z\}} = \bigcup_{(z',r')\in\ccal_t^{(z)}(\Pi)} \bcal(z',r').\]
\end{cor}

\begin{proof}
For any $w\in\rtwo$, note that by the definition of the $\infty$-parent SLFV and Lemma \ref{lem:equivalence_dual_chain},
\[\{w\in S_t^{\{z\}}\} = \{z\in B^{(t,w)}_t\} = \Bigg\{z\in\bigcup_{(z',r') \in \ccal_{t}^{(w)}(\Pi^{(t)})} \bcal(z',r') \Bigg\}.\]
Lemma \ref{lem:strong_self_duality} then allows us to conclude.
\end{proof}

We now rephrase the first hitting time of a point $z$, $\hit{E}{z}$, and the bulk coverage time $\bulk{E}{z}$, in terms of chains of reproduction events. For an arbitrary measurable initial state $E$, we need to use $\Pi^{(t)}$-reproduction events.

\begin{lem}\label{lem:hitting_times_charac_chains}
Consider the $\infty$-parent SLFV started from a measurable set $E$. For all $z \in E^{c}$, we have
\[\hit{E}{z} = \min\left\{ t \geq 0 : \exists (z',r') \in \ccal_{t}^{(z)}(\Pi^{(t)}) \text{ with } \bcal(z',r') \cap E \neq \emptyset \right\}.\]
Moreover, if $(0,0)\in E$, then for all $z \in E^{c}$, we have  
\[\bulk{E}{z} = \min\left\{ t \geq 0 : \forall u\in(0,1],\, \exists \left(z',r'\right) \in \ccal_{t}^{(uz)}(\Pi^{(t)}) \text{ with } \bcal\left( z',r' \right) \cap E \neq \emptyset \right\}.\]
\end{lem}

\begin{proof}
Let $z \in E^{c}$. By definition of the $\infty$-parent SLFV, and using Lemma~\ref{lem:equivalence_dual_chain} to pass from the second to the third line,
\begin{align*}
    \hit{E}{z} &= \min\left\{
t \geq 0 : z \in S^E_{t}
    \right\} \\
    &= \min\left\{
t \geq 0 : E\cap B_{t}^{(t,z)} \neq \emptyset
    \right\} \\
    &= \min\left\{
t \geq 0 : E\cap
\bigcup_{(z',r') \in \ccal_{t}^{(z)}(\Pi^{(t)})} \bcal(z',r') \neq \emptyset 
    \right\} \\
    &= \min\left\{
t \geq 0 : \exists (z',r') \in \ccal_{t}^{(z)}(\Pi^{(t)}) \text{ with } E\cap \bcal(z',r') \neq \emptyset
    \right\}. 
\end{align*}
We proceed similarly for $\bulk{E}{z}$. 
\end{proof}

However, if our initial state $E$ is a singleton, then Corollary \ref{cor:forwards_St} simplifies the expression significantly, allowing us to express $\hit{E}{z}$ and $\bulk{E}{z}$ directly in terms of $\Pi$-reproduction events.

\begin{lem}\label{lem:hitting_times_charac_forward}
Consider the $\infty$-parent SLFV started from the origin, $E=\{0\}$. Then for all $z\in\rtwo\setminus\{0\}$, we have
\[\hit{\{0\}}{z} = \min\left\{ t \geq 0 : \exists (z',r') \in \ccal_{t}^{(0)}(\Pi) \text{ with } z\in \bcal(z',r')\right\}\]
and
\[\bulk{\{0\}}{z} = \min\left\{ t \geq 0 : \forall u\in(0,1],\, \exists \left(z',r'\right) \in \ccal_{t}^{(0)}(\Pi) \text{ with } uz \in \bcal\left( z',r' \right) \right\}.\]
\end{lem}

\begin{proof}
Take $z\in\rtwo\setminus\{0\}$. The proof is almost identical to that of Lemma \ref{lem:hitting_times_charac_chains} but using Lemma~\ref{cor:forwards_St} in place of Lemma~\ref{lem:equivalence_dual_chain}:
\begin{align*}
    \hit{\{0\}}{z} &= \min\left\{
t \geq 0 : z \in S^{\{0\}}_{t}
    \right\} \\
    &= \min\left\{
t \geq 0 : z\in\bigcup_{(z',r')\in\ccal_t^{(0)}(\Pi)} \bcal(z',r')
    \right\} \\
    &= \min\left\{
t \geq 0 : \exists (z',r') \in \ccal_{t}^{(0)}(\Pi) \text{ with } z\in \bcal(z',r')
    \right\}, 
\end{align*}
and similarly for $\bulk{\{0\}}{z}$. 
\end{proof}

As another example application of self-duality, we observe the equality in distribution of plane-to-point and point-to-plane hitting times. Note that such a result was already obtained in \cite{louvet2023measurevalued} for another version of the $\infty$-parent SLFV: see \cite[Lemma~3.8]{louvet2023measurevalued}. 

\begin{cor}\label{cor:hitting_times_equal_distn}
For any $z=(x,y)\in\rtwo$, the two hitting times $\hit{\hcal}{z}$ and $\hit{\{0\}}{\hcal^x}$ are equal in distribution.
\end{cor}

\begin{proof}
Using the translation and reflection invariance of the underlying Poisson point process, the fact that $S_{t}^{\hcal}$ is a non-decreasing set, and then the definition of $S^{\hcal^x}_t$, we have
\[\P\left(\hit{\hcal}{x,0}\le t\right) = \P\left(z\in S^\hcal_t\right) = \P\left(0\in S^{\hcal^x}_t\right) = \P\left(B^{(t,0)}_t\cap \hcal^x \neq \emptyset \right).\]
By Lemma \ref{lem:equivalence_dual_chain}, this equals
\[\P\left(\hcal^x\cap \bigcup_{(z',r') \in \ccal_{t}^{(0)}(\Pi^{(t)})} \bcal(z',r') \neq \emptyset \right),\]
and since $\Pi$ and $\Pi^{(t)}$ are equal in distribution, this equals
\[\P\left(\hcal^x\cap \bigcup_{(z',r') \in \ccal_{t}^{(0)}(\Pi)} \bcal(z',r') \neq \emptyset \right).\]
But by Corollary \ref{cor:forwards_St}, we recognise this as precisely
\[\P\left(\hcal^x\cap S^{\{0\}}_t \neq \emptyset\right) = \P\left(\hit{\{0\}}{\hcal^x}\le t\right)\]
and the proof is complete.
\end{proof}

\subsection{Using a slow coverage strategy to upper bound coverage times}\label{subsec:slow_coverage}

The main aim in this section is to use the framework of chains of reproduction events to show that~$\E[\bulk{E}{x,0}]$ is \textit{at most} linear in~$x$, in the following sense.

\begin{lem}\label{lem:upper_bound_sigma_x}
Consider the $\infty$-parent SLFV started from any measurable set $E\subset\rtwo$ such that $(0,0)\in E$. There exists a constant $C > 0$ depending only on $\mu$ (not on $E$) such that for all $x>0$, 
\begin{equation*}
    \esp\left[ \bulk{E}{x,0} \right] \leq Cx.
\end{equation*}
\end{lem}

In order to show this result, our tactic will be to construct an explicit sequence of reproduction events which swallows up the horizontal axis bit by bit. This chain is very unlikely to be the one that corresponds to $\bulk{E}{x,0}$; we think of it  as a ``slow but sure'' way of covering the axis up to $(x,0)$ that will provide a useful upper bound.

Let $\delta, \eta > 0$ be such that $\mu((3\delta, \infty)) \geq \eta$. For all $j \in \mathbb{Z}$, we set $x_{j} = j\delta$. We define the sequence of reproduction events 
\begin{equation}\label{eq:slow_covering_chain}
\left(\overrightarrow{T}\!_{j}, \overrightarrow{Z}\!_{j}, \overrightarrow{R}\!_{j}\right)_{j \geq 1}
\end{equation}
recursively as follows. First we set $\overrightarrow{T}\!_{0} = 0$. Then, for all $j \geq 1$, let
\begin{equation*}
    \left(
\overrightarrow{T}\!_{j}, \overrightarrow{Z}\!_{j}, \overrightarrow{R}\!_{j}
    \right) \in \Pi
\end{equation*}
be the first reproduction event of radius $\overrightarrow{R}\!_{j} > 3\delta$ that occurs after time $\overrightarrow{T}\!_{j-1}$ and such that
\begin{equation*}
    \overrightarrow{Z}\!_{j} \in (x_{j-1},x_{j}) \times (-\delta, \delta). 
\end{equation*}
This sequence of reproduction events satisfies the following properties. 

\begin{lem}\label{lem:properties_sequence_reproduction_events} (i) The random variables
\begin{equation*}
    \left(
\overrightarrow{T}\!_{j} - \overrightarrow{T}\!_{j-1}
    \right)_{j \geq 1}
\end{equation*}
are i.i.d. and exponentially distributed, with rate bounded from below by $2\delta^{2}\eta$. 

\noindent (ii) For all $j \geq 1$, 
\begin{equation*}
    \left\{
(x,y)  : x_{j-2} \leq x \leq x_{j+1},\, -\delta\le y \le \delta
    \right\} \subseteq \bcal\left(
\overrightarrow{Z}\!_{j}, \overrightarrow{R}\!_{j}
    \right). 
\end{equation*}
In particular, 
\begin{equation*}
    \bcal\left(
\overrightarrow{Z}\!_{j}, \overrightarrow{R}\!_{j}
    \right) \cap \bcal\left(
\overrightarrow{Z}\!_{j+1}, \overrightarrow{R}\!_{j+1}
    \right) \neq \emptyset . 
\end{equation*}
\end{lem}

\begin{proof}
(i) This is a direct consequence of the fact that $\mu((2\delta,+\infty)) \geq \eta$. 

\noindent (ii) Let $j \geq 1$. Then, as $\overrightarrow{Z}\!_{j} \in (x_{j-2},x_{j}) \times (-\delta, \delta)$ and $\overrightarrow{R}\!_{j} > 3\delta$, 
\[\max\left\{
d\left(
\overrightarrow{Z}\!_{j},(x,y)
\right) : x_{j-1} \leq x \leq x_{j+1},\, -\delta\le y\le \delta
\right\} 
\leq \sqrt{8}\delta 
< \overrightarrow{R}\!_{j}.\]
Thus 
\[\left\{
(x,y): x_{j-2} \leq x \leq x_{j},\, -\delta\le y\le \delta
    \right\} \subseteq \bcal\left(
\overrightarrow{Z}\!_{j}, \overrightarrow{R}\!_{j}
    \right),\]
and in particular
\[\bcal\left( \overrightarrow{Z}\!_{j}, \overrightarrow{R}\!_{j} \right) \cap \bcal\left(
\overrightarrow{Z}\!_{j+1}, \overrightarrow{R}\!_{j+1}
\right) \supseteq \left\{
(x_{j},0)
\right\} \neq \emptyset.\qedhere\]
\end{proof}

Therefore, we can construct chains of $\Pi$-reproduction events $\mathbf{C}^{(x)} = ( \mathbf{Z}_{t}^{(x)}, \mathbf{R}_{t}^{(x)} )_{t \geq 0}, x > 0$
as follows. For each $x > 0$,  
\begin{enumerate}
    \item We set $\mathbf{C}^{(x)}_{t} = (0,0)$ for all $t\in[0,\overrightarrow{T}\!_1)$.
    \item For each $j\in\{1,2,\ldots,\lceil x/\delta\rceil-1\}$, we set
    \[\mathbf{C}^{(x)}_{t} = \left( \overrightarrow{Z}\!_{j}, \overrightarrow{R}\!_{j} \right) \,\,\,\, \text{ for all } t\in[\overrightarrow{T}\!_j,\overrightarrow{T}\!_{j+1}).\]
    \item We set 
    \[\mathbf{C}^{(x)}_{t} = \left( \overrightarrow{Z}\!_{\lceil x/\delta\rceil}, \overrightarrow{R}\!_{\lceil x/\delta\rceil} \right) \,\,\,\, \text{ for all } t\ge \overrightarrow{T}\!_{\lceil x/\delta\rceil}.\]
\end{enumerate}
In other words, all the chains $\mathbf{C}^{(x)}$, $x > 0$ start the same, but each one of them is stopped once it reaches location~$(x,0)$. 
We then have the following result. 

\begin{lem}\label{lem:express_chain_bound} 
Consider the $\infty$-parent SLFV started from the origin, $E=\{0\}$. Let $\delta, \eta > 0$ be such that $\mu((3\delta, \infty)) \geq \eta$, and construct
\[\left(\overrightarrow{T}\!_{j}, \overrightarrow{Z}\!_{j}, \overrightarrow{R}\!_{j}\right)_{j \geq 1}\]
as in \eqref{eq:slow_covering_chain}. Then for all $x > 0$,
\begin{equation*}
\sup_{z\in\bcal((x,0),\delta)} \bulk{\{0\}}{z} \leq \overrightarrow{T}\!_{\lceil x/\delta \rceil}.     
\end{equation*}
\end{lem}

\begin{proof}
Let $x > 0$ and take $z\in\bcal((x,0),\delta)$. By Lemma \ref{lem:hitting_times_charac_forward},
\[\bulk{\{0\}}{z} = \min\left\{ t \geq 0 : \forall u\in(0,1],\, \exists \left(z',r'\right) \in \ccal_{t}^{(0)}(\Pi) \text{ with } uz \in \bcal\left( z',r' \right) \right\}.\]
Consider $\mathbf{C}^{(ux)}$ as constructed above. By construction, this is a chain of $\Pi$-reproduction events, with $uz\in\mathbf{C}^{(ux)}_t$ for all $t\ge \overrightarrow{T}\!_{\lceil ux/\delta\rceil}$. Thus, for for every $u\in(0,1]$, since $\overrightarrow{T}\!_{\lceil ux/\delta\rceil}\le \overrightarrow{T}\!_{\lceil x/\delta\rceil}$, we have a chain of $\Pi$-reproduction events---namely $\mathbf{C}^{(ux)}$---with $uz\in\mathbf{C}^{(ux)}_t$ for all $t\ge \overrightarrow{T}\!_{\lceil x/\delta\rceil}$. We deduce the result.
\end{proof}

We can now conclude the proof of Lemma~\ref{lem:upper_bound_sigma_x}.

\begin{proof}[Proof of Lemma \ref{lem:upper_bound_sigma_x}]
By the obvious coupling, the expectation is largest when~$E=\{0\}$, so we consider this case only. Let $x > 0$. By Lemma~\ref{lem:express_chain_bound},
\[\esp\left[ 
\bulk{\{0\}}{x,0} 
\right] \leq \esp\left[ 
\overrightarrow{T}\!_{\lceil x/\delta \rceil} 
\right] = \esp\left[
\sum_{j = 1}^{\lceil x/ \delta \rceil} \left(
\overrightarrow{T}\!_{j} - \overrightarrow{T}\!_{j-1}
\right) 
\right] = \left\lceil \frac{x}{\delta} \right\rceil \esp\left[ 
\overrightarrow{T}\!_{1} - \overrightarrow{T}\!_{0}
\right]\leq \left\lceil \frac{x}{\delta} \right\rceil \times \frac{1}{2\delta^{2} \eta}\]
as required.
\end{proof}

In fact, later we will need to use the same strategy in a slightly more general context; once we have run our process for some time, we would then like to know that we can still cover the line between $(0,0)$ and $(x,0)$ in linear time, without using the events already seen. More precisely, we have the following result.

\begin{lem}\label{lem:more_general_upper_bound_sigma}
Consider the $\infty$-parent SLFV process started from a measurable set $E\subset\rtwo$ such that $(0,0) \in E$. There exists $C > 0$ depending only on $\mu$ (not on $E$) such that for all $T > 0$ and $x > 0$, there exists a random variable $\sigma\!_{x,T}^{(slow)}$ independent of the event $\{\bulk{E}{x,0} > T\}$ such that
\begin{align*}
    \esp\left[ 
    \sigma_{x,T}^{(slow)}
    \right] \leq Cx \quad \text{and} \quad \bulk{E}{x,0} \leq T + \sigma_{x,T}^{(slow)}.
\end{align*}
\end{lem}

\begin{proof}
We proceed exactly as before, except that we now only use reproduction events occurring strictly after time~$T$ (in~$\Pi$) to construct~$\sigma_{x,T}^{(slow)}$. 
\end{proof}

The sequence of reproduction events
\[\left( \overrightarrow{T}_{j}, \overrightarrow{Z}_{j}, \overrightarrow{R}_{j} \right)_{j \geq 1}\]
introduced in (\ref{eq:slow_covering_chain}) can also be used to obtain tail estimates on $\bulk{E}{z}$. Indeed, by Lemma~\ref{lem:express_chain_bound}, for all $x > 0$, 
\begin{equation*}
    \sup_{z\in\bcal((x,0),\delta)} \bulk{\{0\}}{z} \leq \overrightarrow{T}\!_{\lceil x/\delta \rceil}. 
\end{equation*}
Therefore, to show Proposition~\ref{prop:tail_tau_x}, subject to translation and rotation, it is sufficient to obtain a tail estimate on $\overrightarrow{T}\!_{\lceil x/\delta \rceil}/x$.

\begin{lem}\label{lem:tail_estimate_slow_covering_chain}
For all $\beta > 3\eta^{-1}\delta^{-2}$ and for all $x > 0$, 
\begin{equation*}
    \proba\left(
\overrightarrow{T}\!_{\lceil x/\delta \rceil} > \beta x
    \right) \leq \exp\left(
-\delta\eta \beta x
    \right). 
\end{equation*}
\end{lem}

\begin{proof}
Let $\beta > 3\eta^{-1}\delta^{-2}$ and $x > 0$. As $\overrightarrow{T}\!_{0} = 0$, we have
\begin{equation*}
    \overrightarrow{T}\!_{\lceil x/\delta \rceil} = \sum_{j = 1}^{\lceil x/\delta \rceil}\left(
\overrightarrow{T}\!_{j} - \overrightarrow{T}\!_{j-1}
    \right). 
\end{equation*}
By Lemma~\ref{lem:properties_sequence_reproduction_events} (i), the random variables
\begin{equation*}
    \left(
\overrightarrow{T}\!_{j} - \overrightarrow{T}\!_{j-1}
    \right)_{j \geq 1}
\end{equation*}
are i.i.d. and exponentially distributed, with rate bounded from below by~$2\delta^{2}\eta$. Therefore, 
\[ \proba\left( \overrightarrow{T}\!_{\lceil x/\delta \rceil} > \beta x \right) = \proba\left( \sum_{j = 1}^{\lceil x/\delta \rceil} \left( \overrightarrow{T}\!_{j} - \overrightarrow{T}\!_{j-1} \right) > \beta x \right)
\leq \proba\left( \sum_{i = 1}^{\lceil x/\delta \rceil} \ecal_{i} > 2\beta x \delta^{2}\eta \right), \]
where $(\ecal_{i})_{i \geq 1}$ are independent exponential random variables with parameter~$1$. A standard Chernoff bound (consider $\esp[e^{\lambda \sum \ecal_{i}}]$ with $\lambda = 1 - 1/\alpha$) gives that for all~$n \in \mathbb{N} \backslash \{0\}$ and~$\alpha > 1$, 
\begin{equation*}
    \proba\left(
\sum_{i = 1}^{n} \ecal_{i} > \alpha n
    \right) \leq \left(
\alpha e^{1 - \alpha}
    \right)^{n}. 
\end{equation*}
It is easy to check that for~$\alpha > 6$, we have $\alpha e^{1-\alpha} \leq e^{-\alpha / 2}$. Thus, taking $n = \lceil x/\delta \rceil$ and $\alpha = 2\beta \delta^{2}\eta$, we obtain
\begin{equation*}
    \proba\left(
\overrightarrow{T}_{\lceil x/\delta \rceil} > \beta x
    \right) \leq \left(
\exp\left(
-\beta \delta^{2} \eta
\right)\right)^{\lceil x/\delta \rceil} \leq \exp\left(
-\beta x \delta \eta
\right), 
\end{equation*}
which completes the proof. 
\end{proof}

We can now prove Proposition~\ref{prop:tail_tau_x}. 

\begin{proof}[Proof of Proposition~\ref{prop:tail_tau_x}]
Since trivially
\[\hit{\{0\}}{z} \le \sup_{z'\in\bcal(z,\delta)} \bulk{\{0\}}{z'},\]
it suffices to prove the first part of the lemma. By invariance under rotation of the distribution of the underlying Poisson point process, we may assume without loss of generality that $z=(x,0)$, and since the result is trivial when $z=0$, we may assume that $x>0$.
Let~$\beta > 3\eta^{-1}\delta^{-2}$. Then by Lemma \ref{lem:express_chain_bound} we have $\sup_{z\in\bcal((x,0),\delta)} \bulk{\{0\}}{z} \leq \overrightarrow{T}\!_{\lceil x/\delta \rceil}$, so by Lemma~\ref{lem:tail_estimate_slow_covering_chain},
\begin{equation*}
\proba\left( \sup_{z'\in\bcal((x,0),\delta)} \bulk{\{0\}}{z'} > \beta x \right) \le \proba\left( \overrightarrow{T}\!_{\lceil x/\delta \rceil} > \beta x \right) \leq e^{ -\delta \eta \beta x } = e^{ - \delta \eta \beta \|z\|}.
\end{equation*}
We conclude by noting that acceptable values for~$\delta$ and $\eta$ (and therefore $\beta$) are entirely determined by~$\mu$. 
\end{proof}

\subsection{Tree representation of chains of reproduction events}\label{subsec:tree_representation}
We now introduce a tree-based representation of the set of chains of reproduction events. Indeed, the tree structure emerges naturally in the following way. Each chain of reproduction events implicitly chooses, each time a reproduction event \emph{could} be included in its chain, whether to include that event or reject it. We thus obtain a binary tree of possible choices, each path of which corresponds to a chain of reproduction events. We now provide the details.

As with the definition of chains of reproduction events, we work with a general Poisson point process $\Xi$ on $\rmath \times \rtwo \times (0,+\infty)$, but in practice we will always use $\Xi=\Pi$ or $\Xi = \Pi^{(t)}$ for some $t$. Given such a $\Xi$ and $z\in\rtwo$, let $\tcal^{[z]}(\Xi)$ be the rooted tree in which each vertex has two descendants, with vertices denoted according to the Ulam-Harris labelling, and each vertex $u$ also has three associated random variables $T_{u}^{[z]}$, $Z_{u}^{[z]}$ and $R_{u}^{[z]}$ introduced below. We recall that in the Ulam-Harris labelling, the root has label $\emptyset$, and the other vertices have labels formed by strings of $1$s and $2$s, with e.g.~vertex $12$ representing the second child of the first child of the root. For a vertex $u$, we write $u\cdot 1$ and $u\cdot 2$ for its children (so the $\cdot$ notation represents the concatenation of strings); and $\overleftarrow{u}$ for its parent. Moreover, we write $v \leq u$ if $v$ is an ancestor of~$u$, and $v < u$ if $v$ is a strict ancestor of~$u$ (i.e.~$v \neq u$).

To each vertex $u \in \tcal^{[z]}(\Xi)$, we associate three random variables $T_{u}^{[z]}$, $Z_{u}^{[z]}$ and $R_{u}^{[z]}$ recursively as follows. First we set $Z_{\emptyset}^{[z]} = z$, $R_{\emptyset}^{[z]} = 0$ and
\[T_{\emptyset}^{[z]} := \min\left\{ t \ge 0 : \exists (z,r) \in \rtwo \times (0,+\infty) \text{ with } (t,z,r) \in \Xi \text{ and } Z_{\emptyset}^{[z]} \in \bcal(z,r) \right\}.\]
The pair $(z,r)$ is unique a.s. (and will always be unique on \eqref{eqn:condition_on_pi}), which allows us to set
\[ \left( Z_{1}^{[z]}, R_{1}^{[z]} \right) = \left( Z_{\emptyset}^{[z]}, R_{\emptyset}^{[z]} \right) \text{ and } \left( Z_{2}^{[z]}, R_{2}^{[z]} \right) = (z,r);\]
in other words, the first child of the root rejects the first reproduction event, and the second child accepts it.
Then, for all $u \in \tcal^{[z]}(\Xi) \backslash \{\emptyset\}$, we set
\begin{multline*}
    T_{u}^{[z]} := \min\left\{ t > 0 : \exists (z,r) \in \rtwo \times (0,+\infty) \text{ with } \left( t + \sum_{\emptyset \leq v < u} T_{v}^{[z]}, z, r \right) \in \Xi \right. \\
\left.
\phantom{\sum_{\emptyset \leq v < u} T_{v}^{[z]}} \text{ and } \bcal\left( Z_{u}^{[z]}, R_{u}^{[z]} \right) \cap \bcal(z,r) \neq \emptyset
    \right\}. 
\end{multline*}
Moreover, we then set
\[ \left( Z_{u \cdot 1}^{[z]}, R_{u \cdot 1}^{[z]} \right) = \left( Z_{u}^{[z]}, R_{u}^{[z]} \right) \text{ and } \left( Z_{u \cdot 2}^{[z]}, R_{u \cdot 2}^{[z]} \right) = (z,r),\]
where $(z,r)$ is the almost surely unique pair appearing in the definition of $T_{u}^{[z]}$ above.

We note here that this tree does not satisfy the ``branching property'' that a subtree rooted at a vertex $u$ is independent of the rest of the tree given the values of $T_{u}^{[z]}$, $Z_{u}^{[z]}$ and $R_{u}^{[z]}$. Indeed, each reproduction event can affect several different parts of the tree simultaneously, since the corresponding ball $\bcal(z,r)$ can intersect several previous reproduction events. Nevertheless, the tree representation will be useful to formalise the discussion of geodesics in the following section. Of particular interest is the tree $\tcal = \tcal^{[(0,0)]}(\Pi)$.

\section{Fluctuations of geodesics: precise statements and proofs of Theorems \ref{thm:geodesics} and \ref{thm:geodesic_plane_to_point}}\label{sec:geodesics}

\subsection{Definition of geodesics from point to set and precise statement of Theorem \ref{thm:geodesics}}

We begin with Theorem \ref{thm:geodesics}, leaving the more delicate (but shorter) Theorem \ref{thm:geodesic_plane_to_point} for Section \ref{sec:geodesics_plane_to_point}. For now, we focus on the $\infty$-parent SLFV $(S^E_{t})_{t \geq 0}$ when started from a singleton $z$, usually the origin, i.e.~$E=\{(0,0)\}$. Our first goal is to formalise the statement of Theorem~\ref{thm:geodesics}. Recall that for $A\subset\rtwo$, $\hit{E}{A}$ is the first hitting time of the set $A$,
\[\hit{E}{A} := \inf\left\{ t \geq 0 : A\cap S^E_{t}\neq\emptyset \right\}.\]
Note that when we start our process from a singleton, e.g.~$E=\{(0,0)\}$, the infimum above is in fact a minimum, as a consequence of the fact that the process jumps at a finite rate. Indeed, its jump rate is bounded from above by that of a Yule process in which each particle splits in two at a rate proportional to $\text{M}_0$, where $\text{M}_0$ was defined in \eqref{cond:finite_rate}. Moreover, in this case, by Corollary \ref{cor:forwards_St}, for all $t \geq 0$, if $S_{t}^{\{z\}} \neq S_{t-}^{\{z\}}$, then a.s. there exists  $z' \in \rtwo$ and $r' > 0$ such that $(t,z',r') \in \Pi$ and
\begin{equation}\label{eqn:one_endpoint}
    S_{t}^{\{z\}} =  S_{t-}^{\{z\}} \cup \bcal(z,r). 
\end{equation}
We will particularly focus on the case when $A$ is the half-plane of points to the right of $x$,
\[ \mathcal{H}^{x} := \left\{ (x',y') \in \rtwo : x' \geq x \right\}.\]

We now define the concept of geodesics from a point $z\in\rtwo$ to a set~$A\subseteq\rtwo$ for the $\infty$-parent SLFV.

\begin{definition}\label{defn:geodesics}
For all $z\in\rtwo$ and $A\subseteq\rtwo$, a $\Pi$-\textbf{geodesic} from~$z$ to~$A$ is a chain of $\Pi$-reproduction events
\[\mathbf{C}^{[G]} = \left( Z_{t}^{[G]}, R_{t}^{[G]} \right)_{t\ge 0} \in \ccal^{(z)}(\Pi)\]
such that
\[\bcal\left( Z_{\hit{\{z\}}{A}}^{[G]}, R_{\hit{\{z\}}{A}}^{[G]} \right) \cap A \neq \emptyset.\]
For all~$x \geq 0$, let $\gcal^{x}(\Pi)$ be the set of all $\Pi$-geodesics from~$(0,0)$ to~$\mathcal{H}^{x}$. 
\end{definition}
Notice that by (\ref{eqn:one_endpoint}), a.s. all the $\Pi$-geodesics from~$z$ to~$A$ have the same endpoint at time $\hit{\{z\}}{A}$. That is, if~$\mathbf{C}^{[G],1}$ and~$\mathbf{C}^{[G],2}$ are two such $\Pi$-geodesics, then a.s.
\begin{equation*}
    \mathbf{C}_{\hit{\{z\}}{A}}^{[G],1} = \mathbf{C}_{\hit{\{z\}}{A}}^{[G],2}. 
\end{equation*}

For any $L > 0$, we define the strip of radius $L$ about the $x$-axis
\[ \lcal^{(L)} :=\left\{ (x,y) \in \rtwo : |y| \leq L \right\}. \]
We can now state the precise version of Theorem~\ref{thm:geodesics}.

\theoremstyle{plain}
\newtheorem*{geodesics}{Theorem~\ref{thm:geodesics} (precise statement)}
\begin{geodesics}
(i) (There is a geodesic that does not fluctuate more than $\sqrt x$ with high probability.) For all $\eps > 0$, there exists $A_{\eps} \in(0,\infty)$ such that for all $x > 1$, 
\begin{equation*}
\proba\left( \exists \mathbf{C}^{[G]} \in \gcal^{x}(\Pi) : \bcal\left( Z_{s}^{[G]}, R_{s}^{[G]} \right) \subseteq \lcal^{(A_{\eps}\sqrt{x})} \,\, \forall s \in \left[ 0,\hit{\{0\}}{\mathcal{H}^{x}} \right]\right) \geq 1-\eps . 
\end{equation*}

\noindent (ii) (Tail bounds on fluctuations of larger order.) For all $\delta>0$ and $\beta\in(1/2,1]$, there exists $c>0$ such that for all sufficiently large $x$,
\begin{equation*}
\proba\left( \exists \mathbf{C}^{[G]} \in \gcal^{x}(\Pi) : \bcal\left( Z_{s}^{[G]}, R_{s}^{[G]} \right) \subseteq \lcal^{(\delta x^\beta)} \,\, \forall s \in \left[ 0,\hit{\{0\}}{\mathcal{H}^{x}} \right]\right) \geq 1-\exp(-cx^{2\beta-1}).
\end{equation*}

\noindent (iii) (Endpoints of geodesics are on the order of $\sqrt x$ with high probability.) For all $\eps>0$, there exists $a_\eps>0$ such that for all $x$ large enough,
\begin{equation*}
\proba\left( \forall \mathbf{C}^{[G]} \in \gcal^{x}(\Pi),\,\, \bcal\left( Z_{\hit{\{0\}}{\mathcal{H}^{x}}}^{[G]}, R_{\hit{\{0\}}{\mathcal{H}^{x}}}^{[G]} \right) \nsubseteq \lcal^{(a_\eps\sqrt{x})} \right) \geq 1-\eps. 
\end{equation*}

\end{geodesics}

We note again here that we expect that in fact all geodesics from $0$ to $\hcal^x$ have fluctuations of order $\sqrt x$, but we do not currently have a bound on how
many such geodesics there are (in general there will be more than one) and therefore parts (i) and (ii) do not rule out that there may be ``unusual'' geodesics that fluctuate on a larger scale. However, all geodesics from $0$  to $\hcal^x$ end at the same point, so this issue does not arise with part (iii).

We will prove parts (i) and (ii) of this theorem in Section \ref{subsec:upper_bound_displacement} and part (iii) in Section \ref{subsubsec:lower_bound_displ}.

\subsection{Link with the tree representation: the random geodesic path}\label{subsec:random_geodesic}

One complication in our model is that there may be (in fact, we often expect there to be) many $\Pi$-geodesics from $0$ to the half-plane $\hcal^{x}$ for $x>0$. It will be useful for us to identify one such $\Pi$-geodesic, which we choose at random in a natural way. We recall the tree representation from Section \ref{subsec:tree_representation}, and in particular recall that $\tcal = \tcal^{[0]}(\Pi)$ is the tree started from the origin and built using the Poisson point process $\Pi$.

We write $(T_u,Z_u,R_u)$ for the random variables associated to vertex $u\in\tcal$, and write $Z_u = (X_u,Y_u)$. For each vertex $u$ except the root, we introduce the two $\rmath$-valued random variables $\Delta^X_{u} = X_{u}-X_{\overleftarrow u}$ and $\Delta^Y_{u} = Y_{u}-Y_{\overleftarrow u}$, the increments of $u$ relative to its parent in the $x$- and $y$-directions respectively. We also let $\Gamma_u = \sum_{v<u}T_v$ be the ``birth time'' of $u$.

Let $\tcal(x)$ be the sub-tree of $\tcal$ constructed by keeping only paths which are in $\mathcal{H}^{x}$ at all times $t \geq \hit{\{0\}}{ \mathcal{H}^{x}}$. That is, for all $u \in \tcal(x)$, 
\begin{itemize}
    \item If $\Gamma_u \geq \hit{\{0\}}{ \mathcal{H}^{x}}$ (that is, if $u$ is born at or after time $\hit{\{0\}}{ \mathcal{H}^{x}}$), then $X_u + R_{u} \geq x$.
    \item If $\Gamma_u < \hit{\{0\}}{ \mathcal{H}^{x}}$ (that is, if $u$ is born before time $\hit{\{0\}}{ \mathcal{H}^{x}}$), then there exists a path within $\tcal(x)$ starting from $u$ and leading to some $u' \in \tcal$ born at time $\hit{\{0\}}{ \mathcal{H}^{x}}$ and satisfying $X_{u'}+R_{u'} \geq x$.
\end{itemize}
In other words, all paths in $\tcal(x)$ contain particles that were the first to hit the half-plane $\mathcal{H}^x$. We now specify our choice of geodesic.

\begin{definition}[Random geodesic path]
For all $x > 0$, the random geodesic path from $0$ to $\hcal^x$, $(\Theta_{n}^{(x)})_{n \geq 0}$, is the random path over $\tcal$ started from the root and such that for all $n \geq 0$:
\begin{itemize}
    \item If $\Theta_{n}^{\ancx} \cdot 1 \notin \tcal(x)$ (resp. $\Theta_{n}^{\ancx} \cdot 2 \notin \tcal(x)$), then $\Theta_{n+1}^{\ancx} = \Theta_{n}^{\ancx}\cdot 2$ (resp. $\Theta_{n}^{\ancx} \cdot 1$). 
    \item Otherwise, both $\Theta_{n}^{\ancx} \cdot 1$ and $\Theta_{n}^{\ancx} \cdot 2$ are in $\tcal(x)$, and $\Theta_{n+1}^{\ancx} = \Theta_{n}^{\ancx} \cdot 1$ (resp. $\Theta_{n}^{\ancx} \cdot 2$) with probability $1/2$. 
\end{itemize}
\end{definition}

\noindent
This path is well-defined, as the root is in $\tcal(x)$, and every vertex in $\tcal(x)$ has at least one child in $\tcal(x)$. As its name suggests, it also characterises a $\Pi$-geodesic from $0$ to $\hcal^x$. For each $j\ge 0$, define
\[Z^{[x]}_t := Z_{\Theta^{(x)}_j} \,\,\,\,\text{ and } R^{[x]}_t := R_{\Theta^{(x)}_j} \,\,\,\,\text{ for all } t\in\Big[\Gamma_{\Theta^{(x)}_j},\Gamma_{\Theta^{(x)}_{j+1}}\Big),\]
where we recall that $\Gamma_{u}$, $u \in \tcal$ is the birth time of~$u$, 
and let $C^{[x]} = (Z^{[x]}_t, R^{[x]}_t)_{t\ge 0}$. We call $C^{[x]}$ the \emph{random geodesic} from $0$ to $\hcal^x$.

\begin{lem}\label{lem:random_geodesic_is_geodesic}
For any $x>0$, $C^{[x]}\in\gcal^{x}(\Pi)$. In words, the random geodesic from $0$ to $\hcal^x$ is indeed a $\Pi$-geodesic from $0$ to $\hcal^x$.
\end{lem}

\begin{proof}
Since the random geodesic path consists of a path of vertices in $\tcal$, it is immediate that $C^{[x]}$ is a chain of reproduction events. Moreover, since it is a path in $\tcal(x)$, it must contain a vertex $u$ satisfying $\Gamma_u = \hit{\{0\}}{\hcal^x}$ and $X_u+R_u \ge x$. In other words, there exists $n$ such that $\Gamma_{\Theta^{(x)}_n} = \hit{\{0\}}{\hcal^x}$ and
\[\bcal\left(Z_{\Theta^{(x)}_n}, R_{\Theta^{(x)}_n}\right)\cap\hcal^x \neq \emptyset.\]
By the definition of $C^{[x]}$ we thus have
\[\bcal\left(Z^{[x]}_t, R^{[x]}_t\right)\cap\hcal^x \neq \emptyset \,\,\,\,\text{ for all } t\in\Big[\Gamma_{\Theta^{(x)}_n},\Gamma_{\Theta^{(x)}_{n+1}}\Big),\]
and in particular for $t=\Gamma_{\Theta^{(x)}_n} = \hit{\{0\}}{\hcal^x}$. We deduce that $C^{[x]}$ is a $\Pi$-geodesic from $0$ to $\hcal^x$ and the proof is complete. 
\end{proof}

The first two parts of Theorem \ref{thm:geodesics} state the \emph{existence} of a $\Pi$-geodesic with certain properties. It will therefore suffice to show that the random geodesic defined above has these properties. The last part of Theorem \ref{thm:geodesics} concerns the endpoint of all $\Pi$-geodesics from $0$ to the half-plane $\hcal^{x}$ at time $\hit{\{0\}}{\hcal^x}$; but all $\Pi$-geodesics from $0$ to $\hcal^{x}$ share the same endpoint at time $\hit{\{0\}}{\hcal^x}$, so again it will suffice to prove a statement about the endpoint of the random geodesic at time $\hit{\{0\}}{\hcal^x}$.

\subsection{\texorpdfstring{Control of the number of jumps before the random geodesic hits $\hcal^x$}{Control of the number of jumps before the random geodesic hits the half-plane}}

In order to control the fluctuations in the $y$-direction of the random geodesic from $0$ to $\hcal^x$, we need to control the number of jumps of $(\Theta_{n}^{(x)})_{n \geq 0}$. As a first step, we show the following technical lemma, which gives an upper bound on the probability that~$\tcal$ contains a path making more than~$\theta t$ jumps, for $\theta > 0$, before time~$t$. 

For $n \in \nmath$, let $\tcal_{n}$ be the sub-tree of~$\tcal$ containing only the~$n$ first generations. We recall the definition of $\text{M}_0$ from \eqref{cond:finite_rate}; $\text{M}_0$ is the maximum rate at which an individual vertex in the tree is affected by new reproduction events.

\begin{lem}\label{lem:moment_lemma_1}
For any~$t > 0$ and~$\theta > \text{M}_0$, 
\begin{equation*}
\proba\left(
\exists u \in \tcal \backslash \tcal_{\lceil \theta t \rceil - 1} : \Gamma_{u} \leq t
\right) \leq 2\exp\left(
- \theta t \left(
\log\left(\frac{\theta}{2\text{M}_0}\right) - 1 + \frac{\text{M}_0}{\theta}
\right)\right).
\end{equation*}
In particular, taking $\theta = 2e\text{M}_0$, we have that for all $t>0$,
\begin{equation*}
\proba\left(
\exists u \in \tcal \backslash \tcal_{\lceil 2e\text{M}_0 t \rceil -1} : \Gamma_{u} \leq t
\right) \leq 2e^{-\text{M}_0 t}. 
\end{equation*}
\end{lem}

\begin{proof}
Let $\theta, t > 0$. We have
\begin{align*}
\proba\left(
\exists u \in \tcal \backslash \tcal_{\lceil \theta t \rceil - 1}: \Gamma_{u} \leq t
\right) &= \proba\left(
\exists u \in \tcal, |u| = \lceil \theta t \rceil \text{ and } \Gamma_{u} \leq t
\right) \\
&\leq \sum_{|u| = \lceil \theta t \rceil} \proba\left(
\Gamma_{u} \leq t
\right) \\
&= \sum_{|u| = \lceil \theta t \rceil} \proba\left(
\sum_{v < u} T_{v} \leq t
\right). 
\end{align*}
For all $v < u$, $T_{v}$ is the first time at which $\bcal(Z_{v},R_{v})$ is affected by a reproduction event. By \eqref{cond:finite_rate}, $\sum_{v < u} T_{v}$ is stochastically bounded from below by a sum of $\lceil \theta t \rceil$ exponential random variables with parameter $\text{M}_0$. Let $(\ecal_{n})_{n \geq 1}$ be i.i.d. exponential random variables with parameter $\text{M}_0$, and let $\pcal_{t}$ be an independent Poisson random variable with parameter $t\text{M}_0$. Then, using the above observation, 
\[ \proba\left( \exists u \in \tcal \backslash \tcal_{\lceil \theta t \rceil - 1} : \Gamma_{u} \leq t \right) \leq 2^{\lceil \theta t \rceil} \proba\left( \sum_{n = 1}^{\lceil \theta t \rceil} \mathcal{E}_{n} \leq t \right) \leq 2^{\lceil \theta t \rceil} \proba\left( \pcal_{t} \geq \theta t \right).\]
By a standard Chernoff bound, using
\begin{equation*}
    \proba\left(
\pcal_{t} \geq \theta t
    \right) \leq \esp\left[ 
e^{\mu \pcal_{t}}
    \right] e^{-\mu \theta t}
\end{equation*}
with $\mu = \log(\theta / \text{M}_0)$, if $\theta > \text{M}_0$, then 
\begin{align*}
\proba\left(
\exists u \in \tcal \backslash \tcal_{\lceil \theta t \rceil - 1} : \Gamma_{u} \leq t
\right) &\leq 2 \exp\left(
\theta t \log(2) - \theta t \log\left(
\frac{\theta}{\text{M}_0}
\right)\right) \exp\left(
t \text{M}_0 \left(
\frac{\theta}{\text{M}_0} - A
\right)\right) \\
&= 2 \exp\left(
-\theta t \left(
\log\left(
\frac{\theta}{2\text{M}_0}
\right)-1 + \frac{\text{M}_0}{\theta}
\right)\right). 
\end{align*}
The last part of the lemma follows by taking $\theta = 2\text{M}_0 e$. 
\end{proof}

We can use Lemma~\ref{lem:moment_lemma_1} to obtain tail estimates for the number of jumps of the random geodesic before it hits $\hcal^x$. Recalling that $\Gamma_u$ is the birth time of particle $u$, we let $N^{(x)}_t$ be the number of jumps of the random geodesic from $0$ to $\hcal^x$ before time $t$, i.e.
\[N_{t}^{(x)} := \max\left\{ n \in \nmath : \Gamma_{\Theta_{n}^{(x)}} \leq t \right\},\]
and let $\text{N}(x)$ be the number of jumps before the random geodesic actually hits $\hcal^x$, i.e.
\[\text{N}(x) := \max\{n\in\mathbb{N} : \Gamma_{\Theta_n^{(x)}} \le \hit{}{\hcal^x} \} = N_{\hit{}{\hcal^x}}^{(x)},\]
where we have written $\hit{}{\hcal^x}$ instead of $\hit{\{0\}}{\hcal^x}$ since we will be starting our $\infty$-parent SLFV from $E=\{(0,0)\}$ throughout this section.

\begin{lem}\label{lem:tail_estimates_number_jumps}
There exists $c > 0$ such that for all~$x > 0$ and for all sufficiently large~$\theta$, 
\[\proba\left( \emph{\text{N}}(x) \geq \theta x \right) \leq 3\exp\left(-c\theta x\right). \]
\end{lem}

\begin{proof}
Let $x > 0$, and again recall $\text{M}_0 > 0$ from \eqref{cond:finite_rate}. Let $\theta > 0$, and let $\beta = \theta/(2e\text{M}_0)$. Then, by decomposing depending on whether or not~$\hit{}{\hcal^x}$ is larger than~$\beta x$, 
\begin{align*}
\proba\left(
\text{N}(x) \geq \theta x
\right) &\leq \proba\left(
\text{N}(x) \geq \theta x, \hit{}{\hcal^x} \leq \beta x
\right) + \proba\left(
\hit{}{\hcal^x} > \beta x
\right) \\
&\leq \proba\left(
N_{\beta x}^{(x)} \geq \theta x
\right) + \proba\left(
\hit{}{\hcal^x} > \beta x
\right) \\
&\leq \proba\left(
\exists u \in \tcal \backslash \tcal_{\lceil \theta x \rceil - 1} : \Gamma_{u} \leq \beta x
\right) + \proba\left(
\hit{}{\hcal^x} > \beta x
\right). 
\end{align*}
In order to control the first term, we can apply Lemma~\ref{lem:moment_lemma_1} with $t = \beta x$. We obtain
\begin{align*}
\proba\left(
\exists u \in \tcal \backslash \tcal_{\lceil \theta x \rceil - 1} : \Gamma_{u} \leq \beta x
\right) &= \proba\left(
\exists u \in \tcal \backslash \tcal_{\lceil t \theta / \beta \rceil - 1} : \Gamma_{u} \leq t
\right) \\
&= \proba\left(
\exists u \in \tcal \backslash \tcal_{\lceil 2e\text{M}_0 t \rceil - 1}: \Gamma_{u} \leq t
\right) \\
&\leq 2e^{-\text{M}_0 t} = 2e^{-\text{M}_0 \beta x} = 2e^{-\theta x/2e}. 
\end{align*}
We now apply Proposition \ref{prop:tail_tau_x} to bound the second term, $\proba( \hit{}{\hcal^x} > \beta x )$. Applying Proposition \ref{prop:tail_tau_x} with $z=(x,0)$ and $E=\{(0,0)\}$ gives that for some constant $c_\mu$ and all sufficiently large $\beta$,
\[\proba\left( \hit{}{x,0} > \beta x \right) \leq \exp\left( -c_{\mu}\beta x \right).\]
Since trivially $\hit{}{x,0} \ge \hit{}{\hcal^x}$ since the first is the hitting time of the point $(x,0)$ and the second is the hitting time of the (closed) half-plane to the right of $(x,0)$, we deduce that for $\beta$ large enough,
\begin{equation*}
\proba\left( \hit{}{\hcal^x} > \beta x \right) \leq \exp\left( -c_{\mu}\beta x \right) = \exp\left( -\frac{c_{\mu}}{2c_1e}\theta x\right).
\end{equation*}
This completes the proof. 
\end{proof}

We will later need to control the moments of the number of jumps of the random geodesic at time~$\hit{}{\hcal^x}$, in order to provide a lower bound on its fluctuations. We apply the tail bound above to give the following bound.

\begin{cor}\label{cor:moments_number_jumps} For any $k\in\mathbb{N}$, there exists a constant $c(k)\in(0,\infty)$ (depending only on $k$ and $\mu$, not $x$) such that for all $x > 0$, \[\E[\emph{\text{N}}(x)^k]\le c(k) x^k.\]
\end{cor}

\begin{proof} Let $x > 0$ and $k \in \nmath$. 
We simply use
\[\E[\text{N}(x)^k] = \sum_{n=1}^\infty n^{k-1}\P(\text{N}(x)\ge n) \le \int_0^\infty \alpha^{k-1}\P(\text{N}(x)\ge \alpha)\, d\alpha = x^k\int_0^\infty \theta^{k-1}\P(\text{N}(x)\ge \theta x)\, d\theta\]
and by Lemma \ref{lem:tail_estimates_number_jumps}, there exist constants $\theta_0$ and $c$ independent of~$x$ such that this is at most
\[x^k\int_0^{\theta_0} \theta^{k-1}\P(\text{N}(x)\ge \theta x)\, d\theta + x^k\int_{\theta_0}^\infty \theta^{k-1}e^{-c\theta x}\, d\theta \le c(k) x^k,\]
which completes the proof.
\end{proof}

\subsection{Upper bound on the transverse fluctuations of the random geodesic}\label{subsec:upper_bound_displacement}
The goal of this section is to provide an upper bound on how much the random geodesic from $0$ to $\hcal^x$ fluctuates in the $y$-direction. Recall that for a vertex $u\in\tcal$, we write $Z_u = (X_u,Y_u)$ for its position, and for each vertex $u$ except the root, we have $\Delta^X_{u} = X_{u}-X_{\overleftarrow u}$ and $\Delta^Y_{u} = Y_{u}-Y_{\overleftarrow u}$, the increments of $u$ relative to its parent in the $x$- and $y$-directions respectively.

Let $(B_i)_{i\ge 1}$ be a collection of i.i.d.~random variables, also independent of everything else, satisfying
\[\P(B_i=1)=\P(B_i=-1)=1/2\]
for each $i$. Note that for any $n$, the sequence $(Z_{\Theta_{j}^{(x)}})_{j\le n}$ of labels on the random geodesic path up to $\Theta_{n}^{(x)}$ has the same distribution as the sequence $(\tilde Z_{\Theta_{j}^{(x)}})_{j\le n}$ where for each $j$,
\[\tilde Z_{\Theta_{j}^{(x)}} = \left(X_{\Theta_{j}^{(x)}}, \tilde Y_{\Theta_{j}^{(x)}}\right) = \left(\sum_{i=1}^j \Delta^X_{\Theta_{i}^{(x)}}, \sum_{i=1}^j B_i \Delta^Y_{\Theta_{i}^{(x)}}\right).\]

Define $\fcal^\tcal$ to be the $\sigma$-algebra that knows everything about the tree $\tcal$, including all its labels, \emph{and} the information about the random geodesic path $(\Theta^{(x)}_j)_{j\ge 0}$, but not about the $B_i$, i.e.
\[\fcal^\tcal = \sigma(\{T_u,Z_u,R_u : u\in\mathcal T\}\cup\{\Theta_j^{(x)} : j\ge 0\}).\]
Then let $(\fcal_j)_{j\ge 0}$ be the filtration that includes $\fcal^\tcal$ plus the information about the $B_i$ for $i\le j$, i.e.
\[\fcal_j = \fcal^\tcal\vee \sigma\{B_i : i\le j\}.\]
It is then immediate that
\[\E\left[\left.\tilde Y_{\Theta_j^{(x)}}\right|\fcal_{j-1}\right] = \tilde Y_{\Theta_{j-1}^{(x)}}\]
and therefore $\big(\tilde Y_{\Theta_j^{(x)}}\big)_{j\ge 0}$ is a martingale with respect to the filtration $(\fcal_j)_{j\ge 0}$. This allows us to prove the following preliminary but key result.

\begin{lem}\label{lem:prelim_res_transverse_displ} For all $n \geq 1$, $x>0$ and $y>0$,
\begin{equation*}
\proba\left( \max_{j \leq n} \Big| Y_{\Theta_{j}^{(x)}} \Big| > y \right) \leq \frac{n \emph{\text{R}}_{0}^{2}}{y^{2}}. 
\end{equation*}
\end{lem}

\begin{proof}
Let $n \geq 1$ and $x,y>0$. As discussed above, $\tilde Y_{\Theta_j^{(x)}}$ is a martingale with respect to the filtration $\fcal_j$; and therefore $\Big|\tilde Y_{\Theta_j^{(x)}}\Big|^2$ is a submartingale with respect to the same filtration. Thus Doob's submartingale inequality says that
\[\P\left(\max_{j\le n} \Big|\tilde Y_{\Theta^{(x)}_j}\Big| > y\right) = \P\left(\max_{j\le n} \Big|\tilde Y_{\Theta^{(x)}_j}\Big|^2 > y^2\right) \le y^{-2}\,\E\left[\Big|\tilde Y_{\Theta^{(x)}_n}\Big|^2\right].\]
Now, by the independence and orthonormality of the $B_i$,
\[\E\left[\Big|\tilde Y_{\Theta^{(x)}_n}\Big|^2\right] = \E\left[\left(\sum_{i=1}^n B_{i}\Delta^Y_{\Theta^{(x)}_i}\right)^2\right] = \sum_{i,j=1}^n \E\Big[B_{i}B_{j}\Big]\E\left[\Delta^Y_{\Theta^{(x)}_i}\Delta^Y_{\Theta^{(x)}_i}\right] = \sum_{i=1}^n \E\left[\Big(\Delta^Y_{\Theta^{(x)}_i}\Big)^2\right] \le n\text{R}_{0}^{2}\]
almost surely, which allows us to conclude. 
\end{proof}

We can use this result to show that with high probability the random geodesic path from $0$ to $\hcal^x$ moves at most distance of order~$\sqrt{x}$ in the $y$-direction by time $\hit{\{0\}}{\hcal^x}$, in the following sense. Recall that $\lcal^{(y)}$ is the strip of radius $y$ about the $x$-axis.

\begin{lem}\label{lem:tight_upper_bound} For all~$\eps > 0$, there exists~$A_{\eps} > 0$ such that for all~$x > 1$, 
\begin{equation*}
\proba\left( \exists n \leq \emph{\text{N}}(x) : \bcal\left( Z_{\Theta_{n}^{(x)}}, R_{\Theta_{n}^{(x)}} \right) \nsubseteq \lcal^{(A_{\eps} \sqrt{x})} \right) \leq \eps. 
\end{equation*}
\end{lem}

We observe that Theorem~\ref{thm:geodesics}~(i) is then a direct consequence of this result. 

\begin{proof}
Let $A,\eps > 0$, and let $A',x > 1$. We have
\begin{align*}
\proba\left( \exists n \leq \text{N}(x) : \bcal\left( Z_{\Theta_{n}^{(x)}}, R_{\Theta_{n}^{(x)}} \right) \nsubseteq \lcal^{(A \sqrt{x})} \right) &\leq \proba\left(
\exists n \leq \lfloor A' x \rfloor : \bcal\left( Z_{\Theta_{n}^{(x)}}, R_{\Theta_{n}^{(x)}} \right) \nsubseteq \lcal^{(A\sqrt{x})} \right) \\
&\quad\quad + \proba\left( \text{N}(x) \geq \lfloor A' x \rfloor \right). 
\end{align*}
If $A'$ is sufficiently large, we can apply Lemma~\ref{lem:tail_estimates_number_jumps} to obtain
\[ \proba\left( \text{N}(x) \geq \lfloor A'x \rfloor \right) \leq \proba\left( \text{N}(x) \geq (A'-1)x \right) \leq e^{-c_{2}(A'-1)x} \leq e^{-c_{2}(A'-1)} \]
since~$x > 1$. We choose $A'$ large enough so that $e^{-c_{2}(A'-1)} < \eps/2$.

Then, since the radius of reproduction events is bounded from above by~$\text{R}_{0}$, by Lemma~\ref{lem:prelim_res_transverse_displ},
\begin{align*}
\proba\left(
\exists n \leq \lfloor A' x \rfloor : \bcal\left(
Z_{\Theta_{n}^{(x)}}, R_{\Theta_{n}^{(x)}}
\right) \nsubseteq \lcal^{(A\sqrt{x})}
\right) &\leq \proba\left(
\exists n \leq \lfloor A'x \rfloor : \left|
Y_{\Theta_{n}^{(x)}}
\right| > \left| 
A \sqrt{x} - \text{R}_{0}
\right|\right) \\
&= \proba\left(
\max_{n \leq \lfloor A'x \rfloor} \left| 
Y_{\Theta_{n}^{(x)}}
\right| > \left|
A\sqrt{x} - \text{R}_{0}
\right|
\right) \\
&\leq \frac{\text{R}_{0}A' x}{\left(
A\sqrt{x} - \text{R}_{0}
\right)^{2}}.
\end{align*}
We then choose $A$ large enough so that for all~$x > 1$, 
\begin{equation*}
    \frac{\text{R}_{0}A'x}{\left(
A\sqrt{x} - \text{R}_{0}
    \right)^{2}} < \eps/2, 
\end{equation*}
which allows us to conclude the proof. 
\end{proof}

We now apply Lemma~\ref{lem:tight_upper_bound} to complete the proof of Theorem~\ref{thm:geodesics}~(i). 

\begin{proof}[Proof of Theorem~\ref{thm:geodesics}~(i)] Let $\eps > 0$ and $x > 1$. Let~$A_{\eps} > 0$ be given by Lemma~\ref{lem:tight_upper_bound}. By Lemma~\ref{lem:random_geodesic_is_geodesic},
\begin{multline*}
\proba\left( \exists \mathbf{C}^{[G]} \in \gcal^{x}(\Pi) : \bcal\left( Z_{s}^{[G]}, R_{s}^{[G]} \right) \subseteq \lcal^{(A_{\eps}\sqrt{x})} \,\, \forall s \in \big[0, \tau^{\left( \mathcal{H}^{x}\right)}\big] \right)\\
\ge \proba\left( \bcal\left( Z_{s}^{[x]}, R_{s}^{[x]} \right) \subseteq \lcal^{(A_{\eps}\sqrt{x})} \,\, \forall s \in \big[0, \tau^{\left( \mathcal{H}^{x}\right)}\big] \right),
\end{multline*}
or in words, the existence of a $\Pi$-geodesic staying within $\lcal^{(A_{\eps}\sqrt{x})}$ is implied by the event that the random geodesic stays within $\lcal^{(A_{\eps}\sqrt{x})}$. But by Lemma \ref{lem:tight_upper_bound}, this occurs with probability at least $1-\eps$, allowing us to conclude.
\end{proof}

For Theorem \ref{thm:geodesics} (ii), we will also need bounds on the tail behaviour of the transverse fluctuations of the random geodesic, which we can obtain by very similar methods. These tail bounds will also be useful for showing almost sure convergence of hitting times.

\begin{lem}
\label{lem:large_devs_n} For any $\delta>0$, there exists $c>0$ such that for all $x>0$ and $n \geq 1$,
\begin{equation*}
\proba\left(
\max_{j \leq n} \left|
Y_{\Theta_{j}^{(x)}}
\right| > \delta n
\right) \leq e^{-cn}. 
\end{equation*}
Moreover, for any $\delta>0$ and $\beta\in(1/2,1)$, there exists $c>0$ such that for $n$ sufficiently large,
\begin{equation*}
\proba\left( \max_{j \leq n} \left| Y_{\Theta_{j}^{(x)}} \right| > \delta n^\beta \right) \leq e^{-cn^{2\beta-1}}.
\end{equation*}
\end{lem}

\begin{proof}
Fix $x>0$ and $n \geq 1$. We proceed as in the proof of Lemma \ref{lem:prelim_res_transverse_displ} above, noting that since $\tilde Y_{\Theta_j^{(x)}}$ is a martingale with respect to the filtration $\mathcal F_j$, we also have that $\exp\left(\lambda \tilde Y_{\Theta_j^{(x)}}\right)$ is a submartingale with respect to the same filtration for any $\lambda>0$. Thus Doob's submartingale inequality says that
\[\P\left(\max_{j\le n} \tilde Y_{\Theta^{(x)}_j} > y\right) = \P\left(\max_{j\le n} \exp\Big(\lambda\tilde Y_{\Theta^{(x)}_j}\Big) > e^{\lambda y}\right) \le e^{-\lambda y}\,\E\left[\exp\Big(\lambda \tilde Y_{\Theta^{(x)}_n}\Big)\right].\]
Now recalling that $\fcal^\tcal$ be the $\sigma$-algebra that knows everything about the tree, including all its labels, \emph{and} the information about the random geodesic path $(\Theta^{(x)}_j)_{j\ge 0}$, but not about the $B_i$, we have
\[\E\left[\exp\Big(\lambda \tilde Y_{\Theta^{(x)}_n}\Big)\right] = \E\left[\exp\left(\sum_{i=1}^n \lambda B_{i}\Delta^Y_{\Theta^{(x)}_i}\right)\right] = \E\left[\prod_{i=1}^n \E\left[\left. \exp\left(\lambda B_{i}\Delta^Y_{\Theta^{(x)}_i}\right)  \right|\fcal^\tcal\right]\right]\]
by independence of the $B_i$. Since for each $i$,
\[\E\left[\left. \exp\left(\lambda B_{i}\Delta^Y_{\Theta^{(x)}_i}\right)  \right|\fcal^\tcal\right] = \cosh\left( \lambda \Delta^Y_{\Theta^{(x)}_i}\right) \le \cosh\left(\lambda \text{R}_0\right)\]
almost surely, we obtain that
\[\P\left(\max_{j\le n} \tilde Y_{\Theta^{(x)}_j} > y\right) \le e^{-\lambda y} \cosh^n(\lambda\text{R}_0). \]
For the first part of the lemma, we let $y=\delta n$. As the derivative of $\lambda \mapsto \cosh(\lambda \text{R}_0) e^{-\lambda\delta/2}$ at~$0$ is equal to $-\delta/2 < 0$, we can choose
$\lambda = \lambda(\delta,\text{R}_0) >0$ small enough so that $\cosh(\lambda \text{R}_0) \le e^{\lambda\delta/2}$, which completes the proof of the first part of the lemma.

For the second part, we take $y=\delta n^\beta$ for $\beta\in(1/2,1)$. Then choosing the optimal
\[\lambda = \frac{1}{\text{R}_0}\tanh^{-1}\left(\frac{\delta n^{\beta-1}}{\text{R}_0}\right) = \frac{1}{2\text{R}_0}\log\left(\frac{1+\delta n^{\beta-1}/\text{R}_0}{1-\delta n^{\beta-1}/\text{R}_0}\right),\]
and using also that $\cosh(\tanh^{-1}x)= (1-x^2)^{-1/2}$, we obtain that
\[\P\left(\max_{j\le n} \tilde Y_{\Theta^{(x)}_j} > \delta n^\beta\right) \le \left(\frac{1-\delta n^{\beta-1}/\text{R}_0}{1+\delta n^{\beta-1}/\text{R}_0}\right)^{\delta n^\beta/2\text{R}_0} \cdot \left(1-\frac{\delta^2 n^{2\beta-2}}{\text{R}_0^2}\right)^{-n/2}. \]
Now using the approximations $(1-x)^{-1} \le e^{x+x^2}$ for $0\le x\le 1/2$ and $(1+x)^{-1}\le 1-x+x^2$ for $x\ge 0$, for $n$ sufficiently large the above becomes
\begin{align*}
\P\left(\max_{j\le n} \tilde Y_{\Theta^{(x)}_j} > \delta n^\beta\right) &\le \left(1-\frac{2\delta n^{\beta-1}}{\text{R}_0} + \frac{2\delta^2 n^{2\beta-2}}{\text{R}_0^2}\right)^{\delta n^\beta/2\text{R}_0} \cdot \exp\left(\frac{n}{2}\left(\frac{\delta^2 n^{2\beta-2}}{\text{R}_0^2}+ \frac{\delta^4 n^{4\beta-4}}{\text{R}_0^4}\right) \right)\\
&\le \exp\left(-\frac{\delta^2 n^{2\beta-1}}{2\text{R}_0^2} + \frac{\delta^3 n^{3\beta-2}}{\text{R}_0^3} + \frac{\delta^4 n^{4\beta-3}}{2\text{R}_0^4}\right)
\end{align*}
and since $\beta\in(1/2,1)$ this allows us to conclude.
\end{proof}

\begin{lem}\label{lem:large_devs_transverse_displ} For any~$\delta > 0$, there exists~$c > 0$ such that for all~$x > 2\text{R}_0/\delta$, 
\begin{equation*}
\proba\left( \exists n \leq \emph{\text{N}}(x) : \bcal\left( Z_{\Theta_{n}^{(x)}}, R_{\Theta_{n}^{(x)}} \right) \nsubseteq \lcal^{(\delta x)} \right) \leq 2e^{-cx}. 
\end{equation*}
Moreover, for any $\delta>0$ and $\beta\in(1/2,1)$, there exists $c>0$ such that for all $x$ sufficiently large,
\begin{equation*}
\proba\left( \exists n \leq \emph{\text{N}}(x) : \bcal\left( Z_{\Theta_{n}^{(x)}}, R_{\Theta_{n}^{(x)}} \right) \nsubseteq \lcal^{(\delta x^\beta)} \right) \leq \exp\left(-cx^{2\beta-1}\right).
\end{equation*}
\end{lem}

\begin{proof}
For any $x>1$, $a>0$ and $r>0$, we have
\begin{align*}
\proba\left( \exists n \leq \text{N}(x) : \bcal\left( Z_{\Theta_{n}^{(x)}}, R_{\Theta_{n}^{(x)}} \right) \nsubseteq \lcal^{(r)} \right) &\leq \proba\left(
\exists n \leq \lfloor a x \rfloor : \bcal\left( Z_{\Theta_{n}^{(x)}}, R_{\Theta_{n}^{(x)}} \right) \nsubseteq \lcal^{(r)} \right) \\
&\quad\quad + \proba\left( \text{N}(x) \geq \lfloor a x \rfloor \right). 
\end{align*}
By Lemma~\ref{lem:tail_estimates_number_jumps} we can fix $a$ sufficiently large and $c>0$ such that
\[ \proba\left( \text{N}(x) \geq \lfloor ax \rfloor \right) \leq e^{-cx}.\]
Then, since the radius of reproduction events is bounded from above by~$\text{R}_{0}$, provided $r>2\text{R}_0$,
\begin{align*}
\proba\left(
\exists n \leq \lfloor a x \rfloor : \bcal\left(
Z_{\Theta_{n}^{(x)}}, R_{\Theta_{n}}^{(x)}
\right) \nsubseteq \lcal^{(r)}
\right) &\leq \proba\left(
\exists n \leq \lfloor ax \rfloor : \left|
Y_{\Theta_{n}^{(x)}}
\right| > \left| 
r - \text{R}_{0}
\right|\right) \\
&\le \proba\left(
\max_{n \leq \lfloor ax \rfloor} \left| 
Y_{\Theta_{n}^{(x)}}
\right| > 
r/2 \right).
\end{align*}
We then apply Lemma \ref{lem:large_devs_n} to conclude, with $r=\delta x$ for the first part of the lemma and $r=\delta x^\beta$ for the second part.
\end{proof}

\begin{proof}[Proof of Theorem~\ref{thm:geodesics}~(ii)]
Fix $\delta>0$ and $\beta\in(1/2,1]$. Just as for the proof of part (i), we use that by Lemma~\ref{lem:random_geodesic_is_geodesic},
\begin{multline*}
\proba\left( \exists \mathbf{C}^{[G]} \in \gcal^{x}(\Pi) : \bcal\left( Z_{s}^{[G]}, R_{s}^{[G]} \right) \subseteq \lcal^{(\delta x^\beta)} \,\, \forall s \in \big[0, \tau^{\left( \mathcal{H}^{x}\right)}\big] \right)\\
\ge \proba\left( \bcal\left( Z_{s}^{[x]}, R_{s}^{[x]} \right) \subseteq \lcal^{(\delta x^\beta)} \,\, \forall s \in \big[0, \tau^{\left( \mathcal{H}^{x}\right)}\big] \right),
\end{multline*}
or in words, the existence of a $\Pi$-geodesic staying within $\lcal^{(\delta x^\beta)}$ is implied by the event that the random geodesic stays within $\lcal^{(\delta x^\beta)}$. Now we simply apply Lemma \ref{lem:large_devs_transverse_displ}, using the first part if $\beta=1$ and the second part if $\beta\in(1/2,1)$.
\end{proof}

\subsection{Lower bound on the transverse fluctuations of the random geodesic}\label{subsubsec:lower_bound_displ}

Our aim in this section is to prove the following lower bound on 
\begin{equation*}
\text{Y}(x) := Y_{\Theta_{\text{N}(x)}^{(x)}},
\end{equation*}
the $y$-displacement of the random geodesic from $0$ to $\hcal^x$ at the moment it hits $\hcal^x$.

\begin{prop}\label{prop:simple_transverse_lower_bound}
For any $\eps>0$, there exists $\delta>0$ such that for sufficiently large $x$,
\[\P\big(|\text{\emph{Y}}(x)| < \delta \sqrt x\big) \le \eps.\]
\end{prop}

Before proving Proposition \ref{prop:simple_transverse_lower_bound}, we show how it implies Theorem~\ref{thm:geodesics}~(iii).

\begin{proof}[Proof of Theorem~\ref{thm:geodesics}~(iii)]
As noted earlier, by \eqref{eqn:one_endpoint}, all geodesics have the same endpoint; thus we can bound all geodesic endpoints by that of the random geodesic, and in particular
\[\proba\left( \forall \mathbf{C}^{[G]} \in \gcal^{x}(\Pi),\,\, \bcal\left( Z_{\hit{\{0\}}{\mathcal{H}^{x}}}^{[G]}, R_{\hit{\{0\}}{\mathcal{H}^{x}}}^{[G]} \right) \nsubseteq \lcal^{(a\sqrt{x})} \right) \ge \P\big(|\text{\emph{Y}}(x)| \ge a \sqrt x\big). \]
But by Proposition \ref{prop:simple_transverse_lower_bound}, the latter probability can be made larger than $1-\eps$ by choosing $a$ sufficiently small.
\end{proof}

As in Section \ref{subsec:upper_bound_displacement}, let $(B_n)_{n\ge 0}$ be a sequence of i.i.d.~random variables, also independent of everything else, satisfying
\[\P(B_n=1)=\P(B_n=-1)=1/2.\]
We will again take advantage of the fact that the sequence $(Z_{\Theta_{j}^{(x)}})_{j\le n}$ of labels on the path up to $\Theta_{n}^{(x)}$ has the same distribution as the sequence $(\tilde Z_{\Theta_{j}^{(x)}})_{j\le n}$ where for each $j$,
\[\tilde Z_{\Theta_{j}^{(x)}} = \left(X_{\Theta_{j}^{(x)}}, \tilde Y_{\Theta_{j}^{(x)}}\right) = \left(\sum_{i=1}^j \Delta^X_{\Theta_{i}^{(x)}}, \sum_{i=1}^j B_i \Delta^Y_{\Theta_{i}^{(x)}}\right).\]
We will also write
\[\tilde{\text{Y}}(x) := \tilde Y_{\Theta_{\text{N}(x)}^{(x)}},.\]
We first apply the Paley-Zygmund inequality to bound $\tilde{\text{Y}}(x)$ in terms of the squared $y$-increments.

\begin{lem}\label{lem:simple_Ytilde_bound}
For any $x>0$ and $\eps>0$,
\[\P\left(\tilde{\text{\emph{Y}}}(x)^2 < \eps \sum_{i=1}^{\text{\emph{N}}(x)}\left(\Delta^Y_{\Theta^{(x)}_i}\right)^2 \right) \le 2\eps.\]
\end{lem}

\begin{proof}
The result holds trivially if $\eps\ge 1/2$, so fix $x>0$ and $\eps\in(0,1/2)$. Recall that $\fcal^\tcal$ is the $\sigma$-algebra that knows everything about the tree $\tcal$, including all its labels, \emph{and} the information about the random geodesic path, namely the chosen path $(\Theta^{(x)}_n)_{n\ge 0}$. Since the $B_i$ are i.i.d.~with zero mean and unit variance, and all other random variables are $\fcal^\tcal$-measurable,
\begin{equation}\label{eq:2momsimple}
\E\left[\left.\tilde{\text{Y}}(x)^2\,\right|\,\fcal^\tcal\right] = \E\left[\left.\left(\sum_{i = 1}^{\text{N}(x)} B_i\Delta^Y_{\Theta_i^{(x)}}\right)^2 \,\right|\,\fcal^\tcal\right] = \sum_{i = 1}^{\text{N}(x)} \left(\Delta^Y_{\Theta_i^{(x)}}\right)^2 
\end{equation}
and
\begin{equation}\label{eq:4momsimple}
\E\left[\left.\tilde{\text{Y}}(x)^4\,\right|\,\fcal^\tcal\right] = \E\left[\left.\left(\sum_{i = 1}^{\text{N}(x)} B_i\Delta^Y_{\Theta_i^{(x)}}\right)^4 \,\right|\,\fcal^\tcal\right] = \sum_{i = 1}^{\text{N}(x)}\sum_{j = 1}^{\text{N}(x)} \left(\Delta^Y_{\Theta_i^{(x)}}\right)^2\left(\Delta^Y_{\Theta_j^{(x)}}\right)^2 = \E\left[\left.\tilde{\text{Y}}(x)^2\,\right|\,\fcal^\tcal\right]^2.
\end{equation}
By the Paley-Zygmund inequality and \eqref{eq:4momsimple},
\[\P\left( \left. \tilde{\text{Y}}(x)^2 \ge \eps \E\left[\left.\tilde{\text{Y}}(x)^2\,\right|\,\fcal^\tcal\right]\, \right|\, \fcal^\tcal \right) \ge (1-\eps)^2 \frac{\E\left[\left.\tilde{\text{Y}}(x)^2\,\right|\,\fcal^\tcal\right]^2}{\E\left[\left.\tilde{\text{Y}}(x)^4\,\right|\,\fcal^\tcal\right]} = (1-\eps)^2.\]
Thus, using \eqref{eq:2momsimple},
\[\P\left( \left. \tilde Y(x)^2 < \eps \sum_{i=1}^{\text{\emph{N}}(x)}\left(\Delta^Y_{\Theta^{(x)}_i}\right)^2 \,\right|\,\fcal^\tcal \right) \le 1-(1-\eps)^2 \le 2\eps.\]
Taking expectations gives the result.
\end{proof}

It remains to show that the sum of the squares of the transverse increments along the random geodesic path are unlikely to be small. Clearly, one way in which this could occur would be if many of the non-zero displacements in the $y$-direction were all extremely small in magnitude. Since our random geodesic path is one that hits $\hcal^x$ quickly, its displacements in the $x$-direction are in general large, which could \emph{a priori} force its displacements in the $y$-direction to be small due to the shape of reproduction events. Given this dependence between the parallel and transverse fluctuations on the random geodesic path, our approach is to show instead that \emph{nowhere} in the tree is there a path with many extremely small but non-zero $y$-displacements. First we show that any single vertex in $\tcal$ is unlikely to move a very small non-zero distance in the $y$-direction. We recall that if a vertex is the first child of its parent, then it inherits its parent's position; therefore only vertices $u$ that satisfy $u=\overleftarrow u \cdot 2$ have $\Delta^Y_u \neq 0$.

\begin{lem}\label{lem:prob_y_small}
If $u=\overleftarrow{u} \cdot 2$, then for any $y\in(0,1)$
\[\sup_{\rho>0}\P\left(\left.|\Delta^Y_u|<y\, \right|\, R_{\overleftarrow u} = \rho\right) \le p_y \]
where
\[p_y = \frac{4\mu((0,y^{1/2}))}{1\wedge\int_0^{\text{\emph{R}}_{0}}u^2\mu(du)} + \frac{\frac{4}{\pi}(1+\text{\emph{R}}_{0}^2)y^{1/2}}{1\wedge\int_0^{\text{\emph{R}}_{0}} u^2\mu(du)}\to 0\]
as $y\downarrow 0$.
\end{lem}

\begin{proof}
Take $u\in\tcal$ with $u=\overleftarrow{u} \cdot 2$. It is easy to check from the definition in Section \ref{subsec:tree_representation} that
\[\P(R_u\in dr \,|\, R_{\overleftarrow u} = \rho) = \frac{(\rho + r)^2}{\int_0^{\text{R}_{0}} (\rho + u)^2 \mu(du)},\]
and that given $R_u=r$, the displacement $(\Delta^X_u,\Delta^Y_u)$ is chosen uniformly over $\bcal(0,r)$. Thus
\[\P(|\Delta^Y_u|<y \, | \, R_u = r) = \frac{\int_{(-y,y)\times(-r,r)} dz}{\pi r^2} \le \frac{4y}{\pi r}.\]
Combining these facts, we obtain that for $y\in(0,1)$,
\begin{align}
\P(|\Delta^Y_u|<y \,|\,R_{\overleftarrow u}=\rho) &\le \P(|R_u|<y^{1/2}\,|\,R_{\overleftarrow u}=\rho) + \P( |R_u|\ge y^{1/2},\,|\Delta^Y_u|<y \,|\, R_{\overleftarrow u}=\rho)\nonumber\\
&\le \frac{\int_0^{y^{1/2}}(\rho+r)^2\mu(dr)}{\int_0^{\text{R}_{0}}(\rho+u)^2\mu(du)} + \frac{\int_{y^{1/2}}^{\text{R}_{0}} \frac{4y}{\pi r}(\rho+r)^2 \mu(dr)}{\int_0^{\text{R}_{0}}(\rho+u)^2 \mu(du)}.\label{eq:Y_upper_bound_given_rho}
\end{align}
If $\rho\in(0,1)$, then \eqref{eq:Y_upper_bound_given_rho} is at most
\[\frac{\int_0^{y^{1/2}}(1+r)^2\mu(dr)}{\int_0^{\text{R}_{0}}u^2\mu(du)} + \frac{\int_{y^{1/2}}^{\text{R}_{0}} \frac{4y}{\pi r}(1+r)^2 \mu(dr)}{\int_0^{\text{R}_{0}}u^2 \mu(du)} \le \frac{4\mu((0,y^{1/2}))}{\int_0^{\text{R}_{0}}u^2\mu(du)} + \frac{4(1+\text{R}_{0}^2)y^{1/2}}{\pi\int_0^{\text{R}_{0}} u^2\mu(du)}\]
and if $\rho\ge 1$, then \eqref{eq:Y_upper_bound_given_rho} is at most
\[\frac{\int_0^{y^{1/2}}(1+r/\rho)^2\mu(dr)}{\int_0^{\text{R}_{0}}(1+u/\rho)^2\mu(du)} + \frac{\int_{y^{1/2}}^{\text{R}_{0}} \frac{4y}{\pi r}(1+r/\rho)^2 \mu(dr)}{\int_0^{\text{R}_{0}}(1+u/\rho)^2 \mu(du)} \le 4\mu((0,y^{1/2})) + \frac{4}{\pi}(1+\text{R}_{0}^2)y^{1/2}.\]
Thus, for any $\rho>0$,
\[\P(|\Delta^Y_u|<y \,|\,R_{\overleftarrow u}=\rho) \le \frac{4\mu((0,y^{1/2}))}{1\wedge\int_0^{\text{R}_{0}}u^2\mu(du)} + \frac{\frac{4}{\pi}(1+\text{R}_{0}^2)y^{1/2}}{1\wedge\int_0^{\text{R}_{0}} u^2\mu(du)}\to 0\]
as $y\downarrow 0$, as required.
\end{proof}

Write $N_2(v)=\#\{u\le v : u = \overleftarrow{u}\cdot 2\}$ for the number of ancestors of $v$ that are the second child of their parent. Recall that we are interested in vertices that have travelled distance $x$ in the $x$-direction, and which therefore satisfy $N_2(v)\ge ax$ where $a=1/2\text{R}_{0}$. 
We recall that~$\tcal_{n}$, $n \in \nmath$ is the sub-tree of~$\tcal$ containing only the~$n$ first generations. 

\begin{lem}\label{lem:no_long_bad_paths}
For any $a,x>0$, $n\in\mathbb{N}$, and any sufficiently small $y>0$, we have
\[\P\left(\exists v\in\tcal_n : N_2(v)\ge a x\, \text{ and } \,\#\{u\le v : |\Delta^Y_u|\ge y\} < \frac{a x}{2}\right)\\
\le 2^{n+1} \exp\left(\frac{a x}{2}\left(1-2p_y +\log(2p_y)\right)\right).\]
\end{lem}

\begin{proof}
For $u\in\tcal$, let $\fcal_{\overleftarrow u}$ be the $\sigma$-algebra consisting of all information about the strict ancestors of $u$, i.e.
\[\fcal_{\overleftarrow u} = \sigma(\{T_w,Z_w,R_w : w < u\}).\]
Note that the value of $\Delta^Y_u$ depends on $\fcal_{\overleftarrow u}$ only through $R_{\overleftarrow u}$. Thus if $u=\overleftarrow{u} \cdot 2$ and $y\in(0,1)$, by Lemma \ref{lem:prob_y_small} we have
\[\P\left(\left.|\Delta^Y_u|<y\, \right|\, \fcal_{\overleftarrow u}\right) \le p_y \]
almost surely.

Now, fix $v\in\tcal$ and suppose that $N_2(v) = k$. Label the ancestors of $v$ that satisfy $u = \overleftarrow{u}\cdot 2$ as $u_1,u_2,\ldots,u_k$. Then for any $y\in(0,1)$ and $\lambda>0$,
\begin{align}
\P\left(\#\{u\le v : |\Delta^Y_u|\ge y\} < k/2\right) &\le \P\left(\#\{u\le v : u=\overleftarrow{u}\cdot 2 \text{ and } |\Delta^Y_u|< y\} \ge k/2\right)\nonumber\\
&\le \E\left[e^{\lambda \#\{u\le v\, :\, u=\overleftarrow{u}\cdot 2 \text{ and } |\Delta^Y_u|< y\}}\right] e^{-\lambda k/2}\nonumber\\
&= \E\Bigg[\prod_{i=1}^k \exp(\lambda \ind_{\{|\Delta^Y_{u_i}|<y\}})\Bigg]e^{-\lambda k/2}.\label{eq:Y_chernoff_step1}
\end{align}
Conditioning first on $\fcal_{\overleftarrow u_k}$, and applying Lemma \ref{lem:prob_y_small}, we see that
\begin{align*}
\E\Bigg[\prod_{i=1}^k \exp(\lambda \ind_{\{|\Delta^Y_{u_i}|<y\}})\Bigg] &\le \E\Bigg[\E\left[\left.\exp(\lambda \ind_{\{|\Delta^Y_{u_k}|<y\}})\right|\fcal_{\overleftarrow{u}_k}\right]\prod_{i=1}^{k-1} \exp(\lambda \ind_{\{|\Delta^Y_{u_i}|<y\}})\Bigg]\\
&\le (e^\lambda p_y + 1-p_y)\E\Bigg[\prod_{i=1}^{k-1} \exp(\lambda \ind_{\{|\Delta^Y_{u_i}|<y\}})\Bigg].
\end{align*}
By recursion, we obtain that
\[\E\Bigg[\prod_{i=1}^k \exp(\lambda \ind_{\{|\Delta^Y_{u_i}|<y\}})\Bigg] \le (e^\lambda p_y + 1-p_y)^k.\]
Substituting this into \eqref{eq:Y_chernoff_step1} and choosing the optimal $\lambda = \log(1/2p_y)$ (which is positive provided $y$ is small enough as $p_{y} \to 0$ when $y \to 0$), we obtain that for any $v\in\tcal$ satisfying $\#\{u\le v : u=\overleftarrow{u}\cdot 2\} = k$, we have
\begin{equation}\label{eq:not_many_big_y}
\P\left(\#\{u\le v : |\Delta^Y_u|\ge y\} < k/2\right) \le \exp\left(\frac{k}{2}\left(1-2p_y +\log(2p_y)\right)\right).
\end{equation}
Thus for any $a>0$, $x>0$ and $n\in\mathbb{N}$, and sufficiently small $y>0$,
\begin{align*}
&\P\Big(\exists v\in\tcal_n : N_2(v)\ge a x \text{ and } \#\{u\le v : |\Delta^Y_u|\ge y\} < a x/2\Big)\\
&\le \E\Big[\#\big\{v\in\tcal_n : N_2(v)\ge a x \text{ and } \#\{u\le v : |\Delta^Y_u|\ge y\} < a x/2\big\}\Big]\\
& = \sum_{k=\lceil a x\rceil}^n \sum_{v\in\tcal_n : N_2(v)=k} \P\Big(\#\{u\le v : |\Delta^Y_u|\ge y\} < a x/2\Big)\\
& \le \sum_{k=\lceil a x\rceil}^n \sum_{v\in\tcal_n : N_2(v)=k} \P\Big(\#\{u\le v : |\Delta^Y_u|\ge y\} < k/2\Big)\\
& \le \sum_{k=\lceil a x\rceil}^n \sum_{v\in\tcal_n : N_2(v)=k} \exp\left(\frac{k}{2}\left(1-2p_y +\log(2p_y)\right)\right),
\end{align*}
where the last inequality uses \eqref{eq:not_many_big_y}. Since the exponent above is negative for any $y$, we can bound the exponential term above by its value when $k= a x$, so that
\begin{multline*}
\P\left(\exists v\in\tcal_n : N_2(v)\ge a x\, \text{ and } \,\#\{u\le v : |\Delta^Y_u|\ge y\} < a x/2\right)\\
\le \sum_{k=\lceil a x\rceil}^n \sum_{v\in\tcal_n : N_2(v)=k} \exp\left(\frac{ax}{2}\left(1-2p_y +\log(2p_y)\right)\right),
\end{multline*}
and since there are $2^{n+1}-1$ vertices in total in $\tcal_n$, the result follows.
\end{proof}

\begin{lem}\label{lem:simple_square_lower_bound}
There exist constants $\alpha,\beta,\gamma\in(0,\infty)$ (depending only on $\mu$) such that for all $x>0$,
\[\P\left(\sum_{i=1}^{\text{\emph{N}}(x)}\left(\Delta^Y_{\Theta^{(x)}_i}\right)^2 < \alpha x\right) \le \beta e^{-\gamma x}.\]
\end{lem}

\begin{proof}
By Lemma \ref{lem:tail_estimates_number_jumps}, we can fix $c>0$ and $\theta>0$ such that for all $x>0$,
\[\P(\text{N}(x) \ge \theta x) \le 3\exp(-c\theta x).\]
Recall the quantity $p_y$ from Lemma \ref{lem:prob_y_small}, and recall that $p_y\to 0$ as $y\to 0$. We can therefore fix $y>0$ such that
\begin{equation}\label{eq:setting_y}
2^\theta \exp\left(\frac{1}{4\text{R}_0}(1-2p_y+\log(2p_y))\right) \le \exp(-1).
\end{equation}
Applying Lemma \ref{lem:no_long_bad_paths} with $a=1/2\text{R}_0$ and $n=\lceil\theta x\rceil$, we have
\begin{align*}
&\P\left(\exists v\in\tcal_{\lceil\theta x\rceil} : N_2(v)\ge \frac{x}{2\text{R}_0}\, \text{ and } \,\#\{u\le v : |\Delta^Y_u|\ge y\} < \frac{x}{4\text{R}_0}\right)\\
&\le 2^{\lceil\theta x\rceil + 1} \exp\left(\frac{x}{4\text{R}_0}(1-2p_y+\log(2p_y))\right)\\
&\le 4\exp(-x)
\end{align*}
by \eqref{eq:setting_y}. Define the event
\[\Lambda(x) = \big\{\text{N}(x)\ge \theta x\big\} \cup \left\{\exists v\in\tcal_{\lceil\theta x\rceil} : N_2(v)\ge \frac{x}{2\text{R}_0}\, \text{ and } \,\#\left\{u\le v : |\Delta^Y_u|\ge y\right\} < \frac{x}{4\text{R}_0}\right\}.\]
This is our ``bad event'', which we have shown above has exponentially small probability in $x$.

We now note that, since by definition $X_{\Theta^{(x)}_{\text{N}(x)}} \ge x$, and each reproduction event can move us at most $2\text{R}_0$ to the right, we necessarily have $N_2(\Theta^{(x)}_{\text{N}(x)}) \ge x/2\text{R}_0$. Thus, on $\Lambda(x)^c$, we have
\[\#\left\{i\le \text{N}(x) : \Big|\Delta^Y_{\Theta^{(x)}_i}\Big|\ge y\right\} \ge \frac{x}{4\text{R}_0},\]
and thus
\[\sum_{i=1}^{\text{\emph{N}}(x)}\left(\Delta^Y_{\Theta^{(x)}_i}\right)^2 \ge \frac{xy^2}{4\text{R}_0}.\]
Choosing $\alpha = y^2/4\text{R}_0$, $\beta = 7$ and $\gamma = c\theta \wedge 1$, we deduce the result.
\end{proof}

We can now combine Lemmas \ref{lem:simple_Ytilde_bound} and \ref{lem:simple_square_lower_bound} to prove our main result for this section.

\begin{proof}[Proof of Proposition \ref{prop:simple_transverse_lower_bound}]
As observed earlier, $(\tilde Y_{\Theta^{(x)}_i})_{i\ge 0}$ has the same distribution as $(Y_{\Theta^{(x)}_i})_{i\ge 0}$. Thus
\[\P\big(|\text{Y}(x)| < \delta \sqrt x\big) = \P\left(|\tilde{\text{Y}}(x)| < \delta \sqrt x\right).\]
Fix $\alpha,\beta,\gamma\in(0,\infty)$ from Lemma \ref{lem:simple_square_lower_bound}. Splitting depending on whether the event in Lemma \ref{lem:simple_square_lower_bound} occurs or not,
\begin{align*}
\P\left(|\tilde{\text{Y}}(x)| < \delta \sqrt x\right) &\le \P\left(\sum_{i=1}^{\text{\emph{N}}(x)}\left(\Delta^Y_{\Theta^{(x)}_i}\right)^2 < \alpha x\right) + \P\left(\sum_{i=1}^{\text{\emph{N}}(x)}\left(\Delta^Y_{\Theta^{(x)}_i}\right)^2 \ge \alpha x \,\, \text{ and } \,\, \tilde{\text{Y}}(x)^2 < \delta^2 x\right)\\
&\le \beta e^{-\gamma x} + \P\left(\tilde{\text{Y}}(x)^2 < \frac{\delta^2}{\alpha} \sum_{i=1}^{\text{\emph{N}}(x)}\left(\Delta^Y_{\Theta^{(x)}_i}\right)^2 \right).
\end{align*}
By Lemma \ref{lem:simple_Ytilde_bound}, the last probability on the right-hand side above is at most $2\delta^2/\alpha$. Thus the result follows by choosing $\delta=\sqrt{\alpha\eps}/2$ and $x$ large enough that $\beta e^{-\gamma x} \le \eps/2$.
\end{proof}

\subsection{Geodesics from plane to point: precise statement and proof of Theorem \ref{thm:geodesic_plane_to_point}}\label{sec:geodesics_plane_to_point}

It is important to recall at this point that it is only the self-duality of the $\infty$-parent SLFV process started from a point (see e.g.~Corollary \ref{cor:forwards_St}) that allows us to characterise $\hit{\{0\}}{\hcal^x}$ and the corresponding geodesics in Theorem \ref{thm:geodesics} in terms of chains of $\Pi$-reproduction events (forwards in time from $0$ to $\hcal^x$). Since, in Theorem \ref{thm:geodesic_plane_to_point}, we are instead concerned with the $\infty$-parent SLFV started from $\hcal$, we must use directly the definition of the process (given in Definition \ref{defn:infty_parent_slfv}), which says that $\hit{\hcal}{x,0}\le t$ (by which we mean $(x,0)\in S^\hcal_t$) if and only if there is a chain of $\Pi^{(t)}$-reproduction events starting from $(x,0)$ that intersects $\hcal$ by time $t$.

This makes things significantly more complicated; we cannot simply work with $\Pi^{(t)}$ for fixed $t$, argue that this is equal in distribution to $\Pi$ and therefore the same results hold; we must consider a range of values of $t$ in order to be sure that we have covered $\hit{\hcal}{x,0}$. And of course, $\hit{\hcal}{x,0}$ is nothing like a typical time $t$. Nonetheless, we \emph{will} attempt to use our existing results and replicate the strategy backwards in time, even though there will be subtle complications in transferring these results to gain results directly about geodesics from $\hcal$ to $z$.

First we formally define geodesics from a set to a point, and give a precise statement of Theorem \ref{thm:geodesic_plane_to_point}.

\begin{definition}\label{defn:geodesics_set_point}
For all $z\in\rtwo$ and $A\subseteq\rtwo$, a $\Pi$-\textbf{geodesic} from~$A$ to~$z$ is a chain of $\Pi^{(\hit{A}{z})}$-reproduction events
\[\mathbf{C}^{[G]} = \left( Z_{t}^{[G]}, R_{t}^{[G]} \right)_{t\ge 0} \in \ccal^{(z)}(\Pi^{(\hit{A}{z})})\]
such that
\[\bcal\left( Z_{\hit{\{z\}}{A}}^{[G]}, R_{\hit{\{z\}}{A}}^{[G]} \right) \cap A \neq \emptyset.\]
For any measurable $A\subset\rtwo$ and any $z\in\rtwo$, let $\gcal^{A\to z}(\Pi)$ be the set of all $\Pi$-geodesics from~$A$ to~$z$.
\end{definition}

Recall that for $L > 0$, the strip of radius $L$ about the $x$-axis is denoted
\[ \lcal^{(L)} =\left\{ (x,y) \in \rtwo : |y| \leq L \right\}. \]
We can now state the precise version of Theorem~\ref{thm:geodesic_plane_to_point}.

\theoremstyle{plain}
\newtheorem*{geodesics_plane_to_point}{Theorem~\ref{thm:geodesic_plane_to_point} (precise statement)}
\begin{geodesics_plane_to_point}
For all $\eps > 0$, there exist $A_{\eps},x_\eps \in(0,\infty)$ such that for all $x \ge x_\eps$, 
\begin{equation*}
\proba\left( \exists \mathbf{C}^{[G]} \in \gcal^{\hcal\to(0,x)}(\Pi) : \bcal\left( Z_{s}^{[G]}, R_{s}^{[G]} \right) \subseteq \lcal^{(A_{\eps}\sqrt{x})} \,\, \forall s \in \left[ 0,\hit{\hcal}{x,0} \right]\right) \geq 1-\eps.
\end{equation*}
\end{geodesics_plane_to_point}

As for Theorem \ref{thm:geodesics}, we expect that in fact all geodesics from $\hcal$ to $(0,x)$ have fluctuations of order $\sqrt x$, but we do not currently have a bound on how
many such geodesics there are (in general there will be more than one) and therefore do not rule out that there may be ``unusual'' geodesics that fluctuate on a larger scale.

We now proceed with the proof of Theorem \ref{thm:geodesic_plane_to_point}. As previously mentioned, this will be more intricate than the proof of Theorem \ref{thm:geodesics}(i), but we will make use of several of the tools and results already developed.

We recall from Section \ref{subsec:tree_representation} that $\tcal^{[z]}(\Pi^{(t)})$ is the tree representation of $\Pi^{(t)}$-chains of reproduction events started from $z\in\rtwo$. We define, for $t\in\mathbb{R}$ and $z\in\rtwo\setminus\hcal$,
\[\hit{[z,t]}{\hcal} :=\inf\left\{s\ge 0 : \exists C = (Z_u,R_u)_{u\ge 0} \in \ccal^{(z)}(\Pi^{(t)}) \text{ with } \bcal(Z_s,R_s)\cap\hcal \neq \emptyset \right\},\]
i.e.~the first time $s$ such that there exists a chain of $\Pi^{(t)}$-reproduction events started from $z$ that intersects $\hcal$ by time $s$. Contrast this with
\[\hit{\hcal}{z} = \inf\left\{s\ge 0 : \exists C = (Z_u,R_u)_{u\ge 0} \in \ccal^{(z)}(\Pi^{(s)}) \text{ with } \bcal(Z_s,R_s)\cap\hcal \neq \emptyset \right\},\]
i.e.~the first time $s$ such that there exists a chain of $\Pi^{(s)}$-reproduction events started from $z$ that intersects $\hcal$ by time $s$. Much of the added complexity of starting the $\infty$-parent SLFV from the half-plane, rather than a singleton or other compact set, comes from this difference. We note, however, that if $\hit{\hcal}{z}\le t$, then $\hit{[z,t]}{\hcal}\le t$.

If $z=(x,y)\in\rtwo$, then we also let $(\Theta^{[z,t]}_n)_{n\ge 0}$ be the exact analogue of the random geodesic path $(\Theta^{(x)}_n)_{n\ge 0}$ defined in Section \ref{subsec:random_geodesic}, but for the tree $\tcal^{[z]}(\Pi^{(t)})$ rather than $\tcal^{[0]}(\Pi)$. To be more precise, we let $\tcal^{[z,t]}$ be the subtree of $\tcal^{[z]}(\Pi^{(t)})$ constructed by keeping only paths which are in $\hcal$ at all times $t\ge \hit{[z,t]}{\hcal}$. We then set $\Theta^{[z,t]}_0 = \emptyset$, i.e.~we start from the root of the tree, and recursively define $\Theta^{[z,t]}_{n+1}$ as follows:
\begin{itemize}
\item If $\Theta_{n}^{[z,t]} \cdot 1 \notin \tcal^{[z,t]}$ (resp. $\Theta_{n}^{[z,t]} \cdot 2 \notin \tcal^{[z,t]}$), then $\Theta_{n+1}^{[z,t]} = \Theta_{n}^{[z,t]}\cdot 2$ (resp. $\Theta_{n}^{[z,t]} \cdot 1$). 
\item Otherwise, both $\Theta_{n}^{[z,t]} \cdot 1$ and $\Theta_{n}^{[z,t]} \cdot 2$ are in $\tcal^{[z,t]}$, and $\Theta_{n+1}^{[z,t]} = \Theta_{n}^{[z,t]} \cdot 1$ (resp. $\Theta_{n}^{[z,t]} \cdot 2$) with probability $1/2$. 
\end{itemize}
To be clear, now when we write e.g.~$Y_{\Theta_{n}^{[z,t]}}$, we mean the $y$-displacement of the vertex $\Theta_{n}^{[z,t]}$ \emph{in the tree} $\tcal^{[z]}(\Pi^{(t)})$, rather than in the tree $\tcal=\tcal^{[0]}(\Pi)$. Recall that $\fcal^\tcal$ is the $\sigma$-algebra that knows everything about the tree $\tcal$, including all its labels, \emph{and} the information about  $(\Theta^{(x)}_j)_{j\ge 0}$; similarly let $\fcal^{[z,t]}$ be the $\sigma$-algebra that knows everything about $\tcal^{[z]}(\Pi^{(t)})$ and the path $(\Theta^{[z,t]}_n)_{n\ge 0}$. We now give a lemma similar to Lemma \ref{lem:prelim_res_transverse_displ}.

\begin{lem}\label{lem:backwards_transverse_displ} Fix $t\in\mathbb{R}$ and $z=(x,0)$ for $x>0$. For all $n \geq 1$, $y>0$ and $0\le a<b<\infty$,
\begin{equation*}
\proba\left( \left\{\max_{j \leq n} \Big| Y_{\Theta_{j}^{[z,t]}} \Big| > y\right\}\cap \left\{\hit{\hcal}{z}\in(a,b]\right\} \right) \leq \frac{n \emph{\text{R}}_{0}^{2}}{y^{2}}\P\left(\hit{\hcal}{z}\in(a,b]\right). 
\end{equation*}
\end{lem}

\begin{proof}
Fix $n\ge1$, $y>0$ and $0\le a<b<\infty$. As in Section \ref{subsec:upper_bound_displacement}, we let $(B_i)_{i\ge 1}$ be a collection of i.i.d.~random variables, also independent of everything else, satisfying
\[\P(B_i=1)=\P(B_i=-1)=1/2\]
for each $i$. Define the filtration
\[\fcal^{[z,t]}_j = \fcal^\tcal \vee \fcal^{[z,t]} \vee \sigma\{B_i : i\le j\}.\]
Set $\tilde Y_0=0$ and, for $j\ge 1$,
\[\tilde Y_j = \sum_{i=1}^j B_i \Delta^Y_{\Theta^{[z,t]}_i}.\]
Since $\hit{\hcal}{z}$ is independent of the sign of the $y$-displacements of any individual path in the tree, 
\[\proba\left( \left\{\max_{j \leq n} \Big| Y_{\Theta_{j}^{[z,t]}} \Big| > y\right\}\cap \left\{\hit{\hcal}{z}\in(a,b]\right\} \right) = \proba\left( \left\{\max_{j \leq n} \Big| \tilde Y_j \Big| > y\right\}\cap \left\{\hit{\hcal}{z}\in(a,b]\right\} \right).\]
It is also clear that $(\tilde Y_j \ind_{\{\hit{\hcal}{z}\in(a,b]\}})_{j\ge 0}$ is a martingale with respect to the filtration $\fcal_j$; and therefore $\big(\big|\tilde Y_j\big|^2 \ind_{\{\hit{\hcal}{z}\in(a,b]\}}\big)_{j\ge 0}$ is a submartingale with respect to the same filtration. Thus Doob's submartingale inequality says that
\begin{align*}
\proba\left( \left\{\max_{j \leq n} \Big| \tilde Y_j \Big| > y\right\}\cap \left\{\hit{\hcal}{z}\in(a,b]\right\} \right) &= \P\left(\max_{j\le n} \Big|\tilde Y_j\Big|^2 \ind_{\{\hit{\hcal}{z}\in(a,b]\}} > y^2\right)\\
&\le y^{-2}\,\E\left[\Big|\tilde Y_n\Big|^2 \ind_{\{\hit{\hcal}{z}\in(a,b]\}}\right].
\end{align*}
Now, by the independence and orthonormality of the $B_i$,
\begin{align*}
\E\left[\Big|\tilde Y_n\Big|^2 \ind_{\{\hit{\hcal}{z}\in(a,b]\}}\right] &= \E\left[\left(\sum_{i=1}^n B_{i}\Delta^Y_{\Theta^{(x)}_i}\right)^2\ind_{\{\hit{\hcal}{z}\in(a,b]\}}\right]\\
&= \sum_{i,j=1}^n \E\Big[B_{i}B_{j}\Big]\E\left[\Delta^Y_{\Theta^{(x)}_i}\Delta^Y_{\Theta^{(x)}_i}\ind_{\{\hit{\hcal}{z}\in(a,b]\}}\right]\\
&= \sum_{i=1}^n \E\left[\Big(\Delta^Y_{\Theta^{(x)}_i}\Big)^2\ind_{\{\hit{\hcal}{z}\in(a,b]\}}\right]\\
&\le n\text{R}_{0}^{2}\P\left(\hit{\hcal}{z}\in(a,b]\right)
\end{align*}
almost surely, which allows us to conclude. 
\end{proof}

For $\delta\in(0,1]$, let $\tau^\hcal_\delta(z) = \delta\lceil \hit{\hcal}{z}/\delta\rceil$, i.e.~we round $\hit{\hcal}{z}$ up to the next element of $\delta\mathbb{N}$. Our next proposition \emph{almost} guarantees that the geodesic from $\hcal$ to $z=(x,0)$ stays within a strip of radius $A\sqrt{x}$ for sufficiently large $A$ with high probability; but for now, we have to make do with a chain of reproduction events that starts \emph{just after} time $\tau^\hcal(z)$, specifically at time $\tau^\hcal_\delta(z)$. The uniformity in $\delta$, for which we have to be careful in the proof, will then allow us to deduce that in fact the geodesic itself also stays within the strip with high probability.

\begin{prop}\label{prop:almost_geodesic}
Let $z=(x,0)$. For any $\eps>0$, there exists $A_\eps,x_\eps\in(0,\infty)$ such that for all $\delta\in(0,1]$ and $x\ge x_\eps$,
\begin{multline*}
\P\Big(\exists C = (Z_u,R_u)_{u\ge 0} \in \ccal^{(z)}\big(\Pi^{(\tau^\hcal_\delta(z))}\big) :\\\bcal\Big(Z_{\tau^\hcal_\delta(z)},R_{\tau^\hcal_\delta(z)}\Big)\cap\hcal\neq\emptyset \text{ and } \bcal(Z_u,R_u)\subseteq \lcal^{(A\sqrt{x})} \,\, \forall u\le \tau^\hcal_\delta(z)\Big) \ge 1-\eps.
\end{multline*}
\end{prop}

\begin{proof}
We define, for any $t\in\mathbb{R}$ and $z\in\rtwo\setminus\hcal$,
\[\text{N}^{[z,t]} = \max\left\{n\in\mathbb{N} : \Gamma_{\Theta^{[z,t]}_n} \le \tau^{[z,t]}(\hcal)\right\}.\]
This is the number of jumps the path $(\Theta^{[z,t]}_n)_{n\ge 0}$ makes before hitting $\hcal$. Then for any $C,\theta>0$,
\begin{align*}
&\P\Big(\exists C = (Z_u,R_u)_{u\ge 0} \in \ccal^{(z)}\big(\Pi^{(\tau^\hcal_\delta(z))}\big) :\\[-1mm]
&\hspace{40mm}\bcal\Big(Z_{\tau^\hcal_\delta(z)},R_{\tau^\hcal_\delta(z)}\Big)\cap\hcal\neq\emptyset \text{ and } \bcal(Z_u,R_u)\subseteq \lcal^{(A\sqrt{x})} \,\, \forall u\le \tau^\hcal_\delta(z)\Big)\\
&\ge \sum_{j\in\delta\mathbb{N}} \P\Big(\hit{\hcal}{z}\in(j-\delta,j],\, \exists C = (Z_u,R_u)_{u\ge 0} \in \ccal^{(z)}\big(\Pi^{(j)}\big) :\\[-3mm]
&\hspace{60mm}\bcal\left(Z_{j},R_{j}\right)\cap\hcal\neq\emptyset \text{ and } \bcal\left(Z_u,R_u\right)\subseteq \lcal^{(A\sqrt{x})} \,\, \forall u\le j\Big)\\
&\ge \sum_{j\in\delta\mathbb{N}} \P\Big(\hit{\hcal}{z}\in(j-\delta,j],\, \hit{[z,j]}{\hcal}\le j,\, \bcal\left(Z_{\Theta^{[z,j]}_n},R_{\Theta^{[z,j]}_n}\right)\subseteq \lcal^{(A\sqrt{x})} \,\, \forall n\le \text{N}^{[z,j]}\Big)\\
&\ge \sum_{j\in\delta\mathbb{N}} \P\Big(\hit{\hcal}{z}\in(j-\delta,j],\, \hit{[z,j]}{\hcal}\le j,\, \text{N}^{[z,j]}\le \theta x,\,\bcal\left(Z_{\Theta^{[z,j]}_n},R_{\Theta^{[z,j]}_n}\right)\subseteq \lcal^{(A\sqrt{x})} \,\, \forall n\le \lfloor\theta x\rfloor\Big)\\
&\ge \sum_{j\in\delta\mathbb{N}} \bigg( \P\Big(\hit{\hcal}{z}\in(j-\delta,j],\, \hit{[z,j]}{\hcal}\le j,\, \text{N}^{[z,j]}\le \theta x\Big)\\[-1.5mm]
&\hspace{40mm}-\P\Big(\left\{\hit{\hcal}{z}\in(j-\delta,j]\right\}\cap\left\{ \exists n\le \lfloor\theta x\rfloor : \bcal\left(Z_{\Theta^{[z,j]}_n},R_{\Theta^{[z,j]}_n}\right)\nsubseteq \lcal^{(A\sqrt{x})} \right\}\Big)\bigg).
\end{align*}
We first deal with the sum of negative terms. By Lemma \ref{lem:backwards_transverse_displ},
\begin{align*}
&\sum_{j\in\delta\mathbb{N}}\P\Big(\left\{\hit{\hcal}{z}\in(j-\delta,j]\right\}\cap\left\{ \exists n\le \lfloor\theta x\rfloor : \bcal\left(Z_{\Theta^{[z,j]}_n},R_{\Theta^{[z,j]}_n}\right)\nsubseteq \lcal^{(A\sqrt{x})} \right\}\Big)\\
&\le \sum_{j\in\delta\mathbb{N}}\P\bigg(\left\{\hit{\hcal}{z}\in(j-\delta,j]\right\}\cap\bigg\{ \max_{n\le \lfloor\theta x\rfloor} \left|Y_{\Theta^{[z,j]}_n}\right| > A\sqrt{x} - \text{R}_0 \bigg\}\bigg)\\
&\le \sum_{j\in\delta\mathbb{N}}\frac{\theta x \text{R}_{0}^{2}}{(A\sqrt x-\text{R}_0)^{2}}\P\left(\hit{\hcal}{z}\in(j-\delta,j]\right)\\
&= \frac{\theta x \text{R}_{0}^{2}}{(A\sqrt x-\text{R}_0)^{2}},
\end{align*}
which, for any $\theta>0$ and $x>1$, can be made smaller than $\eps/2$ by choosing $A$ sufficiently large. It therefore suffices to show that
\begin{equation}\label{eq:backwards_N_bound}\sum_{j\in\delta\mathbb{N}} \P\Big(\hit{\hcal}{z}\in(j-\delta,j],\, \hit{[z,j]}{\hcal}\le j,\, \text{N}^{[z,j]}\le \theta x\Big) \ge 1-\eps/2
\end{equation}
for some $\theta>0$ and $x$ sufficiently large. A subtle difficulty arises because we want to eliminate any dependence on $\delta$.

We recall again the tree representation of chains of reproduction events from Section \ref{subsec:tree_representation}. We aim to proceed similarly to the proof of Lemma \ref{lem:tail_estimates_number_jumps}, noting that if $\hit{[z,j]}{\hcal}\le j$ and $\text{N}^{[z,j]} > \theta x$ then there must exist a vertex in $\tcal^{[z]}(\Pi^{(j)})$ outside the first $\lceil\theta x\rceil -1$ generations whose birth time is smaller than $j$. However, since the number of possible values of $j$ depends on $\delta$, we want to reduce the number of trees that we need to inspect.

Consider $0<s\le t$ (we will eventually take $s=j\in\delta\mathbb{N}$ and $t=\lceil j\rceil$). Recall that the vertices of $\tcal^{[z]}(\Pi^{(t)})$ are represented by strings of $1$s and $2$s, and the first child of each vertex inherits its position, whereas the second child moves according to the corresponding reproduction event. Since we will need to work momentarily with both trees $\tcal^{[z]}(\Pi^{(t)})$ and $\tcal^{[z]}(\Pi^{(s)})$, we will explicitly include the relevant point process in the notation of the vertex labels, e.g.~$Z_u^{[z]}(\Pi^{(t)})$ to represent the position of vertex $u$ in the tree $\tcal^{[z]}(\Pi^{(t)})$, and $\Gamma_u^{[z]}(\Pi^{(t)})$ and $T_u^{[z]}(\Pi^{(t)})$ for the birth and death times of the same vertex.

Let $v_j$ be the string of exactly $j$ $1$s, so $Z^{[z]}_{v_j}(\Pi^{(t)}) = z$ and $R^{[z]}_{v_j}(\Pi^{(t)}) = 0$ for all $j$. Let
\[J=J^{[z]}(s,t) = \min\{j : T^{[z]}_{v_j}(\Pi^{(t)}) \ge t-s\},\]
the first $j$ such that the death time of $v_j$ is after $t-s$. Then $\tcal^{[z]}(\Pi^{(s)})$ can be viewed as the subtree of $\tcal^{[z]}(\Pi^{(t)})$ rooted at $v_{J}$, with labels satisfying $Z_u^{[z]}(\Pi^{(s)}) = Z_{v_J u}^{[z]}(\Pi^{(t)})$, $R_u^{[z]}(\Pi^{(s)}) = R_{v_J u}^{[z]}(\Pi^{(t)})$, and $T_u^{[z]}(\Pi^{(s)}) = T_{v_J u}^{[z]}(\Pi^{(t)}) - (t-s)$. In particular, for any $n\in\mathbb{N}$ and $\gamma>0$, if there is a vertex $u$ outside the first $n$ generations of $\tcal^{[z]}(\Pi^{(s)})$ with $\Gamma_u^{[z]}(\Pi^{(s)}) \le \gamma$, then there certainly exists a vertex $u'$ outside the first $n$ generations of $\tcal^{[z]}(\Pi^{(t)})$ with $\Gamma_u^{[z]}(\Pi^{(t)}) \le \gamma+t-s$; indeed, $u' = v_J u$ is such a vertex.

Taking $j\in\delta\mathbb{N}$ and applying this argument with $s=j$, $t=\lceil j\rceil$, $n=\lceil\theta x\rceil -1$, and $\gamma = j$, we have
\begin{align*}
&\P\Big(\hit{\hcal}{z}\in(j-\delta,j],\, \hit{[z,j]}{\hcal}\le j,\, \text{N}^{[z,j]}> \theta x\Big)\\
&\le \P\Big(\hit{\hcal}{z}\in(j-\delta,j],\, \exists u\in \tcal^{[z]}(\Pi^{(j)}) \setminus \tcal^{[z]}_{\lceil\theta x\rceil-1}(\Pi^{(j)}) : \Gamma_u \le j\Big)\\
&\le \P\Big(\hit{\hcal}{z}\in(j-\delta,j],\, \exists u'\in \tcal^{[z]}(\Pi^{(\lceil j\rceil)}) \setminus \tcal^{[z]}_{\lceil\theta x\rceil-1}(\Pi^{(\lceil j\rceil)}) : \Gamma_u \le \lceil j\rceil \Big).
\end{align*}
Thus, for any $i\in\mathbb{N}$,
\begin{multline*}
\sum_{j\in\delta\mathbb{N}\cap(i-1,i]}\P\Big(\hit{\hcal}{z}\in(j-\delta,j],\, \hit{[z,j]}{\hcal}\le j,\, \text{N}^{[z,j]}> \theta x\Big)\\
\le \P\Big(\hit{\hcal}{z}\in(i-1,i],\, \exists u'\in \tcal^{[z]}(\Pi^{(i)}) \setminus \tcal^{[z]}_{\lceil\theta x\rceil-1}(\Pi^{(i)}) : \Gamma_u \le i \Big),
\end{multline*}
and for any $n\in\mathbb{N}$,
\begin{align*}
&\sum_{j\in\delta\mathbb{N}\cap(0,n]}\P\Big(\hit{\hcal}{z}\in(j-\delta,j],\, \hit{[z,j]}{\hcal}\le j,\, \text{N}^{[z,j]}> \theta x\Big)\\
&\le \sum_{i=1}^n\P\Big(\hit{\hcal}{z}\in(i-1,i],\, \exists u'\in \tcal^{[z]}(\Pi^{(i)}) \setminus \tcal^{[z]}_{\lceil\theta x\rceil-1}(\Pi^{(i)}) : \Gamma_u \le i \Big)\\
&\le \sum_{i=1}^n\P\Big(\exists u'\in \tcal^{[z]}(\Pi^{(i)}) \setminus \tcal^{[z]}_{\lceil\theta x\rceil-1}(\Pi^{(i)}) : \Gamma_u \le n \Big).
\end{align*}
Since, for each $i$, the point process $\Pi^{(i)}$ has the same distribution as $\Pi$, and using also the translation invariance of $\Pi$, we obtain that
\begin{equation}\label{eq:delta_sum_union}
\sum_{j\in\delta\mathbb{N}\cap(0,n]}\P\Big(\hit{\hcal}{z}\in(j-\delta,j],\, \hit{[z,j]}{\hcal}\le j,\, \text{N}^{[z,j]}> \theta x\Big) \le n\P\Big(\exists u'\in \tcal \setminus \tcal_{\lceil\theta x\rceil-1} : \Gamma_u \le n \Big).
\end{equation}

We now apply this with the aim of obtaining \eqref{eq:backwards_N_bound}. By Theorem \ref{thm:hitting_times}, we can choose $\beta>0$ such that $\P(\hit{\hcal}{z} > \lfloor \beta x\rfloor)<\eps/4$. Then, using also that if $\hit{\hcal}{z}\le j$ then $\hit{[z,j]}{\hcal}\le j$, we have
\begin{align*}
&\sum_{j\in\delta\mathbb{N}} \P\Big(\hit{\hcal}{z}\in(j-\delta,j],\, \hit{[z,j]}{\hcal}\le j,\, \text{N}^{[z,j]}\le \theta x\Big)\\
&\ge \sum_{j\in\delta\mathbb{N} \cap (0,\lfloor\beta x\rfloor]} \hspace{-1mm} \bigg(\P\Big(\hit{\hcal}{z}\in(j-\delta,j],\, \hit{[z,j]}{\hcal}\le j\Big) - \P\Big(\hit{\hcal}{z}\in(j-\delta,j],\, \hit{[z,j]}{\hcal}\le j,\, \text{N}^{[z,j]}> \theta x\Big)\hspace{-1mm}\bigg)\\
&= \P\Big(\hit{\hcal}{z} \le \lfloor\beta x\rfloor\Big) - \sum_{j\in\delta\mathbb{N} \cap (0,\lfloor\beta x\rfloor]}\P\Big(\hit{\hcal}{z}\in(j-\delta,j],\, \hit{[z,j]}{\hcal}\le j,\, \text{N}^{[z,j]}> \theta x\Big).
\end{align*}
By \eqref{eq:delta_sum_union}, this is at least
\[1-\eps/4 - \beta x \P\Big(\exists u'\in \tcal \setminus \tcal_{\lceil\theta x\rceil-1} : \Gamma_u \le \beta x \Big),\]
and then applying Lemma \ref{lem:moment_lemma_1} with $t=\beta x$ and $\theta = 2e\text{M}_0 \beta$, we have
\[\P\Big(\exists u'\in \tcal \setminus \tcal_{\lceil\theta x\rceil-1} : \Gamma_u \le \beta x \Big) \le 2e^{-\text{M}_0\beta x}\]
and therefore
\[\sum_{j\in\delta\mathbb{N}} \P\Big(\hit{\hcal}{z}\in(j-\delta,j],\, \hit{[z,j]}{\hcal}\le j,\, \text{N}^{[z,j]}\le \theta x\Big) \ge 1-\eps/4 - 2\beta x e^{-\text{M}_0\beta x}.\]
Taking $x$ sufficiently large establishes \eqref{eq:backwards_N_bound} and thus completes the proof.
\end{proof}

It is now a simple exercise to deduce Theorem \ref{thm:geodesic_plane_to_point}.

\begin{proof}[Proof of Theorem \ref{thm:geodesic_plane_to_point}]
Fix $A_\eps$ and $x_\eps$ as in Proposition \ref{prop:almost_geodesic}. For $\delta\in(0,1]$ and $x>0$, with $z=(x,0)$, let $\mathcal A_\delta(x)$ be the event from Proposition \ref{prop:almost_geodesic}, i.e.
\begin{multline*}
\mathcal A_\delta(x) := \Big\{\exists C = (Z_u,R_u)_{u\ge 0} \in \ccal^{(z)}\big(\Pi^{(\tau^\hcal_\delta(z))}\big) :\\\bcal\Big(Z_{\tau^\hcal_\delta(z)},R_{\tau^\hcal_\delta(z)}\Big)\cap\hcal\neq\emptyset \text{ and } \bcal(Z_u,R_u)\subseteq \lcal^{(A_\eps\sqrt{x})} \,\, \forall u\le \tau^\hcal_\delta(z)\Big\}.
\end{multline*}
Note that, since we can always extend a chain of reproduction events to create another chain over a longer time interval by rejecting any additional reproduction events, we have $\mathcal A_\delta(x)\subseteq \mathcal A_{\delta'}(x)$ for any $\delta\le \delta'$. Thus, by Proposition \ref{prop:almost_geodesic} and continuity of probability measures, if $x\ge x_\eps$ then
\[\P\Bigg(\bigcap_{\delta\in(0,1]}\mathcal A_\delta(x)\Bigg) = \lim_{\delta\to 0} \P(\mathcal A_\delta(x)) \ge 1-\eps.\]
But since our $\infty$-parent SLFV process is c\`adl\`ag, and $\tau^\hcal_\delta(z)\to \hit{\hcal}{z}$ as $\delta\to 0$, we deduce that in fact the event
\begin{multline*}
\mathcal A(x) := \Big\{\exists C = (Z_u,R_u)_{u\ge 0} \in \ccal^{(z)}\big(\Pi^{(\tau^\hcal(z))}\big) :\\\bcal\Big(Z_{\tau^\hcal(z)},R_{\tau^\hcal(z)}\Big)\cap\hcal\neq\emptyset \text{ and } \bcal(Z_u,R_u)\subseteq \lcal^{(A_\eps\sqrt{x})} \,\, \forall u\le \tau^\hcal(z)\Big\}
\end{multline*}
satisfies $\P(\mathcal A(x))\ge 1-\eps$ for all $x\ge x_\eps$, which is exactly the statement that there exists a geodesic from $\hcal$ to $z$ that remains within $\lcal^{(A_\eps\sqrt{x})}$ with probability at least $1-\eps$. This completes the proof.
\end{proof}

\section{\texorpdfstring{Hitting times: almost sure convergence, and control of the difference between $\hit{\hcal}{z}$ and $\bulk{\hcal}{z}$}{Hitting times: almost sure convergence, and control of the difference between hitting times and bulk coverage times}}\label{sec:hitting_and_difference}

\subsection{Almost sure convergence of hitting times}\label{sec:almost_sure_conv_hitting_times}

In this section we aim to prove Theorem \ref{thm:hitting_times}. We start with the first bullet point.

\begin{prop}\label{prop:almost_sure_plane}
The rescaled plane-to-line hitting times $\hit{\{0\}}{\hcal^x}/x$ converge as $x\to\infty$ almost surely and in $L^1$ to a constant $\nu\in(0,\infty)$.
\end{prop}

\begin{proof}
Essentially this is a repeat of the argument in \cite[Proposition 4.1]{louvet2023measurevalued} adapted to our notation, which in turn is a standard application of Liggett's strengthening \cite{liggett1985improved} of Kingman's subadditive ergodic theorem \cite{kingman1968ergodic}. We will show that an appropriately chosen sequence of hitting times satisfies the conditions of \cite[Theorem 1.10]{liggett1985improved}.

Let $P_0=(0,0)\in\rtwo$ and for each $n\in\mathbb{N}$, let $P_n=Z_{\Theta^{(n)}_{\text{N}(n)}}$, the centre of the reproduction event with which the random geodesic hits $\hcal^{n}$. (Our choice of $P_n$ is somewhat arbitrary; we simply need any point at which $H^{n}$ is hit for the first time.) In order to apply the subadditive ergodic theorem, we need to consider time-shifts of our $\infty$-parent SLFV process. To this end, for each $n\in\mathbb{N}$ let $\Pi(n)$ be the time-shift of $\Pi$ by $\hit{\{0\}}{\hcal^{n}}$, and not including the point at time $\hit{\{0\}}{\hcal^{n}}$; that is,
\[(t,z,r) \in \Pi(n) \Longleftrightarrow (t+\hit{\{0\}}{\hcal^{n}},z,r) \in \Pi,\, t\neq 0.\]
Then let $S_t(n)$ be the $\infty$-parent SLFV started from $\{P_n\}$ and built using $\Pi(n)$, i.e. using reproduction events that occur strictly after $\hit{\{0\}}{\hcal^{n}}$. We also let $S_t(0) = S_t^{\{0\}}$.

Now, for $m<n\in\{0,1,2,\ldots\}$, let
\[T_{m,n}:= \inf\{t\ge 0 : S_t(m)\cap\hcal^{n} \neq \emptyset\}.\]
In words, we wait until the first time that $\hcal^{m}$ is hit; we take one particular point $P_n$ at which $\hcal^{n}$ is first hit; we start an $\infty$-parent SLFV provess from that point using reproduction events strictly after that time; and $T_{m,n}$ captures the amount of time this SLFV takes to hit $\hcal^{n}$.

We then note that:
\begin{enumerate}
\item $T_{0,n}\le T_{0,m} + T_{m,n}$ for any $0<m<n$;
\item for each $m\ge 0$, the joint distributions of $\{T_{m+1,m+k+1}: k\ge 1\}$ are the same as those of $\{T_{m,m+k} : k\ge 1\}$ (by the Markov property of the underlying Poisson point process $\Pi$);
\item for each $k\ge 1$, the process $(T_{nk,(n+1)k})_{n\ge 1}$ is stationary (by the stationarity of $\Pi$);
\item for each $n$, $T_{0,n}\ge 0$ and $\E[T_{0,n}]<\infty$, by Lemma \ref{lem:upper_bound_sigma_x}.
\end{enumerate}
These are precisely the conditions (1.7), (1.8), (1.9) and (1.3) from \cite{liggett1985improved}. In fact, we observe that the sequence in point 3 above is i.i.d. and therefore ergodic. We therefore deduce from \cite[Theorem 1.10]{liggett1985improved} that the constant
\[\nu := \inf_{n\ge 1} \frac{1}{n}\E[T_{0,n}]\]
satisfies
\[\nu = \lim_{n\to\infty}\frac{1}{n}\E[T_{0,n}]\]
and
\[\nu = \lim_{n\to\infty}\frac{1}{n}T_{0,n} \,\,\text{ almost surely.}\]
The proof is then almost complete, by the fact that $T_{0,n} = \hit{\{0\}}{\hcal^n}$.

To move from the countable sequence $n\to\infty$ through $\mathbb{N}$ to the full limit $x\to\infty$, we simply observe that by monotonicity, we have
\[\frac{n}{n+1}\cdot\frac{\hit{\{0\}}{\hcal^{n}}}{n} \le \inf_{x\in[n,n+1]}\frac{\hit{\{0\}}{\hcal^x}}{x} \le \sup_{x\in[n,n+1]}\frac{\hit{\{0\}}{\hcal^x}}{x} \le \frac{n+1}{n}\cdot\frac{\hit{\{0\}}{\hcal^{n+1}}}{n+1}.\]
Finally, we need to check that $\nu>0$. Note that any chain of reproduction events that reaches $\hcal^x$ must have jumped at least $x/2\text{R}_0$ times, since the maximum radius of a reproduction event is $\text{R}_0$. Let $c$ be the constant from Lemma \ref{lem:moment_lemma_1}. Then, recalling that in the tree representation from Section \ref{subsec:tree_representation}, $\Gamma_u$ is the birth time of particle $u$, we have
\[\P\left(\hit{\{0\}}{\hcal^x} \le \frac{x}{4ce\text{R}_0}\right) \le \P\left(\exists u\in\tcal\setminus\tcal_{\lceil \frac{x}{2\text{R}_0}\rceil -1} : \Gamma_u \le \frac{x}{4ce\text{R}_0}\right).\]
Applying Lemma \ref{lem:moment_lemma_1}, this is at most $2\exp(-\frac{x}{4e\text{R}_0})$, and we deduce that $\nu>\frac{1}{4ce\text{R}_0}$. This completes the proof.
\end{proof}

Our results on the transverse fluctuations of geodesics then allow us to show that the rescaled point-to-point hitting times converge to the same constant $\nu$.

\begin{prop}\label{prop:almost_sure_point}
The rescaled point-to-point hitting times $\hit{\{0\}}{x,0}/x$ also converge as $x\to\infty$ almost surely and in $L^1$ to the same constant $\nu\in(0,\infty)$ as in Proposition \ref{prop:almost_sure_plane}.
\end{prop}

\begin{proof}
The strategy is as follows. Clearly $\hit{\{0\}}{x,0}\ge \hit{\{0\}}{\hcal^x}$ so it suffices to give an upper bound. We know that we hit $\hcal^x$ at roughly time $\nu x$. We know from Lemma \ref{lem:large_devs_transverse_displ} that the random geodesic hits $\hcal^x$ at a point $P(x)$ that is of sublinear distance from $(x,0)$ with probability exponentially close to $1$. Then by Proposition \ref{prop:tail_tau_x}, starting from $P(x)$ we know that the time to cover the whole of the line between $P(x)$ and $(x,0)$ is linear in the distance between the two points (with probability exponentially close to $1$). This gives the result. We now carry out the details.

Take $\beta_\mu,c_\mu>0$ as in Proposition \ref{prop:tail_tau_x}, and let $\eps > 0$. Then apply Lemma \ref{lem:large_devs_transverse_displ} with $\delta = \frac{\eps }{2\beta_\mu}$ to obtain $c>0$ such that for all $x > 4 \text{R}_{0} \beta_{\mu} \eps^{-1}$,
\[\P\left(Y_{\Theta^{(x)}_{\text{N}(x)}} > \frac{\eps x}{2\beta_\mu}\right) \le e^{-cx}.\]
We then have
\begin{multline*}
\P\left(\hit{\{0\}}{\hcal^x}\le (\nu+\eps/2)x \text{ and } \hit{\{0\}}{x,0} > (\nu+\eps)x\right) \\
\le \P\left(Y_{\Theta^{(x)}_{\text{N}(x)}} > \frac{\eps x}{2\beta_\mu}\right) + \P\left(\hit{\{0\}}{\hcal^x}\le (\nu+\eps/2)x,\,\, Y_{\Theta^{(x)}_{\text{N}(x)}} \le \frac{\eps x}{2\beta_\mu},\,\, \text{ and } \,\,\hit{\{0\}}{x,0} > (\nu+\eps)x\right).
\end{multline*}
The first probability on the right-hand side above is at most $e^{-cx}$, and the second probability is, by the strong Markov property, at most
\[\P\left(\sigma^{\{0\}}\left(\frac{\eps x}{2\beta_\mu},0\right)>\eps x/2\right)\]
which by Proposition \ref{prop:tail_tau_x} is at most $\exp(-c_\mu x)$.

We have established that
\[\P\left(\hit{\{0\}}{\hcal^x}\le (\nu+\eps/2)x\text{ and } \hit{\{0\}}{x,0} > (\nu+\eps)x\right) \le \exp(-cx) + \exp(-c_\mu x).\]
We deduce by the Borel-Cantelli lemma that
\[\P\left(\hit{\{0\}}{\hcal^n}\le (\nu+\eps/2)n \text{ and }\hit{\{0\}}{n,0} > (\nu+\eps)n\, \text{ for infinitely many } n\in\mathbb{N}\right) = 0.\]
By Proposition \ref{prop:almost_sure_plane}, we know that the former event occurs for all large $n$, and therefore
\[\P\left(\hit{\{0\}}{n,0} > (\nu+\eps)n\, \text{ for infinitely many } n\in\mathbb{N}\right) = 0.\]

To move from discrete $n$ to continuous $x$, we again apply Proposition \ref{prop:tail_tau_x}, simply observing that by the strong Markov property, for all $n > \beta_{\mu}\eps^{-1},$
\begin{align*}
\P\left(\hit{\{0\}}{n,0} \le (\nu+\eps)n,\, \exists x\in[n,n+1] : \hit{\{0\}}{x,0} > (\nu+2\eps)x\right) &\le \P\left(\bulk{\{(n,0)\}}{n+1,0} > \eps n\right)\\
&= \proba\left(
\bulk{\{(0,0)\}}{1,0} > \eps n
\right) \\
&\le \exp(-c_{\mu} \eps n). 
\end{align*}
Again the Borel-Cantelli lemma tells us that this occurs only finitely often, which allows us to conclude.
\end{proof}

The proof of Theorem \ref{thm:hitting_times} now follows immediately from these results together with Corollary \ref{cor:hitting_times_equal_distn}.

\begin{proof}[Proof of Theorem \ref{thm:hitting_times}]
The first bullet point in Theorem \ref{thm:hitting_times} is precisely the statement in Proposition \ref{prop:almost_sure_plane}, and the second follows from Proposition \ref{prop:almost_sure_point} together with rotational invariance of the process. The third bullet point follows from the first together with the equality in distribution of $\hit{\hcal}{z}$ and $\hit{\{0\}}{\hcal^x}$, proved in Corollary \ref{cor:hitting_times_equal_distn}.
\end{proof}

\subsection{A shape theorem: proof of Theorem \ref{thm:ball_shape}}\label{subsec:shape_thm}

We showed in Proposition \ref{prop:tail_tau_x} that when we start from $0$, we can cover the whole line between $0$ and a point $z\in\rtwo$ in a time that is at most linear in $\|z\|$. We now give a simple corollary which allows us to cover all directions simultaneously, and therefore gives us a first linear upper bound on the time to cover a ball. It says that with high probability we cover the ball of radius $r$ in time at most $\gamma (r+1)$, for suitably large $\gamma$. Note that this is still a long way from proving Theorem \ref{thm:ball_shape}, which says that the correct linear factor is not $\gamma$ but $\nu$, the same constant that appears in Theorem \ref{thm:hitting_times}.

\begin{cor}\label{cor:uniform_tau_bound}
For any $\eps>0$, there exists $\gamma\in(0,\infty)$ such that
\[\P\left(\sup_{z\in\rtwo} \frac{\hit{\{0\}}{z}}{1+\|z\|} > \gamma \right) \le \eps.\]
\end{cor}

\begin{proof}
Fix $\delta=\delta_\mu\wedge 1$ where $\delta_\mu$ is the constant from Proposition \ref{prop:tail_tau_x}. We cover $\rtwo$ with balls of radius $\delta$ whose centres are located on the discrete grid $\mathbb{Z}_{\delta}^2 \setminus\{0\}$. Note that for any $w\in\rtwo$, taking $z_{w}$ to be the closest point in $\mathbb{Z}_{\delta}^2 \setminus\{0\}$, we have $w\in\bcal(z_{w},\delta)$ and $1+\|w\|\ge \|z_{w}\|$. Thus
\[\left\{\frac{\hit{\{0\}}{w}}{1+\|w\|} > \gamma\right\} \subseteq \left\{\sup_{z' \in\bcal(z_{w},\delta) } \bulk{\{0\}}{z'} > \gamma\|z_{w}\| \right\}\]
and therefore
\begin{align*}
\P\left(\sup_{w\in\rtwo} \frac{\hit{\{0\}}{w}}{1+\|w\|} > \gamma \right) &\le \P\left(\exists z\in \mathbb{Z}_{\delta}^2 \setminus\{0\} : \sup_{z' \in\bcal(z,\delta) } \bulk{\{0\}}{z'} > \gamma\|z\| \right)\\
&\le \sum_{z\in \mathbb{Z}_{\delta}^2 \setminus\{0\}} \P\left(\sup_{z' \in\bcal(z,\delta) } \bulk{\{0\}}{z'} > \gamma\|z\| \right).
\end{align*}
Since $\delta\le \delta_\mu$, we can apply Proposition \ref{prop:tail_tau_x} to see that for any $\beta>\beta_\mu$, this is at most
\[\sum_{z\in \mathbb{Z}_{\delta}^2 \setminus\{0\}} \exp(-c_\mu\beta\|z\|) \le \sum_{j=1}^\infty 2^{2j} e^{-c_\mu\beta j\delta},\]
which we can ensure is smaller than $\eps$ by choosing $\beta$ sufficiently large, completing the proof.
\end{proof}

For our next result, which will be very similar to Corollary \ref{cor:uniform_tau_bound} but ``backwards in time'', we note that---although we normally imagine specifying $E$ in advance and watching one $\infty$-parent SLFV process evolve starting from $E$---our definition of the $\infty$-parent SLFV allows us to construct the process starting from several initial conditions simultaneously, using the same Poisson point process $\Pi$. In particular, for two points $w\neq z\in\rtwo$, $S_t^{\{w\}}$ describes the set of points covered by chains of $\Pi^{(t)}$-reproduction events started from $w$, and $S_t^{\{z\}}$ describes the set of points covered by chains of $\Pi^{(t)}$-reproduction events started from $z$; both objects exist simultaneously.

We also need some more notation. For $t\in\mathbb{R}$, write $\Pi_{(t)}$ for the $t$-shift of $\Pi$, i.e.
\begin{equation}\label{eq:pi_timeshift}
(t',z',r')\in\Pi_{(t)}\Longleftrightarrow (t+t',z',r')\in\Pi.
\end{equation}
Then (recalling Lemma \ref{lem:hitting_times_charac_forward}) define, for $w,z\in\rtwo$,
\[\tau_t^{\{w\}}(z) = \min\left\{ s \geq 0 : \exists (z',r') \in \ccal_{s}^{(w)}(\Pi_{(t)}) \text{ with } z\in \bcal(z',r')\right\}.\]
In words, $\tau_t^{\{w\}}(z)$ is the time taken after $t$ to hit $z$ when starting from $\{w\}$ and using reproduction events that occur after time $t$, rather than time $0$ (we emphasise that $t$ could be negative). Similarly define
\[\sigma_t^{\{w\}}(z) = \min\left\{ s \geq 0 : \forall u\in(0,1],\, \exists \left(z',r'\right) \in \ccal_{s}^{(0)}(\Pi_{(t)}) \text{ with } w+u(z-w) \in \bcal\left( z',r' \right) \right\},\]
the time taken after $t$ to cover the whole line between $w$ and $z$ when using reproduction events that occur after time $t$.

We are now ready to state the result. Recall that Corollary \ref{cor:uniform_tau_bound} said that with high probability, starting from $0$, we we can hit each other point $z$ within a time that grows linearly with $\|z\|$ (with a large constant factor), uniformly over all $z$. Our next result says the reverse: that with high probability, starting from each point $z$, we can hit $0$ within a time linear in $\|z\|$, uniformly over all $z$.

\begin{cor}\label{cor:uniform_backwards_tau_bound}
For any $\eps>0$, there exists $\gamma\in(0,\infty)$ such that
\[\P\left(\sup_{z\in\rtwo} \frac{\tau^{\{z\}}_{-\gamma(1+\|z\|)}(0)}{1+\|z\|} > \gamma \right) \le \eps.\]
\end{cor}

\begin{proof}
We proceed similarly to the proof of Corollary \ref{cor:uniform_tau_bound}. Again let $\delta = \delta_\mu\wedge 1$ where $\delta_\mu$ is the constant from Proposition \ref{prop:tail_tau_x}. Note that
\begin{align*}
\P\left(\sup_{z\in\rtwo} \frac{\tau^{\{z\}}_{-\gamma(1+\|z\|)}(0)}{1+\|z\|} > \gamma \right) & \le \P\left( \exists z\in \mathbb{Z}^2_\delta\setminus\{0\} : \sup_{z'\in\bcal(z,\delta)} \sigma^{\{z'\}}_{-\gamma(1+\|z'\|)}(0) > \gamma\|z\|\right)\\
&\le \sum_{z\in \mathbb{Z}^2_\delta\setminus\{0\}} \P\left( \sup_{z'\in\bcal(z,\delta)} \sigma^{\{z'\}}_{-\gamma(1+\|z'\|)}(0) > \gamma\|z\|\right)\\
&\le \sum_{z\in \mathbb{Z}^2_\delta\setminus\{0\}} \P\left( \sup_{z'\in\bcal(z,\delta)} \sigma^{\{z'\}}_{-\gamma\|z\|}(0) > \gamma\|z\|\right)\\
&= \sum_{z\in \mathbb{Z}^2_\delta\setminus\{0\}} \P\left( \sup_{z'\in\bcal(z,\delta)} \bulk{\{z'\}}{0} > \gamma\|z\|\right)\\
&= \sum_{z\in \mathbb{Z}^2_\delta\setminus\{0\}} \P\left( \sup_{z'\in\bcal(0,\delta)} \bulk{\{z'\}}{z} > \gamma\|z\|\right)
\end{align*}
where the penultimate line follows by the stationarity of the Poisson point process $\Pi$, and the last line follows from translation invariance.

We cannot directly apply Proposition \ref{prop:tail_tau_x} as we did in the proof of Corollary \ref{cor:uniform_tau_bound}, but the bound is essentially the same: due to Lemma \ref{lem:properties_sequence_reproduction_events} (ii), we have
\[\sup_{z'\in\bcal(0,\delta)} \bulk{\{z'\}}{z} \le \overrightarrow{T}_{\lceil x/\delta\rceil},\]
and therefore by Lemma \ref{lem:tail_estimate_slow_covering_chain},
\[\P\left( \sup_{z'\in\bcal(0,\delta)} \bulk{\{z'\}}{z} > \gamma\|z\|\right) \le \P\left( \overrightarrow{T}_{\lceil x/\delta\rceil} > \gamma\|z\|\right) \le \exp(-c_\mu\beta\|z\|).\]
Thus
\[\P\left(\sup_{z\in\rtwo} \frac{\tau^{\{z\}}_{-\gamma(1+\|z\|)}(0)}{1+\|z\|} > \gamma \right) \le \sum_{z\in \mathbb{Z}_{\delta}^2 \setminus\{0\}} \exp(-c_\mu\beta\|z\|) \le \sum_{j=1}^\infty 2^{2j} e^{-c_\mu\beta j\delta},\]
which we can make arbitrarily small by choosing $\beta$ sufficiently large.
\end{proof}

Now, by Corollaries \ref{cor:uniform_tau_bound} and \ref{cor:uniform_backwards_tau_bound}, we can fix $\gamma$ such that 
\[\P\left(\sup_{z\in\rtwo} \frac{\hit{\{0\}}{z}}{1+\|z\|} > \gamma \right)\le 1/2 \quad\text{ and }\quad  \P\left(\sup_{z\in\rtwo} \frac{\tau^{\{z\}}_{-\gamma(1+\|z\|)}(0)}{1+\|z\|} > \gamma \right) \le 1/2.\]
Say that $z\in\rtwo$ is \emph{outwardly} $t$-\emph{good} if
\[\sup_{z'\in\rtwo} \frac{\tau_t^{\{z\}}(z')}{1+\|z'-z\|} \le \gamma,\]
i.e.~if starting from $z$ at time $t$, any point $z'\in\rtwo$ is reached in a further time $\gamma(1+\|z'-z\|)$. Say that $z\in\rtwo$ is \emph{inwardly} $t$-\emph{good} if
\[\sup_{z'\in\rtwo} \frac{\tau^{\{z'\}}_{t-\gamma(1+\|z'-z\|)}(z)}{1+\|z'-z\|} \le \gamma,\]
i.e.~if $z$ is reached by time $t$ from any point $z'\in\rtwo$ by starting $\gamma(1+\|z'-z\|)$ time earlier.

Note that by stationarity and translation-invariance of the underlying Poisson point process,  for any $z\in\rtwo$ and $t\ge 0$ we have
\begin{equation}\label{eq:good_invariance}
\P(z\text{ is outwardly $t$-good}) = \P\left(\sup_{z'\in\rtwo} \frac{\tau_t^{\{z\}}(z')}{1+\|z'-z\|} \le \gamma\right) = \P\left(\sup_{z'\in\rtwo} \frac{\hit{\{0\}}{z'}}{1+\|z'\|} \le \gamma \right) \ge 1/2,
\end{equation}
by our choice of $\gamma$, and similarly
\begin{equation}\label{eq:inward_good_invariance}
\P(z\text{ is inwardly $t$-good}) = \P\left(\sup_{z'\in\rtwo} \frac{\tau^{\{z'\}}_{t-\gamma(1+\|z'-z\|)}(z)}{1+\|z'-z\|} \le \gamma\right) = \P\left(\sup_{z\in\rtwo} \frac{\tau^{\{z\}}_{-\gamma(1+\|z\|)}(0)}{1+\|z\|} \le \gamma \right) \ge 1/2.
\end{equation}

For $\eps>0$, say that $z\in\rtwo$ is \emph{outwardly} $\eps$-\emph{great} if there exists a strictly increasing sequence $(n_j)_{j\ge 1}$ of natural numbers such that $n_j z$ is outwardly $(\nu+\eps)n_j\|z\|$-good for all $j\ge 1$, and $n_{j+1}/n_j \to 1$ as $j\to\infty$. Similarly say that $z\in\rtwo$ is \emph{inwardly} $\eps$-\emph{great} if there exists a strictly increasing sequence $(m_j)_{j\ge 1}$ of natural numbers such that $m_j z$ is inwardly $(\nu-\eps)m_j\|z\|$-good for all $j\ge 1$, and $m_{j+1}/m_j \to 1$ as $j\to\infty$. The idea here is that $n_j z$ should have been hit by time $(\nu+\eps)n_j\|z\|$, so we can use reproduction events after this time to create a chain from $n_j z$ to other nearby vertices, and therefore they will also be hit soon afterwards; and similarly if any vertex near to $m_j z$ was hit far too early, we could create an inward chain to $m_j z$ which would cause $m_j z$ to also be hit too early.

\begin{lem}\label{lem:z_great_as}
For any fixed $z\in\rtwo$ and $\eps>0$, $\P(z \text{ is outwardly and inwardly $\eps$-great}) = 1$.
\end{lem}

\begin{proof}
Fix $z\in\rtwo$ and $\eps>0$. Let $G_n$ be the event that $nz$ is outwardly $(\nu+\eps)n\|z\|$-good. Then since the shift map by $z$ in space and $(\nu+\eps)\|z\|$ in time is ergodic, by Von Neumann's $L^1$ ergodic theorem,
\[\frac{1}{n}\sum_{i=1}^{n} \ind_{G_i} \to \P(G_0) = \P(z \text{ is outwardly $(\nu+\eps)\|z\|$-good})  \ge 1/2 \,\, \text{ almost surely}\]
by \eqref{eq:good_invariance}. Thus we almost surely have infinitely many natural numbers $n$ such that $nz$ is $(\nu+\eps)n\|z\|$-good. Label them in increasing order as $n_1, n_2,\ldots$ and note that then also
\[\frac{j}{n_j} = \frac{1}{n_j}\sum_{i=1}^{n_j} \ind_{G_i} \to \P(G_0) \,\, \text{ almost surely},\]
so
\[\frac{n_j}{n_{j+1}} = \frac{n_j}{j}\cdot\frac{j+1}{n_{j+1}}\cdot\frac{j+1}{j} \to 1 \,\, \text{ almost surely}.\]
This shows that $z$ is almost surely outwardly $\eps$-good. The proof that it is almost surely inwardly $\eps$-good is identical, just using the shift map by $z$ in space and $(\nu-\eps)\|z\|$ in time.
\end{proof}

Before we proceed, we will also need the following (almost obvious, but we include a short proof for completeness) simple triangle inequality for hitting times.

\begin{lem}\label{lem:triangle_ineq_chains}
For any $z_1,z_2,z_3\in\rtwo$ and $t\ge 0$, if $\hit{\{z_1\}}{z_2} \le t$, then
\[\hit{\{z_1\}}{z_3} \le t + \tau_t^{\{z_2\}}{z_3}.\]
\end{lem}

\begin{proof}
If $\hit{\{z_1\}}{z_2} \le t$, then there exists a chain $C^{(1)} = (Z^{(1)}_u,R^{(1)}_u)_{u\ge 0}$ of $\Pi$-reproduction events started from $z_1$ such that $z_2\in \bcal(Z^{(1)}_t,R^{(1)}_t)$. If we further know that $\tau_t^{\{z_2\}}{z_3}\le s$, then there exists a chain $C^{(2)} = (Z^{(2)}_u,R^{(2)}_u)_{u\ge 0}$ of $\Pi_{(t)}$-reproduction events started from $z_2$ such that $z_3\in \bcal(Z^{(2)}_s,R^{(2)}_s)$. We then let $C^{(3)}_u = C^{(1)}_u$ for $u\le t$ and $C^{(3)}_u = C^{(2)}_{u-t}$ for $u>t$. (Or, on the event of probability zero that there is a $\Pi$-reproduction event at time $t$, if that event appears in $C^{(2)}_0$ and not $C^{(1)}_t$, then we set $C^{(3)}_t = C^{(2)}_{0}$ instead of $C^{(3)}_t = C^{(1)}_t$.) The process $C^{(3)}$ is then a chain of $\Pi$-reproduction events started from $z_1$ such that $z_3\in \bcal(Z^{(3)}_{t+s},R^{(3)}_{t+s})$, and therefore $\hit{\{z_1\}}{z_3}\le t+s$. This completes the proof.
\end{proof}

We now use the fact that each point is almost surely inwardly and outwardly $\eps$-great to prove Theorem \ref{thm:ball_shape}. We will proceed in two steps, first showing that no point is hit too late, and then showing that no point is hit too early; the main ideas of the two proofs are very similar, but since the details are different, we carry them out separately.

Proposition \ref{prop:no_point_hit_late} says that no point is hit too late. Essentially we can use Theorem \ref{thm:hitting_times} to ensure that all rational points are hit at the right time, and Lemma \ref{lem:z_great_as} to ensure that after hitting each rational point, we can hit all the surrounding points not much later.

\begin{prop}\label{prop:no_point_hit_late}
For any $\eps>0$,
\[\P\left( \bcal\big(0,(\nu^{-1}-\eps)t\big)\subseteq S^{\{0\}}_t \,\,\text{ for all large } t\right) =1.\]
\end{prop}

\begin{proof}
Fix $\delta'>0$, to be specified later. Let $\Upsilon$ be the event that both
\[\frac{\hit{\{0\}}{nz}}{n\|z\|} \to \nu \quad \forall z\in\mathbb{Q}^2\]
and $z$ is outwardly $\delta'$-great for all $z\in\mathbb{Q}^2$. By Theorem \ref{thm:hitting_times} and Lemma \ref{lem:z_great_as}, and the fact that $\mathbb{Q}^2$ is countable, we know that $\P(\Upsilon)=1$.

Fix $\delta>0$. For a contradiction, suppose that there exists a sequence $(w_n)_{n\ge 1}$ in $\rtwo$ with $\|w_n\|\to+\infty$ and
\[\frac{\hit{\{0\}}{w_n}}{\|w_n\|} > \nu + \delta.\]
By compactness, by taking a subsequence if necessary, we may assume that $w_n/\|w_n\|\to w$ for some $w$ (with $\|w\|=1$). Also, for each $n\in\mathbb{N}$, choose $d(n)\in\mathbb{N}$ such that $d(n)\le \|w_n\| < d(n)+1$.

Choose $z\in\mathbb{Q}^2$ such that $\|z-w\|<\delta'$. On $\Upsilon$, we may take $(n_j)_{j\ge 1}$ a strictly increasing sequence in $\mathbb{N}$ such that $n_j z$ is outwardly $(\nu+\delta')n_j\|z\|$-good for all $j$, and $n_j/n_{j+1} \to 1$. We then choose $J\in\mathbb{N}$ such that
\begin{itemize}
\item $n_{j}/n_{j+1} > 1-\delta'$ for all $j\ge J$;
\item $\displaystyle{\left|\frac{\hit{\{0\}}{n_j z}}{n_j\|z\|} - \nu\right|<\delta'}$ for all $j\ge J$ (this is possible since $z$ is rational and we are on $\Upsilon$).
\end{itemize}
Let $\kappa(n)$ be such that
\begin{equation}\label{eq:kappa_defn}
n_{\kappa(n)} \le \|w_n\| \le n_{\kappa(n)+1}.
\end{equation}
Then choose $N$ such that
\begin{itemize}
    \item $\kappa(n)\ge J$ for all $n\ge N$;
    \item $n_{\kappa(n)} \ge 1/\delta'$ for all $n\ge N$;
    \item $\|w_n/\|w_n\| - w\| < \delta'$ for all $n\ge N$.
\end{itemize}
Then for any $n\ge N$, on $\Upsilon$, we have $\hit{\{0\}}{n_{\kappa(n)}z} \le (\nu+\delta')n_{\kappa(n)}\|z\|$, and therefore by Lemma \ref{lem:triangle_ineq_chains},
\[\hit{\{0\}}{w_n} \le (\nu+\delta')n_{\kappa(n)}\|z\| + \tau^{\{n_{\kappa(n)}z\}}_{(\nu+\delta')n_{\kappa(n)}\|z\|}(w_n).\]
Thus, since $n_{\kappa(n)}z$ is outwardly $(\nu+\delta')n_{\kappa(n)}\|z\|$-good, we have
\[\hit{\{0\}}{w_n} \le (\nu+\delta')n_{\kappa(n)}\|z\| + \gamma\big(1+\|w_n - n_{\kappa(n)}z\|\big).\]
Dividing through by$\|w_n\|$ and using~\eqref{eq:kappa_defn}, we obtain that
\[\frac{\hit{\{0\}}{w_n}}{\|w_n\|} \le (\nu+\delta')\|z\| + \frac{\gamma}{n_{\kappa(n)}} + \gamma\left\|\frac{w_n}{\|w_n\|} - \frac{n_{\kappa(n)}z}{\|w_n\|}\right\|.\]
By our choice of $z$ we have $\|z\|\le 1+\delta'$; by our choice of $N$ we have $n_{\kappa(n)}\ge 1/\delta'$; and so, using the triangle inequality,
\[\frac{\hit{\{0\}}{w_n}}{\|w_n\|} \le (\nu+\delta')(1+\delta') + \gamma\delta' + \gamma\left\|\frac{w_n}{\|w_n\|} - w\right\| + \gamma\|w-z\| + \gamma\left\|z-\frac{n_{\kappa(n)}z}{\|w_n\|}\right\|.\]
Finally, by our choice of $N$ the third and fourth terms on the right-hand side are at most $\gamma\delta'$, and using \eqref{eq:kappa_defn} and the fact that (again by our choice of $N$) $n_{\kappa(n)}/n_{\kappa(n)+1} > 1-\delta'$, we have
\begin{align*}
\frac{\hit{\{0\}}{w_n}}{\|w_n\|} &\le (\nu+\delta')(1+\delta') + 3\gamma\delta' + \gamma\|z\|\left(1-\frac{n_{\kappa(n)}}{n_{\kappa(n)+1}}\right)\\
&\le (\nu+\delta')(1+\delta') + 3\gamma\delta' + \gamma(1+\delta')\delta'.
\end{align*}
By choosing $\delta'$ sufficiently small, we can make this smaller than $\nu+\delta$, contradicting our choice of the sequence $w_n$. We deduce that for any $\delta>0$, there cannot exist a sequence $(w_n)_{n\ge 1}$ in $\rtwo$ with $\|w_n\|\to+\infty$ and
\[\frac{\hit{\{0\}}{w_n}}{\|w_n\|} > \nu + \delta;\]
and thus for any $\eps>0$ we must have
$\bcal\big(0,(\nu^{-1}-\eps)t\big)\subseteq S^{\{0\}}_t$ for all large $t$.
\end{proof}

Proposition \ref{prop:no_point_hit_early}, which says that no point is hit too early, has a very similar proof. The idea is that if there were ``bad'' points that were hit too early, then by Lemma \ref{lem:z_great_as} there would be nearby points with rational co-ordinates that were also hit too early; but by Theorem \ref{thm:hitting_times} we know that all rational points are hit at the right time.

\begin{prop}\label{prop:no_point_hit_early}
For any $\eps>0$,
\[\P\left( S^{(0)}_t \subseteq \bcal\big(0,(\nu^{-1}+\eps)t\big) \,\,\text{ for all large } t\right) =1.\]
\end{prop}

\begin{proof}
Fix $\delta'>0$, to be specified later. Let $\Upsilon$ be the event that both
\[\frac{\hit{\{0\}}{nz}}{n\|z\|} \to \nu \quad \forall z\in\mathbb{Q}^2\]
and $z$ is inwardly $\delta'$-great for all $z\in\mathbb{Q}^2$. By Theorem \ref{thm:hitting_times} and Lemma \ref{lem:z_great_as}, and the fact that $\mathbb{Q}^2$ is countable, we know that $\P(\Upsilon)=1$.

Fix $\delta>0$. For a contradiction, suppose that there exists a sequence $(w_n)_{n\ge 1}$ in $\rtwo$ with $\|w_n\|\to+\infty$ and
\[\frac{\hit{\{0\}}{w_n}}{\|w_n\|} < \nu - \delta.\]
By compactness, by taking a subsequence if necessary, we may assume that $w_n/\|w_n\|\to w$ for some $w$ (with $\|w\|=1$). Also, for each $n\in\mathbb{N}$, choose $d(n)\in\mathbb{N}$ such that $d(n)\le \|w_n\| < d(n)+1$.

Choose $z\in\mathbb{Q}^2$ such that $\|z-w\|<\delta'$. On $\Upsilon$, we may take $(m_j)_{j\ge 1}$ a strictly increasing sequence in $\mathbb{N}$ such that $m_j z$ is inwardly $(\nu-\delta')m_j\|z\|$-good for all $j$, and $m_j/m_{j+1} \to 1$. We then choose $J\in\mathbb{N}$ such that
\begin{itemize}
\item $m_{j}/m_{j+1} > 1-\delta'$ for all $j\ge J$;
\item $\displaystyle{\left|\frac{\hit{\{0\}}{m_j z}}{m_j\|z\|} - \nu\right|<\delta'}$ for all $j\ge J$ (this is possible since $z$ is rational and we are on $\Upsilon$).
\end{itemize}
Let $\kappa(n)$ be such that
\begin{equation*}
n_{\kappa(n)} \le \|w_n\| \le n_{\kappa(n)+1}.
\end{equation*}
Then choose $N$ such that
\begin{itemize}
    \item $\kappa(n)\ge J$ for all $n\ge N$;
    \item $m_{\kappa(n)} \ge 1/\delta'$ for all $n\ge N$;
    \item $\|w_n/\|w_n\| - w\| < \delta'$ for all $n\ge N$.
\end{itemize}
Note that then, on $\Upsilon$,
\[\|w_n-m_{\kappa(n)}z\| \le \|w_n\|\cdot\left\|\tfrac{w_n}{\|w_n\|} - w\right\| + \|w_n\|\cdot\left\|w-\tfrac{m_{\kappa(n)}}{\|w_n\|}z\right\| \le 3\delta'\|w_n\|\]
and as a result, by taking $\delta'$ sufficiently small relative to $\delta$, we may ensure that
\[(\nu-\delta)\|w_n\| \le (\nu-\delta')\|m_{\kappa(n)}z\| - \gamma\big(1+\|w_n-m_{\kappa(n)}z\|\big).\]

Let $D(n)$ be the quantity on the right-hand side above, i.e.
\[D(n):=(\nu-\delta')\|m_{\kappa(n)}z\| - \gamma\big(1+\|w_n-m_{\kappa(n)}z\|\big).\]

Since, on $\Upsilon$, by our choice of $w_n$ we have
\[\hit{\{0\}}{w_n} < (\nu - \delta)\|w_n\| \le D(n),\]
we deduce that on $\Upsilon$, using Lemma \ref{lem:triangle_ineq_chains},
\[\hit{\{0\}}{m_{\kappa(n)}z} \le D(n) + \tau^{\{w_n\}}_{D(n)}(m_{\kappa(n)z}).\]
Since $m_{\kappa(n)}z$ is inwardly $(\nu-\delta')m_{\kappa(n)}\|z\|$-good, we have
\[\tau^{\{w_n\}}_{D(n)}(m_{\kappa(n)z}) \le \gamma(1+\|w_n-m_{\kappa(n)}z\|)\]
and therefore, on $\Upsilon$,
\[\hit{\{0\}}{m_{\kappa(n)}z} \le D(n) + \gamma(1+\|w_n-m_{\kappa(n)}z\|) = (\nu-\delta')\|m_{\kappa(n)}z\|. \]
\[\hit{\{0\}}{w_n} \le (\nu+\delta')n_{\kappa(n)}\|z\| + \gamma\big(1+\|w_n - n_{\kappa(n)}z\|\big).\]
But our choice of $J$ ensured that for all $j\ge J$,
\[\hit{\{0\}}{m_j z} > (\nu-\delta')\|m_j z\|,\]
giving a contradiction.

We deduce that for any $\delta>0$, there cannot exist a sequence $(w_n)_{n\ge 1}$ in $\rtwo$ with $\|w_n\|\to+\infty$ and
\[\frac{\hit{\{0\}}{w_n}}{\|w_n\|} < \nu - \delta;\]
and thus for any $\eps>0$ we must have
$S^{\{0\}}_t\subseteq \bcal\big(0,(\nu^{-1}+\eps)t\big)$ for all large $t$.
\end{proof}

\begin{proof}[Proof of Theorem \ref{thm:ball_shape}]
In the case that $E=\{0\}$, the proof of Theorem \ref{thm:ball_shape} is a trivial combination of Propositions \ref{prop:no_point_hit_late} and \ref{prop:no_point_hit_early}; by translation invariance we also obtain that for any $z\in\rtwo$ and any $\eps>0$,
\[\P\left( \bcal\big(z,(\nu^{-1}-\eps)t\big)\subseteq S^{\{z\}}_t \subseteq \bcal\big(z,(\nu^{-1}+\eps)t\big) \,\,\text{ for all large } t\right) =1.\]
For a more general compact set $E$, take some $z\in E$; then clearly
\[\bcal(z,(\nu^{-1}-\eps)t)\subseteq S_t^{\{z\}} \subseteq S_t^E,\]
and also
\[S_t^E \subseteq S_{t+c}^{\{z\}} \subseteq \bcal\big(z,(\nu^{-1}+\eps)(t+c)\big)\]
where $c$ is a sufficiently large constant that $E\subseteq S_c^{\{z\}}$. This allows us to conclude.
\end{proof}

\subsection{\texorpdfstring{Controlling the difference between $\hit{\hcal}{x,0}$ and $\bulk{\hcal}{x,0}$, and between $\hit{\{0\}}{\hcal^x}$ and $\bulk{\{0\}}{x,0}$: proof of Theorem \ref{thm:new_speed_growth_bulk}}{Controlling the difference between hitting times and bulk coverage times: proof of Theorem \ref{thm:new_speed_growth_bulk}}}\label{sec:new_control_diff}

We now turn to proving Theorem \ref{thm:new_speed_growth_bulk}. We begin with the first part; the second part will follow a very similar argument. Recall that we want to ensure that for large enough $\beta$, the whole segment of the $x$-axis between the origin and $(x,0)$ is covered by time $\hit{\hcal}{x,0} + \beta\sqrt{x}$ with high probability.

To do this, we will use Theorem \ref{thm:geodesic_plane_to_point} to show that with high probability there exists a chain of $\Pi^{(\hit{\hcal}{x,0})}$-reproduction events leading from $(x,0)$ to $\hcal$ in time $\hit{\hcal}{x,0}$ that does not wander too far from the $x$-axis. We call this the \emph{right-to-left chain}. We will then show that the extra time $\beta\sqrt{x}$ is enough to ensure that all the other points between $0$ and $(x,0)$ are covered too, by constructing chains of reproduction events that ``join'' to the one already mentioned.

Our next goal is to show that if the right-to-left chain exists, then we can find chains from all other points on the $x$-axis between $0$ and $(x,0)$ in an additional $\beta\sqrt x$ time that ``join up'' with the right-to-left chain. To carry out the ``joining up'' step, our main tool is the following simple observation, which essentially says that if we know one chain of reproduction events has hit the half-plane, then any other chain that intersects it can be extended to a chain that will also hit the half-plane.

\begin{lem}\label{lem:connect_chains}
Suppose that $0\le t_1\le t_2$ and $z_1,z_2\in\rtwo$. Take $\mathbf{C}^{[1]}\in\ccal^{(z_1)}(\Pi^{(t_1)})$ and $\mathbf{C}^{[2]}\in\ccal^{(z_2)}(\Pi^{(t_2)})$. Suppose that $\mathbf{C}^{[1]}$ hits the half-plane $\mathcal{H}$ by time $t_1$, i.e.
\begin{equation*}
    \bcal\left( Z_{t_{1}}^{[1]}, R_{t_{1}}^{[1]} \right) \cap \mathcal{H} \neq \emptyset
\end{equation*}
and that the two chains meet at earlier times, i.e.~there exist $t'_{1}, t'_{2} \geq 0$ such that $t_{2} - t'_{2} > t_{1} - t'_{1} \ge 0$ and 
\begin{equation*}
    \bcal\left( Z_{t'_{1}}^{[1]}, R_{t'_{1}}^{[1]} \right) \cap \bcal\left( Z_{t'_{2}}^{[2]}, R_{t'_{2}}^{[2]} \right) \neq \emptyset \text{ and } \mathbf{C}_{t'_{1}}^{[1]} \neq \mathbf{C}_{t'_{1}-}^{[1]}.
\end{equation*}
Then there exists $\mathbf{C}^{[m]}\in\ccal^{(z_2)}(\Pi^{(t_2)})$ such that 
\begin{equation*}
    \bcal\left( Z_{t_{2}}^{[m]}, R_{t_{2}}^{[m]} \right) \cap \mathcal{H} \neq \emptyset.
\end{equation*}
\end{lem}

\begin{proof}
We construct the desired chain of $\Pi^{(t_2)}$-reproduction events $\mathbf{C}^{[m]}$ as follows. 
\begin{itemize}
    \item For all $0 \leq t \leq t'_{2}$, we set $\mathbf{C}_{t}^{[m]} = \mathbf{C}_{t}^{[2]}$. 
    \item For all $t'_{2} \leq t < t_{2} - t_{1} + t'_{1}$, we set $\mathbf{C}_{t}^{[m]} = \mathbf{C}_{t'_{2}}^{[m]}$. 
    \item For all $t \geq t_{2} - t_{1} + t'_{1}$, we set $\mathbf{C}_{t}^{[m]} = \mathbf{C}_{t - t_{2} + t_{1}}^{[1]}$.\qedhere
\end{itemize}
\end{proof}

Lemma \ref{lem:connect_chains} tells us that, if we already have a right-to-left chain $\mathbf{C}^{[1]}$ of $\Pi^{(\hit{\hcal}{x,0})}$-reproduction events leading from~$(x,0)$ to $\mathcal{H}$, then for each $x'\in(0,x)$ we can obtain a chain of $\Pi^{(\hit{\hcal}{x,0}+\beta\sqrt x)}$-reproduction events leading from~$(x',0)$ to $\mathcal{H}$ by simply finding another chain $\mathbf{C}^{[2]}$ from $(x',0)$ that intersects $\mathbf{C}^{[1]}$ before time $\beta\sqrt x$. Theorem \ref{thm:geodesic_plane_to_point} allows us to ask that the right-to-left chain $\mathbf{C}^{[1]}$ does not move too far from the $x$-axis, and therefore $\mathbf{C}^{[2]}$ does not have to travel far to intersect it. A suitable chain $\mathbf{C}^{[2]}$ can be constructed via the methods in Section \ref{subsec:slow_coverage}, except that since we do not know whether $\mathbf{C}^{[1]}$ will be above or below the $x$-axis when it passes near $(x',0)$, we will in fact need two chains to cover a region both above and below the $x$-axis.  We now carry out the details.

The following easy consequence of the work done in Section \ref{subsec:slow_coverage} says that starting from any point in a sufficiently small segment of the $x$-axis, if we wait time $Ay - 1$ for sufficiently large $A$, then we will see a chain of $\Pi^{(t)}$-reproduction events that goes above $y$ with overwhelming probability. (We use $\alpha y -1$ rather than $\alpha y$ simply because this is how we will apply this corollary later.)

\begin{lem}\label{lem:new_up_down_chains}
Fix $\delta,\eta\in(0,1]$ such that $\mu((3\delta,\infty))\ge \eta$. For all $i\in\mathbb{Z}$, $t\ge 0$, $y\ge 3$ and $\alpha\ge \frac{6}{\delta^2\eta}\vee 2$,
\[\P\left( \exists x'\in[i\delta,(i+1)\delta] : \nexists\mathbf{C} = ((X_s,Y_s),R_s)_{s\ge 0} \in \ccal^{((x',0))}(\Pi^{(t)}) \text{ with } Y_{\alpha y-1} > y\right) \le e^{-\delta\eta \alpha y/2}.\]
\end{lem}

\begin{proof}
Since $\Pi^{(t)}$ has the same distribution as $\Pi$, and using translation and rotation invariance of $\Pi$, we have
\begin{multline*}
\P\left( \exists x'\in[i\delta,(i+1)\delta] : \nexists\mathbf{C} = ((X_s,Y_s),R_s)_{s\ge 0} \in \ccal^{((x',0))}(\Pi^{(t)}) \text{ with } Y_{\alpha y-1} > y\right)\\
= \P\left( \exists x'\in[0,\delta] : \nexists\mathbf{C} = ((X_s,Y_s),R_s)_{s\ge 0} \in \ccal^{((x',0))}(\Pi) \text{ with } X_{\alpha y-1} > y\right).
\end{multline*}
In other words, asking for a chain of $\Pi^{(t)}$-reproduction events from a segment of the $x$-axis to height $y$ has the same probability as asking for a chain of $\Pi$-reproduction events from a segment of the $y$-axis to the half-plane $\hcal^y$. But, using the definition \eqref{eq:slow_covering_chain} of the slow coverage chain and Lemma \ref{lem:properties_sequence_reproduction_events} (ii), this event entails that $\overrightarrow{T}_{\lfloor y/\delta\rfloor + 2} > \alpha y-1$. And letting $x=y+2\delta$, we have
\[\P\left(\overrightarrow{T}_{\lfloor y/\delta\rfloor +2} > \alpha y-1\right) = \P\left(\overrightarrow{T}_{\lfloor x/\delta\rfloor} > \alpha y-1\right) \le \P\left(\overrightarrow{T}_{\lceil x/\delta\rceil} > \alpha y-1\right).\]
Since $\alpha\ge 2$, $\delta\le 1$ and $y\ge 3$, it is easy to check that $\alpha y-1 \ge \alpha x/2$. Thus, putting the ingredients above together, we have that
\[\P\left( \exists x'\in[i\delta,(i+1)\delta] : \nexists\mathbf{C} = ((X_s,Y_s),R_s)_{s\ge 0} \in \ccal^{((x',0))}(\Pi^{(t)}) \text{ with } Y_{\alpha  y-1} > y\right) \le  \P\left(\overrightarrow{T}_{\lceil x/\delta\rceil} > \frac{\alpha x}{2}\right)\]
and by Lemma \ref{lem:tail_estimate_slow_covering_chain} we deduce the result.
\end{proof}

We now take a union bound over points on the $x$-axis and integer times $j$.

\begin{cor}\label{cor:new_up_down_chains}
Fix $\delta,\eta\in(0,1]$ such that $\mu((3\delta,\infty))\ge \eta$. For all $\alpha\ge \frac{6}{\delta^2\eta}\vee 2$ and $A,B\ge 1$ and $x\ge 9$,
\begin{align*}
\P\left( \exists x'\in[0,x], j\in\{0,\ldots,\lceil Bx\rceil\} : \nexists\mathbf{C} = ((X_s,Y_s),R_s)_{s\ge 0} \in \ccal^{((x',0))}(\Pi^{(j)}) \text{ with } Y_{\alpha Ax^{1/2}-1} > Ax^{1/2}\right)&\\
\le (Bx+2)(x/\delta+1)e^{-\delta\eta \alpha A x^{1/2}/2}.&
\end{align*}
\end{cor}

\begin{proof}
By a union bound,
\begin{align*}
&\P\left( \exists x'\in[0,x], j\in\{0,\ldots,\lceil Bx\rceil\} : \nexists\mathbf{C} = ((X_s,Y_s),R_s)_{s\ge 0} \in \ccal^{((x',0))}(\Pi^{(j)}) \text{ with } Y_{\alpha Ax^{1/2}-1} > Ax^{1/2}\right)\\
&\le \sum_{i=0}^{\lfloor x/\delta\rfloor} \sum_{j=0}^{\lceil Bx\rceil} \P\left( \exists x'\in[i\delta,(i+1)\delta] : \nexists\mathbf{C} = ((X_s,Y_s),R_s)_{s\ge 0} \in \ccal^{((x',0))}(\Pi^{(j)}) \text{ with } Y_{\alpha Ax^{1/2}-1} > Ax^{1/2}\right)
\end{align*}
and then applying Lemma \ref{lem:new_up_down_chains} with $t=j$ and $y= Ax^{1/2}$ (which is larger than $3$ since $A\ge 1$ and $x\ge 9$) gives the result.
\end{proof}

We can now put our ingredients together to prove Theorem \ref{thm:new_speed_growth_bulk}. In a few words, Theorem \ref{thm:geodesic_plane_to_point} provides a right-to-left chain from $(x,0)$ to $\hcal$ that remains within a strip about the $x$-axis; Corollary \ref{cor:new_up_down_chains} provides upwards and downwards chains from every point on the $x$-axis that leave the strip, and therefore must intersect the right-to-left chain; and Lemma \ref{lem:connect_chains} ensures that we can tie the upwards and downwards chains together with the right-to-left chain to create a path from every point on the $x$-axis to $\hcal$.

\begin{proof}[Proof of Theorem \ref{thm:new_speed_growth_bulk}: first part]
As above, fix $\delta,\eta\in(0,1]$ such that $\mu((3\delta,\infty))\ge \eta$, and then fix $\alpha = \frac{6}{\delta^2\eta}\vee 2$. Define the event 
\begin{multline*}
\ucal(x,A,B) := \Big\{ \forall x'\in[0,x],\,\, \forall j\in\{0,\ldots,\lceil Bx\rceil\}, \\
\exists\mathbf{C} = ((X_s,Y_s),R_s)_{s\ge 0} \in \ccal^{((x',0))}(\Pi^{(j)}) \text{ with } Y_{\alpha Ax^{1/2}-1} > Ax^{1/2}\\
\text{and } \exists\mathbf{C}' = ((X'_s,Y'_s),R'_s)_{s\ge 0} \in \ccal^{((x',0))}(\Pi^{(j)}) \text{ with } Y_{\alpha Ax^{1/2}-1} < -Ax^{1/2}\Big\}.
\end{multline*}
In words, $\ucal(x,A,B)$ holds if for every $x'\in[0,x]$, there exists a chain of $\Pi^{(j)}$-reproduction events leading upwards from $(x',0)$ to height $Ax^{1/2}$, and another leading downwards from $(x',0)$ to depth $-Ax^{1/2}$, both in time $\alpha Ax^{1/2}-1$; and this occurs starting at every integer time $j$ up to $\lceil Bx\rceil$ (backwards in time from $j$, since we are using $\Pi^{(j)}$). We know from Corollary \ref{cor:new_up_down_chains} that for any $\eps>0$ and any fixed $A$ and $B$, we can make $\P(\ucal(x,A,B)^c)<\eps$ by taking $x$ sufficiently large.

Now fix $A_\eps,x_\eps$ as in Theorem \ref{thm:geodesic_plane_to_point} and let $\acal(x)$ be the event in Theorem \ref{thm:geodesic_plane_to_point}, which when written in terms of chains of reproduction events is
\begin{multline*}
\acal(x) := \Big\{\exists C\in\ccal^{((x,0))}\big(\Pi^{(\hit{\hcal}{x,0})}\big) : \bcal(Z_s,R_s)\subseteq \lcal^{(A\sqrt x)} \, \forall s\le \hit{\hcal}{x,0}\\
\text{ and } \bcal\big(Z_{\hit{\hcal}{x,0}},R_{\hit{\hcal}{x,0}}\big)\cap \hcal \neq \emptyset\Big\}.
\end{multline*}
Theorem \ref{thm:geodesic_plane_to_point} says that for any $x\ge x_\eps$ we have
\[\P(\acal(x)^c)\le \eps.\]
In words, with high probability there is a right-to-left chain of $\Pi^{(\hit{\hcal}{x,0})}$-reproduction events from $(x,0)$ to $\hcal$ that does not stray too far from the origin.

Further, by Theorem \ref{thm:hitting_times}, we can choose $B_\eps\ge 1$ and increase $x_\eps$ if necessary so that for all $x\ge x_\eps$,
\[\P(\hit{\hcal}{x,0} > B_\eps x) \le \eps;\]
and by increasing $x_\eps$ further if necessary, we may assume that for all $x\ge x_\eps$, we have
\[1+A_\eps\alpha\sqrt{x} \le x.\]

With the parameters chosen above, if $x>x_\eps$, then on the event 
\begin{equation}\label{eq:key_event}
\left\{\hit{\hcal}{x,0} \le B_\eps x\right\}\cap \acal(x)\cap \ucal(x,A_\eps,B_\eps+1),
\end{equation}
the following all occur:
\begin{itemize}
\item There is a right-to-left chain, i.e.~a chain of $\Pi^{(\hit{\hcal}{x,0})}$-reproduction events from $(x,0)$ to $\hcal$ that remains within $\lcal^{(A_\eps\sqrt x)}$.
\item We have $\lceil \hit{\hcal}{x,0}\rceil\le B_\eps x + 1$, and therefore since $x\ge x_\eps$, we know that
\[\lceil \hit{\hcal}{x,0}\rceil + \lceil A_\eps\alpha\sqrt{x}\rceil \le (B_\eps + 1)x.\]
\item Let $T=\lceil\hit{\hcal}{x,0}\rceil + \lceil A_\eps\alpha\sqrt{x}\rceil$; then for each $x'\in[0,x]$, there is a chain of $\Pi^{(T)}$-reproduction events leading upwards from $(x',0)$ to height $Ax^{1/2}$, and another leading downwards from $(x',0)$ to depth $-Ax^{1/2}$, both in time $\alpha Ax^{1/2}-1$.
\end{itemize}
By Lemma \ref{lem:connect_chains}, we deduce that for each $x'\in[0,x]$, there is also a a chain of $\Pi^{(T)}$-reproduction events starting from $(x',0)$ and hitting $\hcal$ by time $T$. In other words, on the event \eqref{eq:key_event}, and for $x>x_\eps$, the whole $x$-axis up to $(x,0)$ is contained within $S^\hcal_T$; or in yet other terms, $\bulk{\hcal}{x,0} \le T$. We have therefore shown that
\[\left\{\hit{\hcal}{x,0} \le B_\eps x\right\}\cap \acal(x,A_\eps)\cap \ucal(x,A_\eps,B_\eps+1) \subseteq \left\{\bulk{\hcal}{x,0} \le \lceil\hit{\hcal}{x,0}\rceil + \lceil A_\eps\alpha\sqrt{x}\rceil \right\}.\]
Letting $\beta = A_\eps\alpha+1$, since we have
\[\lceil\hit{\hcal}{x,0}\rceil + \lceil A_\eps\alpha\sqrt{x}\rceil \le \hit{\hcal}{x,0} + A_\eps\alpha\sqrt{x} + 2 \le \hit{\hcal}{x,0} + \beta\sqrt{x}\]
for sufficiently large $x$, and the probability of \eqref{eq:key_event} is at least $1-3\eps$, the proof of the first part of the theorem is complete.
\end{proof}

\begin{proof}[Proof of Theorem \ref{thm:new_speed_growth_bulk}: second part]
We proceed very similarly to above. Again fix $\delta,\eta\in(0,1]$ such that $\mu((3\delta,\infty))\ge \eta$, and then fix $\alpha = \frac{6}{\delta^2\eta}\vee 2$. This time, instead of a right-to-left chain, we have a left-to-right chain which is precisely the geodesic provided by Theorem \ref{thm:geodesics} (i). That is, we let $A_\eps$ be as in (the precise statement of) Theorem \ref{thm:geodesics}(i) and define
\begin{multline*}
\acal'(x) := \Big\{\exists C\in\ccal^{(0)}(\Pi) : \bcal(Z_s,R_s)\subseteq \lcal^{(A_\eps\sqrt x)} \, \forall s\le \hit{\{0\}}{\hcal^x}\\
\text{ and } \bcal\big(Z_{\hit{\{0\}}{\hcal^x}},R_{\hit{\{0\}}{\hcal^x}}\big)\cap \hcal^x \neq \emptyset\Big\},
\end{multline*}
in the knowledge that Theorem \ref{thm:geodesics} (i) ensures that $\P(\acal'(x))\le \eps$.

We similarly define
\begin{multline*}
\ucal'(x,A,B) := \Big\{ \forall x'\in[0,x],\,\, \forall j\in\{0,\ldots,\lceil Bx\rceil\}, \\
\hspace{15mm}\exists\mathbf{C} = ((X_s,Y_s),R_s)_{s\ge 0} \in \ccal^{((x',Ax^{1/2}))}(\Pi_{(j)}) \text{ with } (x',0)\in \mathbf{C}_{\alpha Ax^{1/2}-1}\\
\text{and } \exists\mathbf{C}' = ((X'_s,Y'_s),R'_s)_{s\ge 0} \in \ccal^{((x',-Ax^{1/2}))}(\Pi_{(j)}) \text{ with } (x',0)\in \mathbf{C}_{\alpha Ax^{1/2}-1}\Big\},
\end{multline*}
the event that for each integer $j\le Bx$ and each $x'\in[0,x]$ there are upwards and downwards chains of $\Pi_{(j)}$-reproduction events started from the top and bottom of the strip $\lcal^{(Ax^{1/2})}$ and finishing at $(x',0)$ at time $\alpha Ax^{1/2}-1$; here we recall that $\Pi_{(j)}$ is the time shift of $\Pi$ defined in \eqref{eq:pi_timeshift}. It is easy to show, by reversing the chains in Lemma \ref{lem:new_up_down_chains} and taking a union bound as in Corollary \ref{cor:new_up_down_chains}, that for any $\eps>0$ and any fixed $A$ and $B$, we can make $\P(\ucal'(x,A,B)^c)<\eps$ by taking $x$ sufficiently large.

We also know from Theorem \ref{thm:hitting_times} that $\P\left(\hit{\{0\}}{\hcal^x} > Bx\right)\le \eps$ for $B=B_\eps$ sufficiently large. On the event
\begin{equation}\label{eq:key_event_2}
\left\{\hit{\{0\}}{\hcal^x} \le B_\eps x\right\}\cap \acal'(x,A_\eps)\cap \ucal'(x,A_\eps,B_\eps+1),
\end{equation}
if $x$ is sufficiently large then the following all occur:
\begin{itemize}
\item There is a left-to-right chain, i.e.~a chain of $\Pi$-reproduction events from $0$ to $\hcal^x$ that remains within $\lcal^{(A_\eps\sqrt x)}$.
\item We have $\lceil\hit{\{0\}}{\hcal^x}\rceil\le B_\eps + 1$, and therefore for $x$ large, we know that
\[\lceil\hit{\{0\}}{\hcal^x}\rceil + \lceil A_\eps\alpha\sqrt{x}\rceil \le (B_\eps + 1)x.\]
\item For each $x'\in[0,x]$, there is a chain of $\Pi_{(\lceil\hit{\{0\}}{\hcal^x}\rceil)}$-reproduction events leading downwards from height $Ax^{1/2}$ to $(x',0)$, and another leading upwards from depth $-Ax^{1/2}$ to $(x',0)$, both in time $\alpha Ax^{1/2}-1$.
\end{itemize}
For each $x'\in[0,x]$, by concatenating the left-to-right chain with either the downwards or upwards chain at time $\lceil\hit{\{0\}}{\hcal^x}\rceil$, we obtain a chain of $\Pi$-reproduction events starting from $0$ and hitting $(x',0)$ by time $T':=\lceil\hit{\{0\}}{\hcal^x}\rceil + \lceil A_\eps\alpha\sqrt{x}\rceil$. By Corollary \ref{cor:forwards_St}, we have $(x',0)\in S^{\{0\}}_{T'}$. Since this holds for all $x'\in[0,x]$, we deduce that on the event \eqref{eq:key_event_2}, for sufficiently large $x$, we have $\bulk{\{0\}}{x,0} \le T'$. Since for sufficiently large $x$, the probability of \eqref{eq:key_event_2} is at least $1-3\eps$, we are able to conclude.
\end{proof}

\section{\texorpdfstring{Equivalence of the different definitions of the $\infty$-parent SLFV}{Equivalence of the different definitions of the infinite-parent SLFV}}\label{sec:equivalence_definitions}

The goal of this section is to provide a suitable condition under which we can move seamlessly between two representations of the $\infty$-parent SLFV: the set-valued process $(S_t)_{t\ge 0}$ defined earlier, and the measure-valued process as defined e.g.~in~\cite{louvet2023stochastic,louvet2023measurevalued}. This is non-trivial, since---as we have seen---the set-valued process started from a point will grow at linear speed, whereas the measure-valued process started from any set of zero Lebesgue measure will not grow at all. However, we are able to define a simple condition---which need be checked only for the initial state of the process---under which the two representations do agree for all time. In particular, under this condition, the results previously obtained on the measure-valued process are also true for the set-valued process.

\subsection{\texorpdfstring{Definition and persistence of condition ($\triangle$)}{Definition and persistence of condition triangle}}\label{sec:triangle}

For any measurable set $A \subseteq \rtwo$, we say that \textit{$A$ satisfies condition}~($\triangle$) if
\begin{equation*}
    \Vol\left(
\bcal(z,\eps) \cap A
    \right) > 0 \quad \text{ for all } z \in E \text{ and all } \eps > 0. 
\end{equation*}
For instance, any open set of $\rtwo$ satisfies condition~($\triangle$), while a set of null Lebesgue measure does not satisfy it. If the initial state $E$ of our $\infty$-parent SLFV does not satisfy condition~($\triangle$), then it is \emph{a priori} possible for the intersection of a reproduction event with $E$ to have zero volume, which would immediately lead to a difference between the set-valued version from Definition \ref{defn:infty_parent_slfv} and any measure-valued version. Thus something like condition~($\triangle$) is necessary in order to prove equality in distribution between the set-valued and measure-valued definitions. Our aim now is to show that condition~($\triangle$) is in fact sufficient for this purpose.

The first step is to prove Lemma \ref{lem:triangle_valued_process}, which said that a set-valued $\infty$-parent SLFV process started from a set satisfying condition ($\triangle$) will continue to satisfy condition ($\triangle$) for all times.

\theoremstyle{plain}
\newtheorem*{lemmatriangle}{Lemma~\ref{lem:triangle_valued_process}}
\begin{lemmatriangle}
Let $E \subseteq \rtwo$ be measurable, and let $(S^E_{t})_{t \geq 0}$ be the (set-valued) $\infty$-parent SLFV with initial condition $S_{0} = E$. If $E$ satisfies condition~($\triangle$), then $S^E_{t}$ satisfies condition~($\triangle$) for all $t \geq 0$.
\end{lemmatriangle}

\begin{proof}
Suppose that $E$ satisfies condition ($\triangle$). Let $t \geq 0$ and suppose that $z \in S^E_{t}$. We distinguish two cases. If $z \in S^E_0 = E$, as $(S^E_{s})_{s \geq 0}$ is an increasing set and by condition~($\triangle$), for all $\eps > 0$, 
\begin{equation*}
    \Vol\left(
\bcal(z,\eps) \cap S_{t}^{E}
    \right) \geq \Vol\left(
\bcal(z,\eps) \cap S_{0}^{E}
    \right) > 0.
\end{equation*}
Conversely, if $z \notin S^E_{0}$, let $(t',z',r') \in \Pi$, $0 \leq t' \leq t$ be the last reproduction event to affect~$z$; that is, such that
$z \in \bcal(z',r' )$. Since $z\in S^E_t$, we know that such a reproduction event exists. For all $\eps > 0$, 
\begin{equation*}
    \Vol\left(
\bcal(z,\eps) \cap S_{t}^{E}
    \right) \geq \Vol\left(
\bcal(z,\eps) \cap \bcal(z', r ')
    \right) > 0, 
\end{equation*}
which allows us to conclude. 
\end{proof}

Note that this result stays true if we construct $S^E$ using a Poisson point process on $\rmath \times Q \times (0,+\infty)$ with $Q \subseteq \rtwo$ a square, instead of $\rmath\times\rmath\times(0,+\infty)$. We will make use of this observation later. (There is nothing particularly special about squares here, but later we will need some well-behaved sequence of sets growing to cover the whole of $\mathbb{R}^2$, and squares are a convenient choice.)

We now show how to re-interpret $S^E$ as measure-valued process. Let $\widetilde{\mcal}_{\lambda}$ be the set of measures $M$ on $\rtwo \times \{0,1\}$ whose marginal distribution over $\rtwo$ is Lebesgue measure, or in other words, of the form
\begin{equation}\label{eqn:decomposition_M}
M(dz,A) = \left(\omega_M(z)\ind_{\{1\in A\}} + (1-\omega_M(z))\ind_{\{0\in A\}}\right)dz
\end{equation}
for $z\in\mathbb{R}^2$ and $A\subseteq\{0,1\}$, with $\omega_{M} : \rtwo \to [0,1]$ measurable. Note that the function $\omega_{M}$ is only defined up to a Lebesgue-null set.

 Let $\mcal_{\lambda} \subseteq \widetilde{\mcal}_{\lambda}$ be the set of all measures $M \in \widetilde{\mcal}_{\lambda}$ such that there exists a $\{0,1\}$-valued function~$\omega_{M}$ satisfying (\ref{eqn:decomposition_M}). We will refer to any such function~$\omega_{M}$ as a \textit{density of $M$}. We endow $\mcal_{\lambda}$ with the vague topology. We denote by $D_{\mcal_{\lambda}}[0,+\infty)$ the space of all c\`adl\`ag $\mcal_{\lambda}$-valued paths, endowed with the standard Skorokhod topology.

 We can represent $S^E$ by a measure-valued process living in $\mcal_{\lambda}$ by writing
 \[M^E_{t}(dz,A) := \left(
\mathds{1}_{S^E_{t}}(z) \ind_{\{1\in A\}} + \left(
1 - \mathds{1}_{S^E_{t}}(z)
\right) \ind_{\{0\in A\}}
\right)dz\]
for $z\in\mathbb{R}^2$ and $A\subseteq \{0,1\}$. It is straightforward to check (see \cite[Section 3.1]{louvet2023stochastic}, in particular Definition 3.4 and the proof that follows) that for any measurable initial condition $E$, we have $(M^E_t)_{t\ge 0} \in D_{\mcal_{\lambda}}[0,+\infty)$.

Our goal is to show that if $E$ satisfies condition ($\triangle$) then the process $(M^E_{t})_{t \geq 0}$ is precisely the measure-valued $\infty$-parent SLFV studied in \cite{louvet2023stochastic,louvet2023measurevalued}, implying that the set-valued and measure-valued $\infty$-parent SLFVs correspond to the same process up to a change of state space. 

\subsection{\texorpdfstring{Definition of the measure-valued $\infty$-parent SLFV}{Definition of the measure-valued infinite-parent SLFV}}
One of the ways to define the measure-valued $\infty$-parent SLFV 
studied in \cite{louvet2023stochastic,louvet2023measurevalued}
is as the unique solution to a well-posed martingale problem. One benefit of this construction is that it also provides a characterisation of the process, which we can use later to show that it equals the process $(M^E_t)_{t\ge 0}$ defined above. 

First we introduce the test functions over which the operator associated to the martingale problem is defined. Let $C_{c}(\rtwo)$ be the space of continuous and compactly supported functions $f \in \rtwo \to \rmath$. 
Let $C^{1}(\rmath)$ be the space of continuously differentiable functions $F : \rmath \to \rmath$. For all $M \in \mcal_{\lambda}$, we recall that $\omega_{M}$ denotes a density of $M$ chosen in an arbitrary way, and is $\{0,1\}$-valued. 

For all $f \in C_{c}(\rtwo)$ and $M \in \mcal_{\lambda}$, we set
\begin{equation*}
    \langle
M,f
    \rangle := \int_{\rtwo} f(z)\omega_{M}(z)dz. 
\end{equation*}
Note that if there exists a measurable set $S \subseteq \rtwo$ such that for all $z\in\rtwo$ and $A\subseteq\{0,1\}$,
\begin{equation}\label{eqn:measure_as_set}
M(dz,A) := \left(
\mathds{1}_{S}(z)\ind_{\{1\in A\}} + \left(
1-\mathds{1}_{S}(z)
\right)\ind_{\{0\in A\}}\right)dz,
\end{equation}
then
\begin{equation*}
    \langle
M,f
    \rangle := \int_{\rtwo} f(z)\mathds{1}_{S}(z)dz = \int_S f(z) dz. 
\end{equation*}
Therefore, for any measurable $S\subseteq\rtwo$ we also write
\begin{equation*}
    \langle
S,f \rangle :=  \int_S f(z) dz. 
\end{equation*}
The test functions that we will consider are of the form $\Psi_{F,f} : \mcal_{\lambda} \to \rmath$, for $f \in C_{c}(\rtwo)$ and $F \in C^{1}(\rmath)$, with 
\begin{equation*}
    \Psi_{F,f}(M) := F\left(
\langle M,f \rangle 
    \right)
\end{equation*}
for any $M \in \mcal_{\lambda}$. If $S$ is a measurable subset of $\rtwo$ then we also write
\begin{equation*}
    \Psi_{F,f}(S) := F\left(
\langle S,f \rangle
    \right);
\end{equation*}
this slight abuse of notation is consistent in the sense that if $M$ and $S$ satisfy~(\ref{eqn:measure_as_set}), then 
\begin{equation*}
    \Psi_{F,f}(M) = \Psi_{F,f}(S).
\end{equation*}
We will use this observation later when showing that $(M^E_t)_{t\ge 0}$ is a solution to the martingale problem characterising the measure-valued $\infty$-parent SLFV. 

For all $F \in C^{1}(\rmath)$, $f \in C_{c}(\rtwo)$ and $M \in \mcal_{\lambda}$, we define the operator $\lcal_\mu^\infty$ by setting
\begin{equation}\label{eq:Linftymudef}
\lcal_{\mu}^{\infty}\Psi_{F,f}(M) := \int_{0}^{\text{R}_{0}}\int_{\rtwo} \left(
\Psi_{F,f}(M[z,r]) - \Psi_{F,f}(M)\right) dz\, \mu(dr)
\end{equation}
where $M{[z,r]}\in\mcal_\lambda$ has density
\[\omega_{M{[z,r]}}(\cdot) = \begin{cases}
\omega_M(\cdot)\ind_{\bcal(z,r)^c}(\cdot) + \ind_{\bcal(z,r)}(\cdot) & \text{ if } \int_{\bcal(z,r)} \omega_{M}(z')dz' > 0\\
\omega_M(\cdot) & \text{ otherwise.}
\end{cases}\]

This operator is associated to a well-posed martingale problem, in the following sense. 
\begin{thm}\label{thm:measure_valued_slfv}
(\cite[Theorem 2.14]{louvet2023stochastic}) Let $\varphi \in \mcal_{\lambda}$. There exists a unique $D_{\mcal_{\lambda}}[0,+\infty)$-valued process $(\widetilde M_{t}^{\varphi})_{t \geq 0}$ such that $\widetilde M_{0}^{\varphi} = \varphi$ and, for all $F \in C^{1}(\rmath)$ and $f \in C_{c}(\rtwo)$, 
\begin{equation*}
\left(
\Psi_{F,f}\left(
\widetilde M_{t}^{\varphi}
\right) - \Psi_{F,f}\left(
\widetilde M_{0}^{\varphi}
\right) - \int_{0}^{t} \lcal_{\mu}^{\infty} \Psi_{F,f}\left(
\widetilde M_{s}^{\varphi}
\right)ds
\right)_{t \geq 0}
\end{equation*}
is a martingale. Moreover, the process $(\widetilde M_{t}^{\varphi})_{t \geq 0}$ is Markovian. 
\end{thm}
We refer to the process $\widetilde M^{\varphi} = (\widetilde M_{t}^{\varphi})_{t \geq 0}$ as the \textit{measure-valued $\infty$-parent SLFV} with initial condition~$\varphi$ and intensity~$\mu$. 

Our main goal for this section is to show the following result. 
\begin{thm}\label{thm:equality_slfvs}
Suppose that $E \subseteq \rtwo$ is measurable and satisfies condition~($\triangle$). Let $(S^E_{t})_{t \geq 0}$ be the set-valued SLFV with initial condition~$E$ and intensity~$\mu$, and as before, for all $t \geq 0$, $z\in\mathbb{R}^2$ and $A\subseteq\{0,1\}$, let
\[M^E_{t}(dz,A) := \left(
\mathds{1}_{S_{t}}(z)\ind_{\{1\in A\}} + \left(
1 - \mathds{1}_{S_{t}}(z)
\right)\ind_{\{0\in A\}}\right)dz.\]
Let $\varphi = M^E_0$. Then $(M^E_{t})_{t \geq 0}$ and $(\widetilde M^\varphi_{t})_{t \geq 0}$ are equal in distribution. 
\end{thm}
In order to show this result, we proceed as follows. First we define a version of the set-valued $\infty$-parent SLFV on an increasing sequence of squares which converges to~$\rtwo$, whose initial condition converges to~$E$. We obtain a sequence $(\widetilde M^{[n]})_{n \geq 1}$ of measure-valued processes. We then show that this sequence converges in distribution towards both $(M_{t}^{E})_{t \geq 0}$ and $(\widetilde M_{t}^{\varphi})_{t \geq 0}$. 

\subsection{\texorpdfstring{The set-valued $\infty$-parent SLFV restricted to a compact set}{The set-valued infinite-parent SLFV restricted to a compact set}}
\subsubsection{Definition of the process}
If the set-valued $\infty$-parent SLFV is defined on a square $Q \subseteq \rtwo$ rather than on $\rtwo$, then we can directly define the set-valued $\infty$-parent SLFV as a solution to a martingale problem, due to the observation that~$Q$ is affected by reproduction events at a finite rate. Formally, let $S^{Q,E} = (S_{t}^{Q,E})_{t \geq 0}$ be the set-valued $\infty$-parent SLFV constructed as in Definition~\ref{defn:infty_parent_slfv}, but using only reproduction events $(t,z,r) \in \Pi$ such that $z \in Q$ (i.e., whose centre is in~$Q$), and only adding the part of the reproduction event that intersects $Q$. Since the process $S^{Q,E}$ is then a finite-rate Markov process, it is a solution to the martingale problem associated to the operator $\lcal_{\mu}^{Q}$ such that for all $F \in C^{1}(\rmath)$, $f \in C_{c}(\rtwo)$ and $S \subseteq \rtwo$ measurable, 
\[\lcal_{\mu}^{Q}\Psi_{F,f}(S) := \int_{0}^{\text{R}_{0}}
\int_{Q}\left(F(\langle G_Q(S,z,r),f\rangle) - F(\langle S,f\rangle )\right)dz\,\mu(dr)\]
where
\[G_Q(S,z,r) := \begin{cases}
    S\cup(\bcal(z,r)\cap Q) & \text{ if } S\cap\bcal(z,r)\cap Q \neq \emptyset\\
    S & \text{ otherwise}.
\end{cases}\]

We note that in this martingale problem, reproduction events that are disjoint from the current state of the process have no effect. On the other hand, in the martingale problem associated to $\lcal_\mu^\infty$ in Theorem \ref{thm:measure_valued_slfv}, no effect occurs for reproduction events whose intersection with the current state of the process has measure zero. This distinction is emblematic of the (potential) difference between the set-valued and measure-valued processes, and motivates the following alternative method for ``growing'' a set $S$ using a reproduction event centred at $z$ with radius $r$:
\[\widetilde G_Q(S,z,r) := \begin{cases}
    S\cup(\bcal(z,r)\cap Q) & \text{ if } \text{Vol}(S\cap\bcal(z,r)\cap Q) > 0\\
    S & \text{ otherwise.}
\end{cases}\]
We show below that when starting from an initial state satisfying condition ($\triangle$) and restricting to a compact set, the two alternatives $G_Q$ and $\widetilde G_Q$ give rise to the same operator.

\begin{lem}\label{lem:operator_condition_triangle}
Let $S \subseteq \rtwo$ be measurable. If $S$ satisfies condition~($\triangle$), then for all $F \in C^{1}(\rmath)$ and $f \in C_{c}(\rtwo)$,
\[\lcal_{\mu}^{Q}\Psi_{F,f}(S) = \int_{0}^{\text{R}_{0}}
\int_{Q}\left(F(\langle \widetilde G_Q(S,z,r),f\rangle) - F(\langle S,f\rangle )\right)dz\,\mu(dr).\]
In particular, if $E$ satisfies condition~($\triangle$), then for all $t \geq 0$,
\[\lcal_{\mu}^{Q}\Psi_{F,f}(S_t^{Q,E}) = \int_{0}^{\text{R}_{0}}
\int_{Q}\left(F(\langle \widetilde G_Q(S_t^{Q,E},z,r),f\rangle) - F(\langle S_t^{Q,E},f\rangle )\right)dz\,\mu(dr).\]
\end{lem}

\begin{proof}
Suppose that $S$ satisfies condition~($\triangle$).
Our goal is to show that for all $0 < r \leq \text{R}_{0}$, the set
\begin{equation*}
    \hat{S}_{r} := \left\{
z' \in \rtwo : S\cap \bcal(z',r) \neq \emptyset \text{ but } \text{Vol}(S \cap \bcal(z',r)) = 0
    \right\}
\end{equation*}
is a porous set, and therefore has zero Lebesgue measure. We proceed via proving three claims.

\vspace{3mm}

\noindent
\textbf{Claim 1:} If $z\in \hat S_r$ and $w\in S$, then $\|z-w\|\ge r$.

\noindent
\textbf{Proof:} Take $w\in S$ and suppose that $z$ satisfies $\|z-w\|<r$. Then
\[\bcal\left(w,\tfrac{r-\|z-w\|}{2}\right) \subseteq \bcal(z,r),\]
so by condition ($\triangle$),
\[\text{Vol}\left(\bcal(z,r) \cap S
\right) \geq \text{Vol}\left(
\bcal\left(
w, \tfrac{r-\|z-w\|}{2}\right) \cap S
\right) > 0.\]
Thus $z\not\in\hat S_r$, establishing the claim.

\vspace{3mm}

\noindent
\textbf{Claim 2:} If $z\in\hat S_r$, then for any $\rho<r$, there exists $z'$ such that
\[\bcal(z',\rho/4) \subseteq \bcal (z,3\rho/4)\setminus \hat S_r.\]

\noindent
\textbf{Proof:} Take $z\in\hat S_r$. By definition there exists $z'_r \in S$ such that $z'_r\in B(z,r)$. (By Claim 1, we must in fact have $\|z'_r - z\|=r$.)

Given $\rho<r$, let $z' = z + \frac{\rho}{2r}(z'_r - z)$, and note that
\[z' - z'_r = (z-z'_r)\left(1-\frac{\rho}{2r}\right).\]
Thus, for any $w\in\bcal(z',\rho/4)$, we have
\begin{align*}
\|w-z'_r\| &\le \|w-z'\| + \|z'-z'_r\|\\
&\le \frac{\rho}{4} + \left(1-\frac{\rho}{2r}\right)\|z-z'_r\|\\
&= \frac{\rho}{4} + r - \frac{\rho}{2} < r.
\end{align*}
By Claim 1, $w\not\in \hat S_r$. Also
\[\|z-z'\| = \frac{\rho}{2r}\|z'_r - z\| = \frac{\rho}{2},\]
so $\bcal(z',\rho/4) \subseteq \bcal(z,3\rho/4)\setminus \hat S_r$ as required to complete the proof of Claim 2.

\vspace{3mm}

\noindent
\textbf{Claim 3:} If $z\not\in \hat S_r$, then for any $\rho<r$, there exists $z'$ such that 
\[\bcal(z',\rho/4) \subseteq \bcal (z,\rho)\setminus \hat S_r.\]

\noindent
\textbf{Proof:} Take $z\not\in\hat S_r$ and $\rho<r$. If $B(z,\rho/4)\cap \hat S_r = \emptyset$, then the proof is trivial: just take $z'=z$. So suppose that $B(z,\rho/4)\cap \hat S_r \neq \emptyset$. Take $\tilde z\in B(z,\rho/4)\cap \hat S_r$. By Claim 2, there exists $z'$ such that
\[\bcal(z',\rho/4)\subseteq \bcal(\tilde z,3\rho/4)\setminus \hat S_r \subseteq \bcal(z,\rho)\setminus \hat S_r.\]
This completes the proof of Claim 3. By combining Claims 2 and 3, we have shown that for \emph{any} $z$ and any $\rho<r$, we can find $z'$ such that
\[\bcal(z',\rho/4) \subseteq \bcal (z,\rho)\setminus \hat S_r,\]
which is precisely the statement that $\hat S_r$ is a porous set. We deduce that $\hat S_r$ has zero Lebesgue measure, completing the proof of the first part of the lemma. 
The second part is then a consequence of Lemma~\ref{lem:triangle_valued_process}.
\end{proof}

The lemma above will provide us with a way of moving between the set-valued and measure-valued $\infty$-parent SLFV \textit{on compacts}, as long as the initial state satisfies condition ($\triangle$). We now introduce a sequence of set-valued $\infty$-parent SLFVs, defined on an increasing sequence of compacts, with the aim of showing that the equivalence of the two definitions on compacts transfers to the whole of $\mathbb{R}^2$---again assuming that the initial state satisfies condition ($\triangle$).
Let $(Q_{n})_{n \geq 1}$ be an increasing sequence of squares in $\rtwo$ such that
\begin{equation*}
    \lim\limits_{n \to + \infty} Q_{n} = \rtwo. 
\end{equation*}
Let $E \subseteq \rtwo$ be a measurable set satisfying condition~($\triangle$). For all $n \geq 1$, let $S^{[n]} = (S_{t}^{[n]})_{t \geq 0}$ be the set-valued $\infty$-parent SLFV defined on the square~$Q_{n}$, with initial condition~$E\cap Q_n$. 

We now describe a sequence of measures and a sequence of martingale problems restricted to the sets $Q_n$, in line with what we have just seen; note in particular that the operator $\lcal^{[n]}_\mu$ defined below is identical to $\lcal^\infty_\mu$ defined in (\ref{eq:Linftymudef}), except for the restriction to $Q_n$. For all $n \geq 1$ and $t \geq 0$, set
\begin{equation*}
M_{t}^{[n]}(dz,A) := \left(
\mathds{1}_{S_{t}^{[n]}}(z) \ind_{\{1\in A\}} + \left(
1 - \mathds{1}_{S_{t}^{[n]}}(z)
\right)\ind_{\{0\in A\}}\right)dz, \,\,\,\, z\in\mathbb{R}^2,\, A\subseteq\{0,1\}.
\end{equation*}
For all $n \geq 1$, $F \in C^{1}(\rmath)$, $f \in C_{c}(\rtwo)$ and $M \in \mcal_{\lambda}$, set
\begin{equation*}
\lcal_{\mu}^{[n]}\Psi_{F,f}(M) := \int_{0}^{\text{R}_{0}}\int_{E_n} \left(
\Psi_{F,f}(M[n,z,r]) - \Psi_{F,f}(M)\right) dz\, \mu(dr)
\end{equation*}
where $M[n,z,r]\in\mcal_\lambda$ has density
\[\omega_{M^{[n]}_{z,r}}(\cdot) = \begin{cases}
\omega_M(\cdot)\ind_{(\bcal(z,r)\cap Q_n)^c}(\cdot) + \ind_{\bcal(z,r)\cap Q_n}(\cdot) & \text{ if } \int_{\bcal(z,r)\cap Q_n} \omega_{M}(z')dz' > 0\\
\omega_M(\cdot) & \text{ otherwise.}
\end{cases}\]
Thanks to Lemma~\ref{lem:operator_condition_triangle}, we have an easy task to show that the sequence $(M^{[n]})_{n \geq 1}$ satisfies the martingale problem associated to $\lcal_\mu^{[n]}$. 

\begin{lem}\label{lem:sequence_compact_slfv}
For all $n \geq 1$, $M^{[n]}$ is a solution to the well-posed martingale problem $(\lcal_{\mu}^{[n]}, \delta_{M_{0}^{[n]}})$. 
\end{lem}
\begin{proof}
Let $n \geq 1$, $F \in C^{1}(\rmath)$ and $f \in C_{c}(\rtwo)$. We need to show that
\begin{equation*}
    \left(
\Psi_{F,f}\left(
M_{t}^{[n]}
\right) - \Psi_{F,f}\left(
M_{0}^{[n]}
\right) - \int_{0}^{t} \lcal_{\mu}^{[n]}\Psi_{F,f}\left(
M_{s}^{[n]}
\right)ds\right)_{t \geq 0}
\end{equation*}
is a martingale. But for all $t \geq 0$, 
\begin{align*}
&\Psi_{F,f}\left(
M_{t}^{[n]}
\right) - \Psi_{F,f}\left(
M_{0}^{[n]}
\right) - \int_{0}^{t}\lcal_{\mu}^{[n]}\Psi_{F,f}\left(
M_{s}^{[n]}
\right)ds \\
&\hspace{30mm}= \Psi_{F,f}\left(
S_{t}^{[n]}
\right) - \Psi_{F,f}\left(
S_{0}^{[n]}
\right) - \int_{0}^{t}\lcal_{\mu}^{[n]}\Psi_{F,f}\left(
S_{s}^{[n]}
\right)ds \\
&\hspace{30mm}= \Psi_{F,f}\left(
S_{t}^{[n]}
\right) - \Psi_{F,f}\left(
S_{0}^{[n]}
\right) - \int_{0}^{t}\lcal_{\mu}^{Q_{n}}\Psi_{F,f}\left(
S_{s}^{[n]}
\right)ds
\end{align*}
by Lemma~\ref{lem:operator_condition_triangle}, as $S_{0}^{[n]}$ satisfies Condition~($\triangle$). We conclude using the fact that $(S^{[n]}_t)_{t\ge 0} = (S^{Q_n}_t)_{t\ge 0}$ satisfies the martingale problem associated to $\lcal^{Q_n}_\mu$ with initial condition $S_0^{[n]}$.
\end{proof}

\subsubsection{\texorpdfstring{Convergence of $(M^{[n]}_t)_{t\ge 0}$ towards $(M^E_t)_{t\ge 0}$ as $n\to\infty$}{Convergence of the process on compacts to the infinite-volume limit}}

First we show that $(M^{[n]})_{n \geq 0}$ converges in distribution towards the 
set-valued $\infty$-parent SLFV up to a change of state space. 
\begin{prop}\label{prop:part_1_limit_slv_compact}
As in Section \ref{sec:triangle}, let $(S^E_{t})_{t \geq 0}$ be the set-valued SLFV with initial condition $E$ and intensity~$\mu$, and for each $t \geq 0$, let
\[M^E_{t}(dz,A) := \left(
\mathds{1}_{S_{t}}(z)\ind_{\{1\in A\}} + \left(
1 - \mathds{1}_{S_{t}}(z)
\right)\ind_{\{0\in A\}}\right)dz, \quad z \in \rtwo, A \subseteq \{0,1\}.\]
Then $(M^{[n]}_t)_{t\ge 0}$ converges to $(M^E_t)_{t\ge 0}$ in distribution as $n\to\infty$.
\end{prop}

For the proof we follow the (standard) strategy used in the proof of Theorem~2.10 in~\cite{louvet2023stochastic}:
\begin{enumerate}
    \item First we show that for each fixed $t \geq 0$, the measures $M_{t}^{[n]}$ converge almost surely as $n\to\infty$ to $M^E_{t}$ (which corresponds to~\cite[Lemma~3.9]{louvet2023stochastic}). 
    \item Then we show that the sequence $(M^{[n]})_{n \geq 1}$ is tight in $D_{\mcal_{\lambda}}[0,+\infty)$ (which corresponds to~\cite[Lemma~3.10]{louvet2023stochastic}). 
    \item We conclude using Prokhorov's theorem \cite{prokhorov1956convergence}. 
\end{enumerate}

\begin{lem}\label{lem:prelim_to_equivalent_lemma_3_9} For all $t \geq 0$ and $z \in \rtwo$, 
\begin{equation*}\mathds{1}_{S_{t}^{[n]}}(z) \xrightarrow[n \to + \infty]{} \mathds{1}_{S^E_{t}}(z) \quad \text{ almost surely.} 
\end{equation*}
\end{lem}
\begin{proof}
Let $t \geq 0$ and $z \in \rtwo$. Recall the definition of the set-valued $\infty$-parent SLFV started from $E$:
\[z \in S_{t} \Longleftrightarrow B_{t}^{(t,z)} \cap E \neq \emptyset\]
where $B_t^{(t,z)}$ is the $\infty$-parent ancestral skeleton from Definition \ref{defn:infty_parent_ancestral_skeleton}. Since $Q_n\uparrow \mathbb{R}^2$, there exists a $\nmath$-valued random variable $N^{(t,z)}$ such that for all $n \geq N^{(t,z)}$, 
\begin{equation*}
    B_{t}^{(t,z)} \subseteq Q_{n}. 
\end{equation*}
Therefore, for all $n \geq N^{(t,z)}$, 
\begin{equation*}
    \mathds{1}_{S_{t}^{[n]}}(z) = \mathds{1}_{S^E_{t}}(z). 
\end{equation*}
As $B_{t}^{(t,z)}$ is the union of an almost surely finite number of reproduction events of bounded radius, $N^{(t,z)}$ is almost surely finite, which allows us to conclude. 
\end{proof}

\begin{lem}\label{lem:equivalent_lemma_3_9} For each $t \geq 0$, $(M_{t}^{[n]})_{n \geq 1}$ converges vaguely to $M^E_{t}$ as $n \to + \infty$.
\end{lem}

The proof is a direct adaptation of the one of Lemma~3.9 from~\cite{louvet2023stochastic}, but we include it for completeness. 
\begin{proof}
Let $t \geq 0$. Let $\tilde{f} \in C_{c}(\rtwo \times \{0,1\})$. Then there exist $f_{0},f_{1} \in C_{c}(\rtwo)$ such that for all pairs $(z,\kappa) \in \rtwo \times \{0,1\}$, 
\begin{equation*}
    \tilde{f}(z,\kappa) = f_{0}(z)\un{0}(\kappa) + f_{1}(z)\un{1}(\kappa). 
\end{equation*}
Then, for all $n \geq 1$, by the dominated convergence theorem and Lemma~\ref{lem:prelim_to_equivalent_lemma_3_9},
\begin{align*}
\int_{\rtwo \times \{0,1\}} \tilde{f}(z,\kappa) M_{t}^{[n]}(dz,d\kappa) 
&= \int_{\rtwo}f_{1}(z)\mathds{1}_{S_{t}^{[n]}}(z)dz + \int_{\rtwo}f_{0}(z)\left(
1 - \mathds{1}_{S_{t}^{[n]}}(z)
\right)dz \\ 
&\xrightarrow[n \to + \infty]{} \int_{\rtwo}f_{1}(z)\mathds{1}_{S^E_{t}}(z)dz
+ \int_{\rtwo} f_{0}(z)\left(
1 - \mathds{1}_{S^E_{t}}(z)
\right)dz \\
&= \int_{\rtwo \times \{0,1\}} \tilde{f}(z,\kappa)M^E_{t}(dz,d\kappa)
\end{align*}
which allows us to conclude. 
\end{proof}

\begin{lem}\label{lem:equivalent_lemma_3_10}
The sequence $(M^{[n]})_{n \geq 1}$ is tight in $D_{\mcal_{\lambda}}[0,+\infty)$.
\end{lem}
\begin{proof}
We follow the outline of the proof of Lemma~3.10 in~\cite{louvet2023stochastic}. 
First, by the same argument as in~\cite{louvet2023stochastic}, the result is equivalent to the relative compactness of the sequence $(\Psi_{F,f}(M^{[n]}))_{n \geq 1}$ for all $F \in C^{1}(\rmath)$ and $f \in C_{c}(\rtwo)$, where we recall that
\[\Psi_{F,f}(M) := F\left(
\langle M,f \rangle 
    \right).\]
Therefore, let $F \in C^{1}(\rmath)$ and $f \in C_{c}(\rtwo)$. For each $t \geq 0$, since $f$ has compact support, the sequence $(\Psi_{F,f}(M_{t}^{[n]}))_{n \geq 1}$ is bounded. Moreover, by the same argument as in the proofs of Lemmas~6.3 and~6.4 in~\cite{louvet2023stochastic}, there exist $C_{F,f}, \widetilde{C}_{F,f} > 0$ such that for all $M \in \mcal_{\lambda}$, 
\begin{equation*}
\left|
\lcal_{\mu}^{[n]}\Psi_{F,f}(M)
\right| \leq C_{F,f}
\end{equation*}
and for all $z \in \rtwo$ and $0 < r \leq \text{R}_{0}$, 
\begin{equation*}
\left|
F\left(\langle
S \cup \left(\bcal(z,r) \cap Q_{n}
\right),f
\rangle\right)-F\left(\langle
S,f
\rangle\right)
\right|^{2} \leq \widetilde{C}_{F,f}. 
\end{equation*}
Fix $T > 0$ and $\theta>0$, and let $(T_{k})_{k \geq 2}$ be a sequence of stopping times bounded from above by~$T$. Then by the bound above, for all $k \geq 2$ and $n \geq 1$, 
\begin{equation}\label{eqn:aldous_rebolledo_1}
\left|
\int_{T_{k}}^{T_{k}+\theta}
\lcal_{\mu}^{[n]} \Psi_{F,f}\left(
M_{s}^{[n]}
\right)ds
\right| \leq \theta C_{F,f}. 
\end{equation}
Moreover, defining
\begin{equation*}
\text{Supp}_{r}(f) := \left\{
z \in \rtwo : \exists z' \in \rtwo \text{ with } f(z') \neq 0 \text{ and }
\|z-z'\| \leq r \right\}, 
\end{equation*}
the set of all points within distance $r$ of the support of $f$, we note that when $z\not\in \text{Supp}_{r}(f)$ we have
\begin{equation*}
\left|
F\left(\langle
S \cup \left(\bcal(z,r) \cap Q_{n}
\right),f
\rangle\right)-F\left(\langle
S,f
\rangle\right)
\right|^{2} =0. 
\end{equation*}
Thus we obtain that 
\begin{align*}
&\int_{T_{k}}^{T_{k}+\theta}\int_{0}^{\text{R}_{0}}
\int_{Q_{n}} \left(
F\left(
\langle 
S_{t}^{[n]}\cup\left(\bcal(z,r)\cap Q_n\right), f
\rangle\right)-F\left(
\langle 
S_{t}^{[n]}, f
\rangle\right)
\right)^{2} dz\,\mu(dr)\,ds \\
&\leq \int_{T_{k}}^{T_{k}+\theta}\int_{0}^{\text{R}_{0}}\int_{\text{Supp}_{r}(f) \cap Q_{n}}\widetilde{C}_{F,f}\, dz\, \mu(dr)\, ds, 
\end{align*} 
which is constant for~$n$ large enough, and of the form $\theta \Bar{C}_{F,f}$ for some constant $\Bar{C}_{F,f}>0$. Therefore, we can apply the Aldous-Rebolledo criterion \cite{aldous1978stopping,rebolledo1980existence} to the sequence $(\Psi_{F,f}(M^{[n]}))_{n \geq 1}$ and conclude. 
\end{proof}

We can now show Proposition~\ref{prop:part_1_limit_slv_compact}. 

\begin{proof}[Proof of Proposition~\ref{prop:part_1_limit_slv_compact}]
By Lemmas~\ref{lem:equivalent_lemma_3_9} and~\ref{lem:equivalent_lemma_3_10}, we can use Prokhorov's theorem and conclude. 
\end{proof}

\subsubsection{Proof of Theorem~\ref{thm:equality_slfvs}}
To conclude the proof of Theorem~\ref{thm:equality_slfvs}, we now need to show that the set-valued $\infty$-parent SLFV~$M^{E}$ is a solution to the martingale problem~$(\lcal_{\mu}^{\infty},\delta_{\varphi})$ characterizing~$\widetilde{M}^{\varphi}$, the measure-valued $\infty$-parent SLFV started from~$\varphi = M_{0}^{E}$. 

\begin{prop}\label{prop:part_2_limit_slv_compact}
The process $(M_{t}^{E})_{t \geq 0}$ is a solution to the martingale problem $(\lcal_{\mu}^{\infty},\delta_{\varphi})$. 
\end{prop}

\begin{proof}
Our goal is to show that for all $F \in C^{1}(\rmath)$, $f \in C_{c}(\rtwo)$, $t,s \geq 0$, $k \geq 0$, $0 \leq t_{1} < t_{2} < ... < t_{k} \leq t$ and $h_{1}, ..., h_{k} \in C_{b}(\mcal_{\lambda})$,
\begin{equation}\label{eqn:statement_martingale}
\esp\left[
\left(
\Psi_{F,f}\left(
M_{t+s}^{E}
\right) - \Psi_{F,f}\left(
M_{t}^{E}
\right) - \int_{t}^{t+s}\lcal_{\mu}^{\infty} \Psi_{F,f}\left(
M_{u}^{E}
\right)du
\right) \times \left(
\prod_{i = 1}^{k} h_{i}\left(
M_{t_{i}}^{E}
\right)\right)\right].
\end{equation}
To do so, let $F \in C^{1}(\rmath)$, $f \in C_{c}(\rtwo)$, $t,s \geq 0$, $k \geq 0$, $0 \leq t_{1} < t_{2} < ... < t_{k} \leq t$ and $h_{1}, ..., h_{k} \in C_{b}(\mcal_{\lambda})$. Since $\Psi_{F,f} \in C_{b}(\mcal_{\lambda})$ and $h_{1},...,h_{k} \in C_{b}(\mcal_{\lambda})$, by Proposition~\ref{prop:part_1_limit_slv_compact}, we have
\begin{align*}
\Psi_{F,f}\left(
M_{t+s}^{E}
\right) &= \lim\limits_{n \to + \infty} \Psi_{F,f}\left(
M_{t}^{[n]}
\right) \\
\Psi_{F,f}\left(
M_{t}^{E}
\right) &= \lim\limits_{n \to + \infty} \Psi_{F,f}\left(
M_{t}^{[n]}
\right) \\
\text{and } \forall 1 \leq i \leq k, h_{i}\left(
M_{t_{i}}^{E}
\right) &= \lim\limits_{n \to + \infty} h_{i}\left(
M_{t_{i}}^{[n]}
\right). 
\end{align*}
While we do not have $\lcal_{\mu}^{\infty} \Psi_{F,f} \in C_{b}(\mcal_{\lambda})$, by construction and by Lemma~\ref{lem:equivalent_lemma_3_9}, the sequence~$(M^{[n]})_{n \geq 1}$ and its limits~$M^{E}$ satisfy the assumptions of~\cite[Lemma~5.4]{louvet2023stochastic} (note that in \cite{louvet2023stochastic}, the densities encode the \textit{empty areas} rather than the occupied areas as in this article). Therefore, 
\begin{equation*}
\forall t \leq u \leq t+s, \lcal_{\mu}^{\infty} \Psi_{F,f}\left(
M_{u}^{E}
\right) = \lim\limits_{n \to + \infty} \lcal_{\mu}^{\infty} \Psi_{F,f}\left(
M_{u}^{[n]}
\right). 
\end{equation*}
Since $\Psi_{F,f}(\cdot)$ and $\lcal_{\mu}^{\infty}\Psi_{F,f}(\cdot)$ are bounded, by the dominated convergence theorem, (\ref{eqn:statement_martingale})
is equivalent to 
\begin{align*}
\lim\limits_{n \to + \infty} \esp\left[
\left(
\Psi_{F,f}\left(
M_{t+s}^{[n]}
\right) - \Psi_{F,f}\left(
M_{t}^{[n]}
\right) - \int_{t}^{t+s} \lcal_{\mu}^{\infty} \Psi_{F,f}\left(
M_{u}^{[n]}
\right)du
\right) \times \left(
\prod_{i = 1}^{k} h_{i}\left(
M_{t_{i}}^{[n]}
\right)
\right)
\right] = 0.
\end{align*}

Let $n \geq 1$. As $M^{[n]}$ is a solution to the martingale problem $(\lcal_{\mu}^{[n]},\delta_{M^{[n]}_{0}})$, 
\begin{align*}
&\esp\left[
\left(
\Psi_{F,f}\left(
M_{t+s}^{[n]}
\right) - \Psi_{F,f}\left(
M_{t}^{[n]}
\right) - \int_{t}^{t+s} \lcal_{\mu}^{\infty} \Psi_{F,f}\left(
M_{u}^{[n]}
\right)du
\right) \times \left(
\prod_{i = 1}^{k} h_{i}\left(
M_{t_{i}}^{[n]}
\right)
\right)
\right] \\
&= \esp\left[ 
\left(
\int_{t}^{t+s}\left(
\lcal_{\mu}^{[n]}\Psi_{F,f}\left(
M_{u}^{[n]}\right) - \lcal_{\mu}^{\infty} \Psi_{F,f}\left(
M_{u}^{[n]}
\right)
\right)ds\right) \times \left(
\prod_{i = 1}^{k} h_{i}\left(M_{t_{i}}^{[n]}\right)
\right)\right].
\end{align*}
Moreover, let $n_{0} \in \nmath \backslash \{0\}$ be such that $\text{Supp}_{2\text{R}_{0}}(f) \subseteq Q_{n_{0}}$. Then, as $(Q_{n})_{n \geq 1}$ is increasing, for all $n \geq n_{0}$, the operators $\lcal_{\mu}^{[n]}$ and $\lcal_{\mu}^{\infty}$ are equal, which allows us to conclude. 
\end{proof}

We can now show Theorem~\ref{thm:equality_slfvs}.

\begin{proof}[Proof of Theorem~\ref{thm:equality_slfvs}]
By Proposition~\ref{prop:part_2_limit_slv_compact}, $(M_{t}^{E})_{t \geq 0}$ is solution to the martingale problem~$(\lcal_{\mu}^{\infty},\delta_{\varphi})$. Moreover, by Theorem~\ref{thm:measure_valued_slfv}, this martingale problem is well-posed and characterizes the measure-valued SLFV~$(\widetilde{M}_{t}^{\varphi})_{t \geq 0}$. Therefore, $(M^E_t)_{t\ge 0}$ and $(\widetilde M^{\varphi}_t)_{t\ge 0}$ are equal in distribution.
\end{proof}

\section*{Acknowledgements}
Both authors would like to thank the Royal Society for their generous funding of grant URF\textbackslash R\textbackslash 211038.
AL acknowledges partial support from the chair program "Mathematical Modelling and Biodiversity" of Veolia Environment-Ecole Polytechnique-National Museum of Natural History-Foundation X. 

\bibliographystyle{plain}

\begin{thebibliography}{10}

\bibitem{aldous1978stopping}
D.~Aldous.
\newblock Stopping times and tightness.
\newblock {\em Annals of {P}robability}, pages 335--340, 1978.

\bibitem{auffinger201750}
Antonio Auffinger, Michael Damron, and Jack Hanson.
\newblock {\em 50 years of first-passage percolation}, volume~68.
\newblock American Mathematical Soc., 2017.

\bibitem{barton2010new}
N.~Barton, A.~Etheridge, and A.~V{\'e}ber.
\newblock A new model for evolution in a spatial continuum.
\newblock {\em Electronic Journal of Probability}, 15(none):162 -- 216, 2010.

\bibitem{benaim2008exponential}
Michel Bena{\"\i}m and Rapha{\"e}l Rossignol.
\newblock Exponential concentration for first passage percolation through
  modified {P}oincar{\'e} inequalities.
\newblock In {\em Annales de l'IHP Probabilit{\'e}s et statistiques},
  volume~44, pages 544--573, 2008.

\bibitem{berestycki2013large}
Nathana{\"e}l Berestycki, Alison~M Etheridge, and Amandine V{\'e}ber.
\newblock Large scale behaviour of the spatial $\lambda$-{F}leming-{V}iot
  process.
\newblock {\em Annales de l'IHP Probabilit{\'e}s et Statistiques},
  49(2):374--401, 2013.

\bibitem{birzu2018fluctuations}
G.~Birzu, O.~Hallatschek, and K.S. Korolev.
\newblock Fluctuations uncover a distinct class of traveling waves.
\newblock {\em Proceedings of the {N}ational {A}cademy of {S}ciences},
  115(16):E3645--E3654, 2018.

\bibitem{biswas:SLFV_fluctuating_selection}
Niloy Biswas, Alison Etheridge, and Aleksander Klimek.
\newblock {The spatial Lambda-Fleming-Viot process with fluctuating selection}.
\newblock {\em Electronic Journal of Probability}, 26:1 -- 51, 2021.

\bibitem{chatterjee2013universal}
Sourav Chatterjee.
\newblock The universal relation between scaling exponents in first-passage
  percolation.
\newblock {\em Annals of Mathematics}, pages 663--697, 2013.

\bibitem{cox1981some}
J~Theodore Cox and Richard Durrett.
\newblock Some limit theorems for percolation processes with necessary and
  sufficient conditions.
\newblock {\em The Annals of Probability}, pages 583--603, 1981.

\bibitem{deijfen2003asymptotic}
Maria Deijfen.
\newblock Asymptotic shape in a continuum growth model.
\newblock {\em Advances in Applied Probability}, 35(2):303--318, 2003.

\bibitem{deijfen2004coexistence}
Maria Deijfen and Olle H{\"a}ggstr{\"o}m.
\newblock Coexistence in a two-type continuum growth model.
\newblock {\em Advances in applied probability}, 36(4):973--980, 2004.

\bibitem{durrett2016genealogies}
R.~Durrett and W.-T.~L. Fan.
\newblock {Genealogies in expanding populations}.
\newblock {\em The {A}nnals of {A}pplied {P}robability}, 26(6):3456 -- 3490,
  2016.

\bibitem{etheridge2008drift}
Alison Etheridge.
\newblock Drift, draft and structure: some mathematical models of evolution.
\newblock {\em Banach center publications}, 1(80):121--144, 2008.

\bibitem{etheridge:BBM_mean_curvature}
Alison Etheridge, Nic Freeman, and Sarah Penington.
\newblock {Branching Brownian motion, mean curvature flow and the motion of
  hybrid zones}.
\newblock {\em Electronic Journal of Probability}, 22:1 -- 40, 2017.

\bibitem{excoffier2009genetic}
L.~Excoffier, M.~Foll, and R.J. Petit.
\newblock Genetic consequences of range expansions.
\newblock {\em Annual Review of Ecology, Evolution, and Systematics},
  40:481--501, 2009.

\bibitem{fan2021stochastic}
W.-T.~L. Fan.
\newblock {Stochastic PDEs on graphs as scaling limits of discrete interacting
  systems}.
\newblock {\em Bernoulli}, 27(3):1899 -- 1941, 2021.

\bibitem{felsenstein1975pain}
Joseph Felsenstein.
\newblock A pain in the torus: some difficulties with models of isolation by
  distance.
\newblock {\em The {A}merican {N}aturalist}, 109(967):359--368, 1975.

\bibitem{forien2022stochastic}
Rapha{\"e}l Forien and Bastian Wiederhold.
\newblock Stochastic partial differential equations describing isolation by
  distance under various forms of power-law dispersal.
\newblock {\em arXiv preprint arXiv:2211.16286}, 2022.

\bibitem{gouere2008continuous}
Jean-Baptiste Gou{\'e}r{\'e} and R{\'e}gine Marchand.
\newblock Continuous first-passage percolation and continuous greedy paths
  model: linear growth.
\newblock {\em The Annals of Applied Probability}, pages 2300--2319, 2008.

\bibitem{grebenkov2017anisotropy}
Denis~S Grebenkov and Dmitry Beliaev.
\newblock How anisotropy beats fractality in two-dimensional on-lattice
  diffusion-limited-aggregation growth.
\newblock {\em Physical Review E}, 96(4):042159, 2017.

\bibitem{grimmett2012percolation}
Geoffrey Grimmett and Harry Kesten.
\newblock Percolation since {S}aint-{F}lour.
\newblock In {\em Percolation Theory at Saint-Flour}. Springer, 2012.

\bibitem{hallatschek2008gene}
O.~Hallatschek and D.R. Nelson.
\newblock Gene surfing in expanding populations.
\newblock {\em Theoretical {P}opulation {B}iology}, 73(1):158--170, 2008.

\bibitem{hallatschek2010life}
O.~Hallatschek and D.R. Nelson.
\newblock Life at the front of an expanding population.
\newblock {\em Evolution: {I}nternational {J}ournal of {O}rganic {E}volution},
  64(1):193--206, 2010.

\bibitem{hallatschek2007genetic}
Oskar Hallatschek, Pascal Hersen, Sharad Ramanathan, and David~R Nelson.
\newblock Genetic drift at expanding frontiers promotes gene segregation.
\newblock {\em Proceedings of the National Academy of Sciences},
  104(50):19926--19930, 2007.

\bibitem{johansson2000transversal}
Kurt Johansson.
\newblock Transversal fluctuations for increasing subsequences on the plane.
\newblock {\em Probability theory and related fields}, 116(4):445--456, 2000.

\bibitem{kardar1986dynamic}
Mehran Kardar, Giorgio Parisi, and Yi-Cheng Zhang.
\newblock Dynamic scaling of growing interfaces.
\newblock {\em Physical Review Letters}, 56(9):889, 1986.

\bibitem{kingman1968ergodic}
John~FC Kingman.
\newblock The ergodic theory of subadditive stochastic processes.
\newblock {\em Journal of the Royal Statistical Society: Series B
  (Methodological)}, 30(3):499--510, 1968.

\bibitem{klimek:SLFV_random_environment}
Aleksander Klimek and Tommaso~Cornelis Rosati.
\newblock {The spatial $\Lambda$-Fleming-Viot process in a random environment}.
\newblock {\em The Annals of Applied Probability}, 33(3):2426 -- 2492, 2023.

\bibitem{liggett1985improved}
Thomas~M Liggett.
\newblock An improved subadditive ergodic theorem.
\newblock {\em The Annals of Probability}, 13(4):1279--1285, 1985.

\bibitem{louvet2023stochastic}
A.~Louvet.
\newblock Stochastic measure-valued models for populations expanding in a
  continuum.
\newblock {\em ESAIM: {P}robability and {S}tatistics}, 27:221--277, 2023.

\bibitem{louvet2023measurevalued}
A.~Louvet and A.~V{\'e}ber.
\newblock Measure-valued growth processes in continuous space and growth
  properties starting from an infinite interface.
\newblock {\em arXiv preprint arXiv:2205.03937}, 2023.

\bibitem{prokhorov1956convergence}
Y.V. Prokhorov.
\newblock Convergence of random processes and limit theorems in probability
  theory.
\newblock {\em Theory of {P}robability and {I}ts {A}pplications},
  1(2):157--214, 1956.

\bibitem{quastel2011introduction}
Jeremy Quastel.
\newblock Introduction to kpz.
\newblock {\em Current developments in mathematics}, 2011(1), 2011.

\bibitem{rebolledo1980existence}
R.~Rebolledo.
\newblock Sur l’existence de solutions {\`a} certains probl{\`e}mes de
  semi-martingales.
\newblock {\em {C.R. Acad. Sci. Paris}}, 780:843--846, 1980.

\end{thebibliography}

\end{document}